\documentclass[preprint,1p]{elsarticle}
\usepackage[utf8x]{inputenc}

\usepackage{amsthm,amsmath,amssymb,amsfonts} 
\usepackage{graphicx,subfigure} 
\usepackage{comment}
\usepackage{enumerate}

\newtheorem{Theorem}{Theorem}[section]{\bfseries}{\itshape}
\newtheorem{theorem}[Theorem]{Theorem}{\bfseries}{\itshape}
{\bfseries}{\itshape}
\newtheorem{lemma}[Theorem]{Lemma}{\bfseries}{\itshape}
{\bfseries}{\itshape}
\newtheorem{corollary}[Theorem]{Corollary}{\bfseries}{\itshape}
{\bfseries}{\itshape}
{\bfseries}{\itshape}
\newtheorem{remark}[Theorem]{Remark}{\bfseries}{\itshape}
{\bfseries}{\itshape}
{\bfseries}{\itshape}
{\bfseries}{\itshape}
\newtheorem{definition}[Theorem]{Definition}{\bfseries}{\itshape}
{\bfseries}{\itshape}

\newenvironment{numberedproof}[1]{{\bf Proof #1:}}{{}\hfill{\hbox{$\Box$}}\par\bigskip}

\def\BbbR{{\mathbb R}}
\def\BbbN{{\mathbb N}}

\def\T{{\mathcal T}}
\def\S{{\mathcal S}}

\def\P{{\mathcal P}}

\newcommand{\eremk}{\hbox{}\hfill\rule{0.8ex}{0.8ex}}

\numberwithin{equation}{section}
\newif\iftechreport
\techreporttrue
\iftechreport
\else
\fi
\def\langversion{\cite{melenk-wurzer13}{}}
\begin{document}
\begin{frontmatter}

\iftechreport 
\title{On the stability of the boundary trace of the polynomial $L^2$-projection on 
triangles and tetrahedra (extended version)}
\fi

 \author{J.M. Melenk\corref{cor1}}
\ead{melenk@tuwien.ac.at}
 \author{T. Wurzer}
\ead{tobias@wurzer-it.at}
\address{Vienna University of Technology, Wiedner Hauptstra\ss e 8-10, A-1040 Vienna}

\begin{abstract}
\noindent For the reference triangle or tetrahedron $\T$, we study the stability 
properties of the $L^2(\T)$-projection $\Pi_N$ onto the space of polynomials 
of degree $N$. We show 
$\|\Pi_N u\|_{L^2(\partial \T)}^2 \leq C \|u\|_{L^2(\T)} \|u\|_{H^1(\T)}$. 
This implies optimal convergence rates for the approximation error 
$\|u - \Pi_N u\|_{L^2(\partial \T)}$ for all 
$u \in H^k(\T)$, $k > 1/2$. 
\end{abstract}

\end{frontmatter}

\section{Introduction and main results}
\label{sec:intro}
The study of polynomials and their properties as the polynomial 
degree tends to infinity has a very long history 
in numerical mathematics. Concerning approximation and stability
properties of various high order approximation operators, the 
univariate case is reasonably well understood (in the way of 
examples, we mention the monographs \cite{szego39} for 
orthogonal polynomials and \cite{devore93} for issues concerning approximation). 
One-dimensional results can often be generalized to the case of the 
$d$-dimensional hyper cube by tensor product arguments. Indeed, 
a significant number of results is available for the tensorial case, and 
we point the reader to the area of \emph{spectral methods} 
\cite{bernardi-maday97,canuto-hussaini-quarteroni-zhang07}
and to \cite{guo09} and references therein. 
The situation is less developed for simplices, possibly due to 
a presumed lack of product structure. 
It is the purpose of this note to contribute to this field 
by studying 
the stability properties of the polynomial $L^2$-projection on 
triangles and tetrahedra, paying special attention to its 
trace on the  boundary.  Our 
main result is Theorem~\ref{thm:trace-stability} below on the 
stability of the boundary trace of the $L^2$-projection. It 
generalizes known results for tensor product domains to the case
of triangles/tetrahedra: Theorem~\ref{thm:trace-stability}
is the analog of \cite[Lemma~{4.2}]{georgoulis-hall-melenk10} 
(and correspondingly, Corollary~\ref{cor:trace-approximation} is the 
analog of 
\cite[Rem.~{4.3}]{georgoulis-hall-melenk10} and 
\cite [Lem.~{3.5}]{houston-schwab-suli02}). We mention that the
two-dimensional case of triangles was studied in \cite{wurzer10}, where 
Theorem~\ref{thm:trace-stability}, Corollary~\ref{cor:trace-approximation}, 
and Corollary~\ref{cor:H1-stability} are shown for triangles. 
\noindent Independently, closely related results have recently been obtained 
in \cite{chernov12} and \cite{egger-waluga13}. The novelty of the present work 
over \cite{chernov12} is that, in the language of Corollary~\ref{cor:trace-approximation}
below, we extend the approximation result of \cite{chernov12} from $s \ge 1$
to $s > 1/2$. The elegant proof of Corollary~\ref{cor:H1-stability} given in 
\cite[Lem.~{6.4}]{egger-waluga13} relies heavily on Corollary~\ref{cor:trace-approximation}; 
the analysis presented here provides the tools for an alternative 
proof of Corollary~\ref{cor:H1-stability}, which we outline as well. 

\noindent 
The stability result of Theorem~\ref{thm:trace-stability} has applications 
in the analysis of the $hp$-version of discontinuous Galerkin methods 
($hp$-DGFEM) as demonstrated in \cite{stamm-wihler10}. More generally, 
simplicial elements are, due to their greater geometric flexibility as 
compared to tensor product elements, commonly used in high order 
finite element codes so that an understanding of stability 
and approximation properties of polynomial operators defined on simplices 
could be useful in other applications of high order finite element methods ($hp$-FEM)
as well. We refer to \cite{schwab98,karniadakis-sherwin99,
demkowicz07,demkowicz-kurtz-pardo-paszynski-rachowicz-zdunek08,solin-segeth-dolezel03}
for various aspects of $hp$-FEM.

\bigskip
To fix the notation, we introduce the reference triangle $\T^{2}$,  
the reference tetrahedron $\T^{3}$ as well as the reference cube $\S^d$ by 
\begin{subequations}
        \label{eq:reference-elements}
\begin{align}
        \T^2 &:= \{(x,y)\in\BbbR^2: -1 < x < 1~,~-1< y < -x\}, \\
        \T^3 &:= \{(x,y,z)\in\BbbR^3: -1 < x,y,z < 1,~x+y+z < -1\},\\
        \S^d &:= (-1,1)^d,\quad d\in \{1,2,3\}. 
\end{align}
\end{subequations}
Throughout, we will denote by ${\mathcal P}_N$ the space of polynomials
of (total) degree $N$. We then have: 
\begin{theorem}
	\label{thm:trace-stability} Let $\T$ be the reference triangle or tetrahedron
	and denote by $\Pi_N:L^2(\T) \rightarrow {\mathcal P}_N$ the $L^2(\T)$-projection
	onto the space of polynomials of degree $N \in \BbbN_0$. Then there exists a constant $C > 0$
	independent of $N$ such that 
	\begin{align}
		\label{thm:trace-stability-1}
		\|\Pi_N u\|_{L^2(\partial \T)}^2 \leq C \|u\|_{L^2(\T)} \|u\|_{H^1(\T)} 
		\qquad \forall u \in H^1(\T). 
	\end{align}
	In particular, therefore, 
	\begin{align}
		\label{thm:trace-stability-2}
		\|\Pi_N u\|_{L^2(\partial \T)} \leq C \|u\|_{B^{1/2}_{2,1}(\T)} 
		\qquad \forall u \in B^{1/2}_{2,1}(\T),
	\end{align}
	where the Besov space $B^{1/2}_{2,1}(\T) = (L^2(\T),H^1(\T))_{1/2,1}$ is defined 
        as an interpolation space using the so-called real method 
        (see, e.g.,  \cite{triebel95,tartar07} for details).
\end{theorem}
As already mentioned above, the following corollary extends the admissible range 
for the parameter $s$ from $s \ge 1$ shown in \cite{chernov12} to $s > 1/2$: 
\begin{corollary}
	\label{cor:trace-approximation} 
	Let $\T$ be the reference triangle or tetrahedron. Then  for every $s > 1/2$
	there exists a constant $C_s > 0$ such that 
	\begin{align}
		\label{eq:cor:trace-approximation-1}
		\|u - \Pi_N u\|_{L^2(\partial \T)} \leq C_s (N+1)^{-(s-1/2)} \|u\|_{H^{s}(\T)}
		\qquad \forall u \in H^{s}(\T). 
	\end{align}
\end{corollary}
The following result, which generalizes the analogous result 
on hyper cubes of \cite[Thm.~{2.2}]{canuto-quarteroni82}, is derived in 
\cite[Lem.~{6.4}]{egger-waluga13} with an integration by parts argument, 
a polynomial inverse estimate, 
and Corollary~\ref{cor:trace-approximation} with $s = 1$. 
\begin{corollary}[{\protect{\cite[Lem.~{6.4}]{egger-waluga13}}}]
	\label{cor:H1-stability}
	Let $\T$ be the reference triangle or tetrahedron. Then there exists a constant
	$C > 0$ such that for all $N\in\BbbN_0$
	\begin{align}
		\|\Pi_N u\|_{H^1(\T)} \leq C (N+1)^{1/2} \|u\|_{H^1(\T)} \qquad \forall u \in H^1(\T). 
	\end{align}
\end{corollary}

\noindent 
We will only prove Theorem~\ref{thm:trace-stability} and Corollary~\ref{cor:trace-approximation} 
for the 3D case at the end of Section~\ref{sec:trace-stability}. 
We refer to the Bachelor Thesis \cite{wurzer10} for the 2D version. 
We will present the key steps for an alternative proof of Corollary~\ref{cor:H1-stability} 
at the end of Section~\ref{sec:H1-stability}. 

\begin{remark}
{\rm 
The regularity requirement $u \in B^{1/2}_{2,1}(\T)$ in the
 stability estimate (\ref{thm:trace-stability-2}) 
is essentially the minimal one, if uniformity in $N$ is sought. 
To see this, we first note that taking the limit $N \rightarrow \infty$ 
reproduces the known result (see, e.g., \cite[Chap.~{32}]{tartar07}) 
$\|\gamma_0 u\|_{L^2(\partial\T)} \leq C \|u\|_{B^{1/2}_{2,1}(\T)}$, 
where $\gamma_0$ is the trace operator. Next, we recall that for any
$\varepsilon > 0$ we have 
$H^{1/2+\varepsilon}(\T) \subset B^{1/2}_{2,1}(\T) \subset H^{1/2-\varepsilon}(\T)$ 
but 
the trace operator $\gamma_0$ {\em cannot} be extended
to a bounded linear operator $H^{1/2-\varepsilon}(\T) \rightarrow L^2(\partial\T)$. 
}\eremk
\end{remark}
Numerical results indicating for the 1D and the 2D situation the sharpness 
of Theorem~\ref{thm:trace-stability} and Corollary~\ref{cor:H1-stability} are given 
in Section~\ref{sec:numerics} below. 
\subsection{Outline of the proof of Theorem~\ref{thm:trace-stability}}
\label{sec:outline}
In the tensor-product case of squares and hexahedra, the arguments
leading to Theorem~\ref{thm:trace-stability} and Corollary~\ref{cor:H1-stability}
can be reduced to a one-dimensional setting and have been worked out
in \cite{georgoulis-hall-melenk10}---we recapitulate the key steps 
later in this subsection. 
Most of this reduction to one-dimensional settings is also possible for 
simplices, where Jacobi polynomials come into play instead
of the simpler Legendre or Gegenbauer/ultraspherical polynomials. Let 
us highlight some reasons why a reduction to 
one-dimensional situations is possible and what the key ideas of the 
proof of Theorem~\ref{thm:trace-stability} are:

\newcounter{Points}
\renewcommand{\thePoints}{(\Roman{Points})}
\refstepcounter{Points}\label{Point:1}
{\bf \thePoints} 
The first basic tool for a dimension reduction is 
a classical one, the so-called Duffy transformation 
(see (\ref{def:3d-duffy-trafo}) below), which maps the simplex into 
a hyper cube. As noted already by 
\cite{koornwinder75,dubiner91,karniadakis-sherwin99} orthogonal
polynomials on the simplex 
(see Lemma~\ref{lemma:orthogonal-polynomials-on-tet}) 
can be defined in the transformed variables
through products of univariate Jacobi polynomials, 
which expresses 
the desired reduction to one-dimensional settings. These Jacobi polynomials
arise since the transformation from the simplex to the cube changes the 
classical (unweighted) Lebesgue measure to a weighted measure. 
The situation is technically more complicated than the 
tensor-product setting: first, Jacobi polynomials arise (due to the weights) 
in contrast to the more common Gegenbauer/ultraspherical polynomials 
for the tensor-product setting. 
Second, the weight in the 
Jacobi polynomials is not fixed so that the dependence on both the polynomial 
degree and the weight needs to be tracked explicitly (cf.\ the definition 
of the orthogonal polynomial on the simplex in 
Lemma~\ref{lemma:orthogonal-polynomials-on-tet}). Nevertheless, the 
Jacobi polynomials are classical orthogonal polynomials, and one can draw 
on a plethora of known properties for the purpose of both analysis and 
design of algorithms. We mention in passing that these observations 
have been made and exploited previously in different connections, 
for example, in the works 
\cite{beuchler-schoeberl06,li-shen10,chernov12,karniadakis-sherwin99}. 

\refstepcounter{Points}\label{Point:2}
{\bf \thePoints} 
A second ingredient to the reduction to one-dimensional problems 
arises from the fact that we aim at trace estimates in 
Theorem~\ref{thm:trace-stability}. Such estimates hark back to techniques associated with 
the names of Gagliardo and Nirenberg and are closely connected to 
the 1D Sobolev embedding $H^1 \subset L^\infty$. For example,  
for the half space $\BbbR^d_+:=\{(x',x_d)\,|\, x'\in \BbbR^{d-1},x_d > 0\}$, 
the trace estimate can be cast in multiplicative form as 
$\|u(\cdot,0)\|_{L^2(\BbbR^{d-1})}^2 \lesssim 
\|u\|_{L^2(\BbbR^d_+)} 
\left[ \|u\|_{L^2(\BbbR^d_+)} + \|\partial_a u\|_{L^2(\BbbR^d_+)}\right]$
where one has some freedom to choose the vector (field) $a$ as long as it is 
{\em not} tangential to the plane $\BbbR^{d-1} \times \{0\}$ 
(the classical proofs usually take $a$ to be the normal to 
$\BbbR^{d-1} \times \{0\}$). This freedom
to suitably choose the vector field $a$ is exploited in the proof of 
Corollary~\ref{cor:trace-approximation} given in \cite{chernov12}, 
where $a$ is a vector field pointing from one face
of the reference tetrahedron to the vertex opposite. 
Our analysis here and that of \cite{wurzer10} effects 
a similar thing: It performs the analysis on the reference cube $\S^3$
for the Duffy-transformed function $\widetilde u$ and singles out 
$\partial_{\eta_3} \widetilde u$,  which is closely
related to the directional derivative $\partial_a u$ on the simplex
selected by \cite{chernov12}. 

\refstepcounter{Points}\label{Point:3}
{\bf \thePoints} 
Given that we study the $L^2$-projection, the observation \ref{Point:2}
shows that we need a connection between the expansion of $u$ 
in orthogonal polynomials on the triangle/tetrahedron and 
the expansion of $\partial_a u$. This is at the heart of the analysis
of \cite{chernov12}. Likewise it is the key step in \cite{wurzer10} and
the present article. Working in the transformed variables
on the cube $\S^d$ and denoting the transformed function by $\widetilde u$, 
we have to relate the expansion of $\widetilde u$ 
to that of $\partial_{\eta_3} \widetilde u$. Ultimately, the issue is 
to understand the relation between the sequences 
$(\hat u_q)_{q=0}^\infty$ and 
$(\hat b_q)_{q=0}^\infty$, where the terms $\hat u_q$ 
and $\hat b_q$ 
are the coefficients of the function 
$u = \sum_q \hat u_q P^{(\alpha,0)}_q$ 
and its derivative $u^\prime = \sum_q \hat b_q P^{(\alpha,0)}_q$ 
in terms of the expansion in series of Jacobi polynomials 
$P^{(\alpha,0)}_q$. A technical 
complication over the tensor product case is that a family of expansions 
(parametrized by $\alpha$) has to be considered and 
that the dependence on $\alpha$ has to be tracked. It is easy to express
the coefficients $\hat u_q$ in terms of the 
coefficients $\hat b_q$---see Lemma~\ref{lemma:connection-U-and-Uprime}  for the general case and 
(\ref{eq:ui-in-terms-of-bi}) for the special case $\alpha = 0$. It is 
harder to control the coefficients $\hat b_q$ in terms of 
the coefficients $\hat u_q$. In the present work, this is achieved in 
Lemma~\ref{lemma:estimate-of-bq-in-terms-of-product-of-sums} and discussed
in more detail in the following point~\ref{Point:4}.

\refstepcounter{Points}\label{Point:4}
{\bf \thePoints} 
The refinement of the present analysis and \cite{wurzer10} over 
\cite{chernov12} is the multiplicative structure of the estimate. This
results from a refined
connection between the expansion of a function and its derivative. 
Let us review 
the origin of the multiplicate estimate (\ref{thm:trace-stability-1})
for tensor product domains as  given in \cite[Lem.~{4.1}]{georgoulis-hall-melenk10}), 
since similar ideas underlie the arguments here. 
For the case $N \ge 2$ and $\T^1 := (-1,1)$, 
we aim to show 
$|(\Pi_N u)(1)|^2 \lesssim \|u\|_{L^2(\T^1)} \|u\|_{H^1(\T^1)}$. 
By the 1D Sobolev embedding theorem already mentioned above we have 
$\|u\|^2_{L^\infty(\T^1)} \lesssim \|u\|_{L^2(\T^1)} \|u\|_{H^1(\T^1)}$. It therefore 
suffices to establish the inequality for $u - \Pi_N u$.
Following \cite[Lem.~{3.5}]{houston-schwab-suli02}, we expand $u$ and its derivative $u^\prime$ in
Legendre series (we write, as is common, $L_q = P^{(0,0)}_q$):
\begin{align*}
u = \sum_{q=0}^\infty \hat u_q L_q, 
\qquad 
u^\prime = \sum_{q=0}^\infty \hat b_q L_q, 
\end{align*}
with coefficients $\hat u_q$ and $\hat b_q$ 
explicitly given by 
$\hat u_q = \frac{2q+1}{2} \int_{\T^1} u(x) L_q(x)\,\,dx$ and 
$\hat b_q = \frac{2q+1}{2} \int_{\T^1} u^\prime(x) L_q(x)\,\,dx$. 
Orthogonality properties of the Legendre polynomials  imply
(see Lemma~\ref{lemma:connection-U-and-Uprime} below with $\alpha =0$)
\begin{equation}
\label{eq:ui-in-terms-of-bi}
\hat u_q = \frac{\hat b_{q-1}}{2q-1} - \frac{\hat b_{q+1}}{2q+3},
\qquad q \ge 2.
\end{equation}
Since $L_q(1) = 1$ for all $q \in {\mathbb N}_0$, we get
\begin{equation}
\label{eq:1D-u-PiNu-at-1}
(u - \Pi_N u)(1) = 
\sum_{q=N+1}^\infty \hat u_q = \frac{\hat b_{N}}{2N+1} + \frac{\hat b_{N+1}}{2N+3}.
\end{equation}
The terms in the last expression are now estimated using a telescoping sum:
\begin{eqnarray*}
\left(\frac{\hat b_{N}}{2N+1}\right)^2   & = &
\sum_{r=N}^\infty
\left(\frac{\hat b_{r}}{2r+1}\right)^2  -
\left(\frac{\hat b_{r+2}}{2(r+2)+1}\right)^2 \\
&=&
\sum_{r=N}^\infty
\left(\frac{\hat b_{r}}{2r+1} - \frac{\hat b_{r+2}}{2(r+2)+1}\right)
\left(\frac{\hat b_{r}}{2r+1} + \frac{\hat b_{r+2}}{2(r+2)+1}\right) \\
&=&
\sum_{r=N}^\infty
 \hat u_{r+1}
\left(\frac{\hat b_{r}}{2r+1} + \frac{\hat b_{r+2}}{2(r+2)+1}\right) \\
&\lesssim& \left(\sum_{r=N}^\infty \frac{1}{2(r+1)+1} |\hat u_r|^2\right)^{1/2}
       \left(\sum_{r=N}^\infty (2r+1) \left(\frac{\hat b_r}{2r+1}\right)^2\right)^{1/2}\\
&\lesssim & \|u\|_{L^2(-1,1)} \|u^\prime\|_{L^2(-1,1)},
\end{eqnarray*}
where we have used
$\|u\|^2_{L^2(\T^1)} = \sum_{q=0}^\infty |\hat u_q|^2 \frac{2}{2q+1}$ and
$\|u^\prime\|^2_{L^2(\T^1)} = \sum_{q=0}^\infty |\hat b_q|^2 \frac{2}{2q+1}$.
We therefore conclude
$
|(u - \Pi_N u)(1)|^2 \lesssim \|u\|_{L^2} \|u^\prime \|_{L^2}. 
$ In particular, the above developments provide a simple proof
of Lemma~\ref{lemma:estimate-of-bq-in-terms-of-product-of-sums}
below for the special case $\alpha =0$.  
This lemma is at the heart of the multiplicative structure of 
of Theorem~\ref{thm:trace-stability}. 
\section{One-dimensional results}
\label{sec:1D}
As mentioned in Section~\ref{sec:outline}, many aspects of the proof of 
Theorem~\ref{thm:trace-stability} and Corollary~\ref{cor:H1-stability} 
can be reduced to one-dimensional settings. In this section, we collect
the univariate results for the proof of Theorem~\ref{thm:trace-stability}
in Section~\ref{sec:relation-u-uprime} and for the 
alternative proof of Corollary~\ref{cor:H1-stability} 
in Section~\ref{sec:truncated-jacobi-expansion}. 

\subsection{Preliminaries}
We denote by 
$P^{(\alpha,\beta)}_n$, $\alpha$, $\beta > -1$, $n\in\BbbN_0$, 
the Jacobi polynomials, \cite{szego39}. From 
\cite[(4.3.3)]{szego39}
we have the following orthogonality relation
for Jacobi polynomials and $p$, $q \in \BbbN_0$: 
\begin{align}
	\int_{-1}^1 (1-x)^{\alpha}(1-x)^{\beta}P_p^{(\alpha,\beta)}(x)P_q^{(\alpha,\beta)}(x)dx =
			\gamma_p^{(\alpha,\beta)} \delta_{p,q} ; 
	\label{eq:jacobipoly-orthogonality}
\end{align}
here, $\delta_{p,q}$ represents the Kronecker symbol and 
\begin{align}
	\gamma_p^{(\alpha,\beta)} := \frac{2^{\alpha+\beta+1}}{2p+\alpha+\beta+1}\frac{\Gamma(p+\alpha+1)\Gamma(p+\beta+1)}{p!\Gamma(p+\alpha+\beta+1)}. 
	\label{eq:abbrev-jacobipoly-orthogonality}
\end{align}
Furthermore, we abbreviate factors that will appear naturally in our
computations:
\begin{alignat}{3}
	\label{def:hi-gi}
		h_1(q,\alpha) &:= - \frac{2(q+1)}{(2q+\alpha+1)(2q+\alpha+2)}, 
		&\qquad g_1(q,\alpha) &:= \frac{2q+2\alpha}{(2q+\alpha-1)(2q+\alpha)}, \notag\\
		h_2(q,\alpha) &:=  \frac{2\alpha}{(2q+\alpha+2)(2q+\alpha)}, 
		&\qquad g_2(q,\alpha) &:= \frac{2\alpha}{(2q+\alpha-2)(2q+\alpha)}, \\
		h_3(q,\alpha) &:=  \frac{2(q+\alpha)}{(2q+\alpha+1)(2q+\alpha)},
		&\qquad g_3(q,\alpha) &:= -\frac{2q-2}{(2q+\alpha-1)(2q+\alpha-2)}. \notag
\end{alignat}
By a direct calculation, we can establish relations between the $h_i$ and $g_i$.

\begin{lemma}
	\label{lemma:relations-between-hi-gi}
	Let $h_i$, $g_i$, $i \in \{1,2,3\}$, be defined in (\ref{def:hi-gi}).
	Then there holds for any $q\geq 1$ and $\alpha \in \BbbN_0$
\begin{align}
\label{eq:relation-gi-hi}
		&	\frac{g_1(q+1,\alpha)}{\gamma_q^{(\alpha,0)}} 
            = \frac{h_3(q+1,\alpha)}{\gamma_{q+1}^{(\alpha,0)}}, 
\quad 
			\frac{g_2(q+1,\alpha)}{\gamma_q^{(\alpha,0)}} = \frac{h_2(q,\alpha)}{\gamma_{q}^{(\alpha,0)}}, 
\quad 
			\frac{g_3(q+1,\alpha)}{\gamma_q^{(\alpha,0)}} = \frac{h_1(q-1,\alpha)}{\gamma_{q-1}^{(\alpha,0)}}, \\
\label{eq:magic-cancellation}
&		(-1)^q \frac{1}{\gamma^{(\alpha,0)}_q} h_1(q,\alpha) 
		+ (-1)^{q+1} \frac{1}{\gamma^{(\alpha,0)}_{q+1}} h_2(q+1,\alpha) 
		+ (-1)^{q+2} \frac{1}{\gamma^{(\alpha,0)}_{q+2}} h_3(q+2,\alpha) = 0.
	\end{align}
	Furthermore, for any $q \ge 0$
	\begin{align}
\label{eq:h2-h1=h3}
	\end{align}
\end{lemma}
\begin{proof}
	This follows directly from the definitions and simple calculations. 
\iftechreport 
  Details can be found in Appendix~\ref{app:B}. 
\fi
\end{proof}

\noindent We will denote by $\widehat{P}_q^{(\alpha,0)}$ the antiderivative of $P_{q-1}^{(\alpha,0)}$, i.e., 
\begin{align}
	\label{def:antiderivative-of-jacobi}
	\widehat{P}_q^{(\alpha,0)}(x) := \int_{-1}^x P^{(\alpha,0)}_{q-1}(t)\,dt.
\end{align}
The following lemma states important relations between Jacobi polynomials, their derivatives, and 
their antiderivatives.
\begin{lemma} 
	\label{lemma:relations-of-jacobi-in-terms-of-jacobi}
	Let $\alpha\in\BbbN_0$ and $h_i$, $g_i$, $i\in\{1,2,3\}$, be given by 
	(\ref{def:hi-gi}) and $\gamma_p^{(\alpha,\beta)}$ by 
	(\ref{eq:abbrev-jacobipoly-orthogonality}). Then we have
	\begin{enumerate}[(i)]
	\item 
        \label{item:lemma:relations-of-jacobi-in-terms-of-jacobi-i}
               for $q \ge 1$
		\begin{align*}
			& \int_{-1}^x (1-t)^\alpha P^{(\alpha,0)}_q(t)\,dt = \\
			& \qquad \mbox{} - (1-x)^\alpha \Big( h_1(q,\alpha) P_{q+1}^{(\alpha,0)}(x) + h_2(q,\alpha) P_q^{(\alpha,0)}(x)
			+ h_3(q,\alpha) P_{q-1}^{(\alpha,0)}(x)\Big),
		\end{align*}
	\item 
        \label{item:lemma:relations-of-jacobi-in-terms-of-jacobi-ii}
               for $q \ge 2$
		\begin{align*}
			\widehat{P}_q^{(\alpha,0)}(x) = g_1(q,\alpha) P^{(\alpha,0)}_q(x) 
			+ g_2(q,\alpha) P^{(\alpha,0)}_{q-1}(x) + g_3(q,\alpha) P^{(\alpha,0)}_{q-2}(x),
		\end{align*} 
	\item 
        \label{item:lemma:relations-of-jacobi-in-terms-of-jacobi-iii}
                for $q \ge 1$
		\begin{align*}
			& \frac{1}{\gamma_q^{(\alpha,0)}} P_q^{(\alpha,0)}(x) = \\
			& \qquad \frac{h_1(q-1,\alpha)}{\gamma_{q-1}^{(\alpha,0)}} \big(P_{q-1}^{(\alpha,0)}\big)^\prime(x) 
			+ \frac{h_2(q,\alpha)}{\gamma_{q}^{(\alpha,0)}} \big(P_{q}^{(\alpha,0)}\big)^\prime(x)
			+ \frac{h_3(q+1,\alpha)}{\gamma_{q+1}^{(\alpha,0)}} \big(P_{q+1}^{(\alpha,0)}\big)^\prime(x).  
		\end{align*}
	\end{enumerate} 
\end{lemma} 

\begin{proof}
        The proof of 
        (\ref{item:lemma:relations-of-jacobi-in-terms-of-jacobi-i})
	relies on known relations satisfied by Jacobi polynomials (specifically, 
\cite[(A.3), (A.4), (A.8), (A.9)]{karniadakis-sherwin99}); 
\iftechreport see Appendix~\ref{app:B} for details.
\fi
        (\ref{item:lemma:relations-of-jacobi-in-terms-of-jacobi-ii}) is taken from \cite{beuchler-schoeberl06}; 
        (\ref{item:lemma:relations-of-jacobi-in-terms-of-jacobi-iii}) is obtained by differentiating 
        (\ref{item:lemma:relations-of-jacobi-in-terms-of-jacobi-ii}) and using 
        Lemma~\ref{lemma:relations-between-hi-gi}. 
\end{proof}
\subsection{The relation between expansions of a function and its derivative}
\label{sec:relation-u-uprime} 
The essential ingredient of the one-dimensional analysis in 
\cite[Thm.~{2.2}]{canuto-quarteroni82}, 
\cite[Lem.~{3.5}]{houston-schwab-suli02}, 
\cite[Lem.~{4.1}]{georgoulis-hall-melenk10}
is the ability to relate the expansion coefficients 
$(\hat u_q)_{q=0}^\infty$ of the Legendre expansion 
$u = \sum_{q} \hat u_q P^{(0,0)}_q$ to the 
expansion coefficients $(\hat b_q)_{q=0}^\infty$ of the 
Legendre expansion $u^\prime = \sum_{q} \hat b_q P^{(0,0)}_q$; 
we illustrated this point already in \ref{Point:4} of Section~\ref{sec:outline} 

This relation generalizes to the case of expansions in Jacobi polynomials. 
A first result in this direction is
(see also \cite[Lem.~{2.1}]{chernov12} and \cite[Lem.~{2.2}]{beuchler-schoeberl06}):
\begin{lemma} 
	\label{lemma:connection-U-and-Uprime}
        Let $\alpha \in \BbbN_0$. 
	Let $U\in C^1(-1,1)$ and let $(1-x)^\alpha U(x)$ as well as $(1-x)^{\alpha+1} U^\prime(x)$ 
be integrable. Furthermore, assume 
	$\displaystyle 	\lim_{x \rightarrow 1} (1-x)^{1+\alpha} U(x) = 0\quad and \quad
		\lim_{x\rightarrow -1}(1+x) U(x) = 0.$ Then the expansion coefficients 
	\begin{align*}
		u_q := \int_{-1}^1 (1-x)^\alpha U(x) P_q^{(\alpha,0)}(x) dx, \qquad 
		b_q := \int_{-1}^1 (1-x)^\alpha U^\prime(x) P_q^{(\alpha,0)}(x) dx. 
	\end{align*}
        satisfy the following connection formula for $q \ge 1$: 
	\begin{align}
\label{eq:connection-formula}
		u_q = h_1(q,\alpha) b_{q+1} + h_2(q,\alpha) b_q + h_3(q,\alpha) b_{q-1}.
	\end{align}
\iftechreport
Furthermore, we have the representations 
\begin{eqnarray*}
U &=& \sum_{q=0}^\infty \frac{1}{\gamma_{q}^{(\alpha,0)}} 
u_q P_q^{(\alpha,0)} \\
U^\prime &=& \sum_{q=0}^\infty \frac{1}{\gamma_{q}^{(\alpha,0)}} b_q P_q^{(\alpha,0)} 
\end{eqnarray*}
and the equalities 
\begin{eqnarray*}
\int_{-1}^1 |U(x)|^2 (1-x)^\alpha\,dx &=& \sum_{q=0}^\infty 
\frac{1}{\gamma_{q}^{(\alpha,0)}} |u_q|^2, \\
\int_{-1}^1 |U^\prime(x)|^2 (1-x)^\alpha\,dx &=& \sum_{q=0}^\infty 
\frac{1}{\gamma_{q}^{(\alpha,0)}} |b_q|^2.
\end{eqnarray*}
\fi
\end{lemma}

\begin{proof}
Follows from an integration by parts and the representation of antiderivatives
of Jacobi polynomials in terms of Jacobi polynomials given in 
Lemma~\ref{lemma:relations-of-jacobi-in-terms-of-jacobi} (i). 
\iftechreport 
  We refer to Appendix~\ref{app:B}  for details. 
\fi
\end{proof}
This connection formula between the coefficients  $u_q$ and $b_q$ allows us
to bound a weighted sum of the coefficients $u_q$ by a weighted sum of the 
coefficients $b_q$: 
\begin{lemma}
	\label{lemma:connection-U-and-Uprime-2}
	Let $U\in C^1(-1,1)$ and assume 
		$\int_{-1}^1 |U(x)|^2(1-x)^\alpha dx < \infty$ 
           as well as $\int_{-1}^1 |U^\prime(x)|^2(1-x)^\alpha dx < \infty$. 
	Let $u_q$ and $b_q$ be defined as in Lemma~\ref{lemma:connection-U-and-Uprime}.
	Then there exist constants $C_1$, $C_2>0$ independent of $\alpha$ and $U$ such that
	\begin{align*}
		\sum_{q=1}^\infty \frac{1}{\gamma_{q}^{(\alpha,0)}} (q+\alpha)^2 |u_q|^2 \leq C_1 \sum_{q=0}^\infty \frac{1}{\gamma_{q}^{(\alpha,0)}} |b_q|^2
		\leq C_2 \int_{-1}^1 \left|U^\prime(x)\right|^2(1-x)^\alpha dx.
	\end{align*}
\end{lemma}

\begin{proof}
The result follows from the relation between $u_q$ and $b_q$ given in 
Lemma~\ref{lemma:connection-U-and-Uprime} and from bounds for $h_1$, $h_2$, $h_3$.
\end{proof}

\noindent 
The following simple lemma is merely needed for the proof of 
Lemma~\ref{lemma:estimate-of-bq-in-terms-of-product-of-sums} below.

\begin{lemma}
	\label{lemma:estimate-trigamma-function}
	Let $\alpha\in\BbbN_0$ and $q \geq 1$. Then there exists a constant $C > 0$ independent 
	of $q$ and $\alpha$ such that
	\begin{align*}
		\alpha \sum_{j=q+\alpha}^N \frac{1}{j^2} \leq 2 \frac{\alpha}{q+\alpha}\qquad \forall N \ge q+\alpha. 
	\end{align*}
\end{lemma}

\begin{proof}
Follows by the standard argument of majorizing the sum by an integral. 
\end{proof}

\noindent While Lemma~\ref{lemma:connection-U-and-Uprime} shows that the coefficients $u_q$ can be expressed 
as a short linear combination of the coefficients $b_q$ (a maximum of $3$ coefficients suffices), the converse 
is not so easy. The following lemma may be regarded as a weak converse of Lemma~\ref{lemma:connection-U-and-Uprime}
since it allows us to bound the coefficients $b_q$ in terms of the coefficients $u_q$ and weighted
sums of the coefficients $b_q$. This is the main result of this section and the key ingredient of the proof of 
Theorem~\ref{thm:trace-stability} as it is responsible for the multiplicative structure of the 
bound in Theorem~\ref{thm:trace-stability}. 
In the proof of Lemma~\ref{lemma:estimate-of-bq-in-terms-of-product-of-sums}, 
the reader will recognize several arguments from 
\ref{Point:4} of Section~\ref{sec:outline}.

\begin{lemma}
	\label{lemma:estimate-of-bq-in-terms-of-product-of-sums}
Assume the hypotheses of Lemma~\ref{lemma:connection-U-and-Uprime}
and let $u_q$ and $b_q$ be defined as in Lemma~\ref{lemma:connection-U-and-Uprime}. Then there exists a 
	constant $C>0$ independent of $q\ge 1$ and $\alpha$ such that
	\begin{align*}
		|b_{q-1}|^2 + |b_q|^2 \leq C 2^{\alpha+1} 
		\left( \sum_{j \ge q} \frac{1}{\gamma_j^{(\alpha,0)}} u_j^2 \right)^{1/2} 
		\left( \sum_{j \ge q-1} \frac{1}{\gamma_j^{(\alpha,0)}} b_j^2 \right)^{1/2}.
	\end{align*}
\end{lemma}

\begin{proof}  
	We may assume that the right-hand side of the estimate in the lemma is finite.
	In view of the sign properties of $h_1$, $h_2$, $h_3$ and (\ref{eq:h2-h1=h3}) we have 
	\begin{align}
		\label{eq:estimate-of-bq-lemma-1}
		|h_1(q,\alpha)| + |h_2(q,\alpha)| = |h_3(q,\alpha)|. 
	\end{align}
	We introduce the abbreviations 
	\begin{align*}
		\alpha_q &:= \frac{h_2(q,\alpha)}{h_3(q,\alpha)}
		= \frac{\alpha(2q+\alpha+1)}{(2q+\alpha+2)(q+\alpha)} \in [0,1], \\
		\varepsilon_q &:= \alpha_q (1 - \alpha_{q+1}) 
		= \frac{\alpha (q+2)(2q+\alpha+1)}{(2q+\alpha+4)(q+1+\alpha)(q+\alpha)} \in [0,1].
	\end{align*}
	By rearranging terms in Lemma \ref{lemma:connection-U-and-Uprime} and using the triangle inequality we get
	\begin{align*}
		|h_3(q,\alpha)| \,|b_{q-1}| \leq |u_q| + |h_2(q,\alpha)|\,|b_q| + |h_1(q,\alpha)|\,|b_{q+1}|. 
	\end{align*}
	We set
	\begin{align*}
		z_q:= \frac{|u_q|}{|h_3(q,\alpha)|}
	\end{align*}
	and by applying (\ref{eq:estimate-of-bq-lemma-1}) we arrive at 
	\begin{align}
		\label{eq:estimate-of-bq-lemma-3}
		|b_{q-1}| \leq z_q  + \alpha_q |b_q| + (1-\alpha_q) |b_{q+1}|. 
	\end{align}
	Iterating (\ref{eq:estimate-of-bq-lemma-3}) once gives 
	\begin{align*}
		|b_{q-1}| 
		&\leq z_q  + \alpha_q \big( z_{q+1} + \alpha_{q+1} |b_{q+1}| + (1-\alpha_{q+1}) |b_{q+2}| \big) 
		+ (1-\alpha_q) |b_{q+1}|  \\
		&\leq z_q  + \alpha_q  z_{q+1} + \big( 1-\alpha_q (1-\alpha_{q+1}) \big) |b_{q+1}| 
		+ \alpha_q (1-\alpha_{q+1}) |b_{q+2}| \\
		&= z_q  + \alpha_q  z_{q+1} + (1-\varepsilon_q) |b_{q+1}| + \varepsilon_q |b_{q+2}|.
	\end{align*}
	Squaring and applying the Cauchy-Schwarz inequality yields 
	\begin{align*}
		b_{q-1}^2 
		&\leq  (z_q + \alpha_q z_{q+1})^2 + 2 (z_q + \alpha_q z_{q+1}) 
		\big( (1-\varepsilon_q) |b_{q+1}| + \varepsilon_q |b_{q+2}| \big) \\
		&\qquad + (1-\varepsilon_q)^2 b_{q+1}^2 + \varepsilon_q^2 b_{q+2}^2 
		+ 2 \varepsilon_q (1-\varepsilon_q) |b_{q+1}|\,|b_{q+2}| \\
		&\leq (z_q + \alpha_q z_{q+1})^2 + 2 (z_q + \alpha_q z_{q+1}) 
		\big( (1-\varepsilon_q) |b_{q+1}| + \varepsilon_q |b_{q+2}| \big) \\
		&\qquad + \big( (1-\varepsilon_q)^2  + \varepsilon_q (1-\varepsilon_q) \big) b_{q+1}^2 
		+ \big( \varepsilon_q^2  + \varepsilon_q (1-\varepsilon_q) \big) b_{q+2}^2. 
	\end{align*}
	We abbreviate the first two addends by 
	\begin{align}
		\label{eq:estimate-of-bq-lemma-4} 
		f_q:= (z_q + \alpha_q z_{q+1})^2 + 2 (z_q + \alpha_q z_{q+1}) 
		\big( (1-\varepsilon_q) |b_{q+1}| + \varepsilon_q |b_{q+2}| \big)
	\end{align}
        and obtain 
	\begin{align*} 
		b_{q-1}^2 \leq f_q + (1 - \varepsilon_q) b_{q+1}^2 + \varepsilon_q b_{q+2}^2, 
	\end{align*}
	which we rewrite as 
	\begin{align}
		\label{eq:estimate-of-bq-lemma-5} 
		b_{q-1}^2 - b_{q+1}^2 \leq f_q + \varepsilon_q\left(  b_{q+2}^2 - b_{q+1}^2\right).
	\end{align}
	Next, as was done in \ref{Point:4} of Section~\ref{sec:outline}, we use a telescoping sum. 
As we assume that the sums on the right-hand side 
        of the statement of this lemma are finite, i.e.,
	\begin{align}
		\label{eq:estimate-of-bq-lemma-5a}
		\sum_{j} \frac{1}{\gamma_j^{(\alpha,0)}} u_j^2 < \infty, \qquad \sum_{j} \frac{1}{\gamma_j^{(\alpha,0)}} b_j^2 < \infty,
	\end{align}
	and since,
$\frac{1}{\gamma_j^{(\alpha,0)}} = (2j+\alpha+1) 2^{-(\alpha+1)}$ we have
$\displaystyle 	
		\sqrt{q}\,|b_q| \rightarrow 0$ for  $\displaystyle q\rightarrow\infty.$
	Hence, 
	\begin{align*}
		&b_{q-1}^2 + b_q^2 
		= \sum_{j=0}^\infty b_{q-1+2j}^2 - b_{q-1+2j+2}^2 + b_{q+2j}^2 - b_{q+2j+2}^2 \\
		&\qquad \leq  \sum_{j=0}^\infty f_{q+2j} + \varepsilon_{q+2j} \left( b_{q+2+2j}^2 - b_{q+1+2j}^2\right) 
		+ f_{q+1+2j} + \varepsilon_{q+1+2j}\left( b_{q+3+2j}^2 - b_{q+2+2j}^2\right) \\
		&\qquad = \sum_{j=0}^\infty f_{q+j} 
		- \sum_{j=0}^\infty \varepsilon_{q+2j} b_{q+1+2j}^2 
		+ \sum_{j=0}^\infty \left(\varepsilon_{q+2j} - \varepsilon_{q+2j+1}\right) b_{q+2+2j}^2 
		+ \sum_{j=0}^\infty \varepsilon_{q+1+2j}  b_{q+3+2j}^2 \\
		&\qquad = \sum_{j=0}^\infty f_{q+j}  
		- \varepsilon_q b_{q+1}^2  
		+\sum_{j=0}^\infty \left(\varepsilon_{q+1+2j} - \varepsilon_{q+2+2j}\right) b_{q+3+2j}^2 
		+ \sum_{j=0}^\infty \left(\varepsilon_{q+2j} - \varepsilon_{q+2j+1}\right) b_{q+2+2j}^2 \\
		&\qquad = \sum_{j=0}^\infty f_{q+j}  
		- \varepsilon_q b_{q+1}^2  
		+\sum_{j=0}^\infty \left(\varepsilon_{q+j} - \varepsilon_{q+j+1}\right) b_{q+2+j}^2.
	\end{align*} 
	We conclude, noting that $\varepsilon_q \geq 0$,
	\begin{align}
		\label{eq:estimate-of-bq-lemma-6}
		b_{q-1}^2 + b_{q}^2 \leq b_{q-1}^2 + b_q^2 + \varepsilon_q b_{q+1}^2 
		\leq F_q + S_{q+2},
	\end{align}
	where
	\begin{align}
		\label{eq:estimate-of-bq-lemma-7}
		F_q &:= \sum_{j \ge q} f_j, \\
		\label{eq:estimate-of-bq-lemma-8}
		S_q &:= \sum_{j \ge q} \varepsilon_j^\prime b_j^2 \quad\textnormal{with}\quad \varepsilon_j^\prime := |\varepsilon_{j-2} - \varepsilon_{j-1}|.	
	\end{align}
	By positivity of $\varepsilon_j^\prime$ and $f_j$ we have $S_{q+1} \leq S_q$ as well as $F_{q+1} \leq F_q$. Therefore, 
	we get from (\ref{eq:estimate-of-bq-lemma-6}) and the definition of $S_q$
	\begin{align*}
		S_q  
		&= \varepsilon_{q}^\prime b_{q}^2 +  \varepsilon_{q+1}^\prime b_{q+1}^2 + S_{q+2} 
		\leq S_{q+2} + \max\{ \varepsilon_{q}^\prime,\varepsilon_{q+1}^\prime \}S_{q+3} + \max\{ \varepsilon_{q}^\prime,\varepsilon_{q+1}^\prime \}F_{q+1} \\
		&\leq (1+ \max\{ \varepsilon_{q}^\prime,\varepsilon_{q+1}^\prime \}) S_{q+2} + \max\{ \varepsilon_{q}^\prime,\varepsilon_{q+1}^\prime \}F_{q}.
	\end{align*}
	Abbreviating 
	\begin{align*}
		\varepsilon_q^{\prime\prime} := \max\{ \varepsilon_{q}^\prime,\varepsilon_{q+1}^\prime \}
	\end{align*}
	we therefore have
	\begin{align}
		\label{eq:estimate-of-bq-lemma-9}
		S_q \leq (1+ \varepsilon_q^{\prime\prime}) S_{q+2} + \varepsilon_q^{\prime\prime}F_{q}.
	\end{align}
	Iterating (\ref{eq:estimate-of-bq-lemma-9}) $N$ times leads to
	\begin{align}
		\label{eq:estimate-of-bq-lemma-9a}
		S_q \leq S_{q+2N+2} \prod_{j=0}^N (1+\varepsilon_{q+2j}^{\prime\prime}) 
		+ \sum_{j=0}^N \varepsilon_{q+2j}^{\prime\prime} F_{q+2j} \prod_{i=0}^{j-1} (1+\varepsilon_{q+2i}^{\prime\prime}).
	\end{align}
	A calculation shows 
	\begin{align}
	\label{eq:estimate-of-bq-lemma-9b} 
		\varepsilon_j^\prime\lesssim\frac{\alpha (\alpha + j)^3}{(\alpha+j)^5} 
		= \frac{\alpha}{(\alpha+j)^2}.
	\end{align}
	From the definition of $S_q$ in (\ref{eq:estimate-of-bq-lemma-8}), (\ref{eq:estimate-of-bq-lemma-5a}), 
	and (\ref{eq:estimate-of-bq-lemma-9b}) it follows that $\lim_{q \rightarrow \infty} S_q = 0$. 
	Furthermore, we can bound the product appearing in 
        (\ref{eq:estimate-of-bq-lemma-9a}) uniformly in $N\,$
with the aid of the elementary fact $\ln (1+x) \leq x$ for $x\geq 0$: 
\begin{align}
		\label{eq:estimate-of-bq-lemma-10}
		\prod_{j=0}^{N} (1 + \varepsilon_{q+2j}^{\prime\prime}) 
		= \exp\left( \sum_{j=0}^{N} \ln (1 + \varepsilon_{q+2j}^{\prime\prime}) \right)
		\leq \exp\left( \sum_{j=0}^{N} \varepsilon_{q+2j}^{\prime\prime} \right),
	\end{align}
	From (\ref{eq:estimate-of-bq-lemma-9b}) we get 
	\begin{align}
		\label{eq:estimate-of-bq-lemma-10a}
		\sum_{j=0}^{N} \varepsilon_{q+2j}^{\prime\prime} \lesssim \sum_{j=0}^{N} \frac{\alpha}{(\alpha+q+2j)^2}
		\lesssim \alpha \sum_{j = q}^{N} \frac{1}{(\alpha+j)^2} \lesssim \frac{\alpha}{\alpha+q} \quad \forall N\geq q,
	\end{align}
	where we used Lemma~\ref{lemma:estimate-trigamma-function} in the last step. 
	Since $\frac{\alpha}{\alpha+q} < 1$, inserting (\ref{eq:estimate-of-bq-lemma-10a}) in
	(\ref{eq:estimate-of-bq-lemma-10}) gives 
	\begin{align}
		\label{eq:estimate-of-bq-lemma-10b}
		\prod_{j=0}^{N} (1 + \varepsilon_{q+2j}^{\prime\prime}) \leq C.
	\end{align}
	Now, by passing to the limit $N \rightarrow \infty$ in (\ref{eq:estimate-of-bq-lemma-9a}), we obtain a closed form bound for $S_q$:
	\begin{align*}
		S_q \leq \sum_{j=0}^\infty \varepsilon_{q+2j}^{\prime\prime} F_{q + 2j} 
		\prod_{i=0}^{j-1} (1 + \varepsilon_{q+2i}^{\prime\prime}).
	\end{align*}
	Applying (\ref{eq:estimate-of-bq-lemma-10a}), (\ref{eq:estimate-of-bq-lemma-10b}), (\ref{eq:estimate-of-bq-lemma-9b}), and the definition of $F_q$ 
	we can simplify
	\begin{align*}
		S_q \lesssim \sum_{j=0}^\infty \varepsilon_{q+2j}^{\prime\prime} F_{q+2j} 
		\lesssim \sum_{j \geq q} \sum_{i \geq j} f_i \frac{\alpha}{(\alpha + j)^2}
		= \sum_{i \ge q} f_i \sum_{j = q}^i \frac{\alpha}{(\alpha+ j)^2}
		\lesssim \frac{\alpha}{\alpha+q} F_q.
	\end{align*}
	Inserting this estimate in (\ref{eq:estimate-of-bq-lemma-6}) and using $\frac{\alpha}{\alpha+q+2} < 1$, we arrive at
	\begin{align}
		\label{eq:estimate-of-bq-lemma-11}
		b_{q-1}^2 + b_{q}^2 \lesssim F_{q} + \frac{\alpha}{\alpha+q+2} F_{q+2} \lesssim F_q + F_{q+2} \lesssim F_{q}.
	\end{align}
	We are left with estimating $F_q$. By the definition of $F_q$ in (\ref{eq:estimate-of-bq-lemma-7}) and the definition 
	of $f_q$ in (\ref{eq:estimate-of-bq-lemma-4}) we have
	\begin{align}
		\label{eq:estimate-of-bq-lemma-12}
		F_q = \sum_{j \geq q} (z_j + \alpha_j z_{j+1})^2 + 2 \sum_{j \geq q} (z_j + \alpha_j z_{j+1}) 
		\big( (1-\varepsilon_j) |b_{j+1}| + \varepsilon_j |b_{j+2}| \big).
	\end{align}
	Now we estimate both sums separately starting with the first one:
	\begin{align}
		\label{eq:estimate-of-bq-lemma-12a}
		\sum_{j \geq q} (z_j + \alpha_j z_{j+1})^2
		\lesssim \sum_{j \geq q} z_j^2 + \underbrace{\alpha_j^2}_\text{$\leq 1$} z_{j+1}^2
		\lesssim \sum_{j \geq q} z_j^2.
	\end{align}
	To proceed further, we use the relation between $u_q$ and $b_q$ from Lemma~\ref{lemma:connection-U-and-Uprime}. 
        Also, we note that $h_3(q,\alpha) \gtrsim 2^{-(\alpha+1)} \gamma_q^{(\alpha,0)}$. Hence, we obtain 
	\begin{align*}
		z_q^2 = \frac{|u_q|^2}{|h_3(q,\alpha)|^2} 
		&\lesssim 2^{\alpha+1} \frac{|u_q|}{\gamma_q^{(\alpha,0)}} \frac{|u_q|}{h_3(q,\alpha)} \\ 
		&= 2^{\alpha+1} \frac{|u_q|}{\gamma_q^{(\alpha,0)}} \frac{1}{h_3(q,\alpha)} 
		\left| h_1(q,\alpha) b_{q+1} + h_2(q,\alpha) b_q + h_3(q,\alpha) b_{q-1}\right| \\
		&\lesssim 2^{\alpha+1} \frac{|u_q|}{\gamma_q^{(\alpha,0)}}
		\big( (1-\alpha_q)  |b_{q+1}| + \alpha_q |b_q| + |b_{q-1}| \big).
	\end{align*}
	Inserting this in (\ref{eq:estimate-of-bq-lemma-12a}), we get by applying the Cauchy-Schwarz inequality for sums
	\begin{align}
\nonumber 
		\sum_{j \geq q} (z_j + \alpha_j z_{j+1})^2
		&\lesssim 2^{\alpha+1} \sum_{j \geq q} \frac{1}{\gamma_j^{(\alpha,0)}} |u_j| 
		\big( (1-\alpha_j)  |b_{j+1}| + \alpha_j |b_j| + |b_{j-1}| \big) \\
		&\lesssim 2^{\alpha+1} \left( \sum_{j \geq q} \frac{1}{\gamma_j^{(\alpha,0)}} |u_j|^2 \right)^{1/2}
		\left( \sum_{j \geq q-1} \frac{1}{\gamma_j^{(\alpha,0)}} |b_{j}|^2 \right)^{1/2}.
	\label{eq:estimate-of-bq-lemma-100}
	\end{align}
	We continue by estimating the second sum in (\ref{eq:estimate-of-bq-lemma-12}). Using again 
  $z_q \lesssim 2^{\alpha+1} |u_q|/\gamma_q^{(\alpha,0)}$ we get 
	\begin{align} 
\nonumber 
		&\sum_{j \geq q} (z_j + \alpha_j z_{j+1}) \big( (1-\varepsilon_j)|b_{j+1}|+\varepsilon_j |b_{j+2}| \big) \\
\nonumber 
		&\qquad\lesssim 2^{\alpha+1} \sum_{j \geq q} \frac{1}{\gamma_j^{(\alpha,0)}} 
		\big( |u_j| + \underbrace{\alpha_j}_\text{$\leq 1$}|u_{j+1}| \big) 
		\big( \underbrace{(1-\varepsilon_j)}_\text{$\leq 1$}|b_{j+1}|+\underbrace{\varepsilon_j}_\text{$\leq 1$}|b_{j+2}| \big) \\
\nonumber 
		&\qquad\lesssim 2^{\alpha+1} \sum_{j \geq q} \frac{1}{\gamma_j^{(\alpha,0)}} 
    \left( |u_j| + |u_{j+1}| \right) \left( |b_{j+1}| + |b_{j+2}| \right) \\
\nonumber 
		&\qquad\lesssim 2^{\alpha+1} \left( \sum_{j \geq q} \frac{1}{\gamma_j^{(\alpha,0)}} |u_j|^2 \right)^{1/2} 
		\left( \sum_{j \geq q} \frac{1}{\gamma_{j+1}^{(\alpha,0)}} |b_{j+1}|^2 \right)^{1/2} \\
		&\qquad\lesssim 2^{\alpha+1} \left( \sum_{j \geq q} \frac{1}{\gamma_j^{(\alpha,0)}} |u_j|^2 \right)^{1/2} 
		\left( \sum_{j \geq q-1} \frac{1}{\gamma_j^{(\alpha,0)}} |b_{j}|^2 \right)^{1/2}.	
	\label{eq:estimate-of-bq-lemma-200}
	\end{align}
	In view of (\ref{eq:estimate-of-bq-lemma-11})
		(\ref{eq:estimate-of-bq-lemma-12}), the bounds 
	(\ref{eq:estimate-of-bq-lemma-100}), 
	(\ref{eq:estimate-of-bq-lemma-200})
 allow us to conclude the proof. 
\end{proof}

\subsection{Stability of truncated Jacobi expansions}
\label{sec:truncated-jacobi-expansion} 
In this section, we study the stability of the operator that 
effects the truncation of an expansion in Jacobi polynomials. 
We analyze this operator in the weighted $H^1$-norm. In other
words: The main result of this section, 
Lemma~\ref{lemma:H1-stability-of-truncation},  
generalizes \cite[Thm.~{2.2}]{canuto-quarteroni82}, 
where the case $\alpha = 0$ is studied, which corresponds to the 
analysis of the $H^1$-stability of the $L^2$-projection. 
This section is closely tied to an alternative proof
of Corollary~\ref{cor:H1-stability} and not immediately 
required for the proof of Theorem~\ref{thm:trace-stability}. 

\begin{lemma}
	\label{lemma:norm-of-Pqprime}
        For all $\alpha$, $q\in\BbbN_0$ 
	\begin{align}
	\label{eq:lemma:norm-of-Pqprime-10}
		\int_{-1}^1 (1-x)^\alpha \left|\big(P_q^{(\alpha,0)}\big)^\prime(x)\right|^2\,dx
		\leq 4 q (q+1+\alpha)^2 \gamma_q^{(\alpha,0)}.
	\end{align}
\end{lemma}

\begin{proof}
	We abbreviate $P_q := P_q^{(\alpha,0)}$ and
	$I_q^2 := \int_{-1}^1 (1-x)^\alpha |P_q^\prime(x)|^2\,dx$.
	The assertion is trivial for the case $q=0$. For $\alpha = 0$ see \cite[(5.3)]{bernardi-maday97}. 
        A direct calculation shows 
	\begin{align*}
		I_0^2 = 0, 
		\qquad I_1^2 = \frac{(\alpha+2)^2}{4} \frac{2^{\alpha+1}}{\alpha+1}, 
                \qquad I_2^2 = \frac{(3+\alpha)(\alpha+2)}{2 (\alpha+1)} 2^{\alpha+1}; 
	\end{align*}
        the assertion of the lemma is therefore true  for $q \in \{0,1,2\}$ and all $\alpha$. 
        Thus, we may assume $\alpha \ge 1$, $q \ge 2$. 
	From Lemma~\ref{lemma:relations-of-jacobi-in-terms-of-jacobi}, (ii) with $q+1$ and $q$ 
        there we get 
	\begin{align}
		\label{eq:lemma-norm-of-Pqprime-1}
		P_{q+1}^\prime = \frac{1}{g_1(q+1,\alpha)} P_q - 
		\frac{g_2(q+1,\alpha)}{g_1(q+1,\alpha) g_1(q,\alpha)} P_{q-1} + (1 - \varepsilon_q) P_{q-1}^\prime - 
		\varepsilon_q P_{q-2}^\prime,
	\end{align}
	where 
	\begin{align*}
		\varepsilon_q:= - \frac{g_2(q+1,\alpha) g_3(q,\alpha)}{g_1(q+1,\alpha) g_1(q,\alpha)} = 
		\frac{\alpha (2q+1+\alpha)(q-1)}{(q+1+\alpha)(2q+\alpha-2)(q+\alpha)}.
	\end{align*}
	We note that $0 \leq \varepsilon_q \leq 1$. Furthermore, we calculate
	\begin{align*}
		\left((1-\varepsilon_q) P_{q-1}^\prime - \varepsilon_q P_{q-2}^\prime\right)^2 
		= (1-\varepsilon_q)^2 (P_{q-1}^\prime)^2 + 
		\varepsilon_q^2 (P_{q-2}^\prime)^2 - 2 \varepsilon_q (1-\varepsilon_q) P_{q-1}^\prime P_{q-2}^\prime
	\end{align*}
	so that by integration, Cauchy-Schwarz, and $0 \leq \varepsilon_q \leq 1$: 
	\begin{align*}
		\int_{-1}^1 (1-x)^\alpha
		\left((1-\varepsilon_q) P_{q-1}^\prime(x) - \varepsilon_q P_{q-2}^\prime(x)\right)^2 
		\,dx \leq 
		\left((1 - \varepsilon_q) I_{q-1} + \varepsilon_q I_{q-2}\right)^2.
	\end{align*}
	By the orthogonality properties of the Jacobi polynomials, we conclude in view of 
	(\ref{eq:lemma-norm-of-Pqprime-1})
	\begin{align*}
		I_{q+1}^2 &= \left(\frac{1}{g_1(q+1,\alpha)}\right)^2 \gamma_q^{(\alpha,0)} + 
		\left(\frac{g_2(q+1,\alpha)}{g_1(q+1,\alpha) g_1(q,\alpha)}\right)^2 \gamma_{q-1}^{(\alpha,0)} \\
		&\qquad + \int_{-1}^1 (1-x)^\alpha
		\left((1-\varepsilon_q) P_{q-1}^\prime(x) - \varepsilon_q P_{q-2}^\prime(x)\right)^2 \,dx \\
		& \leq  \left(\frac{1}{g_1(q+1,\alpha)}\right)^2 \gamma_q^{(\alpha,0)} + 
		\left(\frac{g_2(q+1,\alpha)}{g_1(q+1,\alpha) g_1(q,\alpha)}\right)^2 \gamma_{q-1}^{(\alpha,0)} 
		+ \left((1 - \varepsilon_q) I_{q-1} + \varepsilon_q I_{q-2}\right)^2.
	\end{align*}
	We proceed now by an induction argument on $q$ for fixed $\alpha$. 
        The induction hypothesis and the fact that $q \mapsto q (q+1+\alpha) \gamma^{(\alpha,0)}_{q}$
is monotone increasing in $q$ provides 
$$
\left((1 - \varepsilon_q) I_{q-1} + \varepsilon_q I_{q-2}\right)^2 
\leq 4 (q-1) ((q-1)+\alpha+1)^2\gamma_{q-1}^{(\alpha,0)}, 
$$
and we obtain by some tedious estimates for the other two terms appearing in the bound 
of $I^2_{q+1}$
\iftechreport 
(see Appendix~\ref{app:B})
\fi
: 
	\begin{align*}
		I_{q+1}^2 &\leq 
		\left(\frac{1}{g_1(q+1,\alpha)}\right)^2 \gamma_q^{(\alpha,0)} + 
		\left(\frac{g_2(q+1,\alpha)}{g_1(q+1,\alpha) g_1(q,\alpha)}\right)^2 \gamma_{q-1}^{(\alpha,0)}
		+ 4 (q-1) (q+\alpha)^2 \gamma_{q-1}^{(\alpha,0)}\\
		&\leq 4 (q+1) ((q+1)+1+\alpha)^2 \gamma_{q+1}^{(\alpha,0)}\left[ 
		\frac{4}{4 (q+1) } + \frac{4}{4 (q+1)} + 
		\frac{(q+\alpha)^2 (q-1) \gamma_{q-1}^{(\alpha,0)}}{(q+2+\alpha)^2 (q+1) \gamma_{q+1}^{(\alpha,0)}}
		\right] \\
		&= 4 ((q+1)+1+\alpha)^2 (q+1) \gamma_{q+1}^{(\alpha,0)}
         \left[ 1 - 4 (q-1) \frac{q(q+\alpha+1)-1}{(q+1)(2q+\alpha-1)(q+\alpha+2)^2} \right].
	\end{align*}
        The proof is completed by observing that the expression in brackets is bounded by $1$. 
\end{proof}
\noindent 
We now study the stability of truncating a Jacobi expansion. 
\begin{lemma}
	\label{lemma:H1-stability-of-truncation}
	Let $\alpha \in \BbbN_0$. Let $u_q$ and $b_q$ be defined as in 
	Lemma~\ref{lemma:connection-U-and-Uprime}. 
	Then there exists a constant $C>0$ (which is explicitly available from the proof) 
        independent of $\alpha$ and $N$ such that for every $N\in\BbbN$ we have
	\begin{align*}
		\int_{-1}^1 (1-x)^\alpha 
		\left| 
		\sum_{q=0}^N \frac{1}{\gamma_q^{(\alpha,0)}} u_q \big(P^{(\alpha,0)}_q\big)^\prime (x)
		\right|^2 dx 
		\leq C N  \sum_{q=0}^\infty \frac{1}{\gamma_q^{(\alpha,0)}} |b_q|^2 .
	\end{align*}
\end{lemma}

\begin{proof}
	We abbreviate $P_q:= P_q^{(\alpha,0)}$ and compute 
with the connection formula  (\ref{eq:connection-formula}) between the coefficients 
$u_q$ and $b_q$
	\begin{align*}
		& \sum_{q \ge N+1} \frac{1}{\gamma_q^{(\alpha,0)}} u_q P_q^\prime
		= \sum_{q \ge N+1} \frac{1}{\gamma_q^{(\alpha,0)}} 
		\left[ h_1(q,\alpha) b_{q+1} + h_2(q,\alpha) b_q + h_3(q,\alpha) b_{q-1}\right] P_q^\prime \\
		& = \!\!\sum_{q \ge N} b_q \frac{1}{\gamma_{q+1}^{(\alpha,0)}} h_3(q+1,\alpha) P_{q+1}^\prime 
		+\!\!\! \sum_{q \ge N+1}  b_q \frac{1}{\gamma_{q}^{(\alpha,0)}} h_2(q,\alpha) P_{q}^\prime 
		+ \!\!\!\sum_{q \ge N+2} b_q \frac{1}{\gamma_{q-1}^{(\alpha,0)}} h_1(q-1,\alpha) P_{q-1}^\prime \\
		&= \sum_{q \ge N+2} b_q \left[ \frac{1}{\gamma_{q-1}^{(\alpha,0)}} h_1(q-1,\alpha)P_{q-1}^\prime
		+ \frac{1}{\gamma_{q}^{(\alpha,0)}} h_2(q,\alpha) P_{q}^\prime 
		+ \frac{1}{\gamma_{q+1}^{(\alpha,0)}} h_3(q+1,\alpha) P_{q+1}^\prime \right] \\
		&\qquad+ \frac{1}{\gamma_{N+1}^{(\alpha,0)}} h_2(N+1,\alpha) P_{N+1}^\prime b_{N+1} 
		+ \sum_{q=N}^{N+1} \frac{1}{\gamma_{q}^{(\alpha,0)}} h_3(q+1,\alpha) P_{q+1}^\prime b_{q}.
	\end{align*}
	\noindent With Lemma~\ref{lemma:relations-of-jacobi-in-terms-of-jacobi}, (iii) we therefore conclude 
	\begin{align*}
		& \sum_{q \ge N+1} \frac{1}{\gamma_q^{(\alpha,0)}} u_q P_q^\prime \\
		&= \sum_{q \ge N+2} b_q \frac{1}{\gamma_q^{(\alpha,0)}} P_q 
		+ \frac{1}{\gamma_{N+1}^{(\alpha,0)}} h_2(N+1,\alpha) P_{N+1}^\prime b_{N+1}
		+ \sum_{q=N}^{N+1} \frac{1}{\gamma_{q}^{(\alpha,0)}} h_3(q+1,\alpha) P_{q+1}^\prime b_{q}.
	\end{align*}
	Inserting the result of Lemma~\ref{lemma:norm-of-Pqprime} gives 
	\begin{align*}
		& \int_{-1}^1 (1-x)^\alpha 
		\left| \sum_{q \ge N+1} \frac{1}{\gamma_q^{(\alpha,0)}} u_q P_q^\prime \right|^2\,dx \\
		&\leq \sum_{q \ge N+2} \frac{3}{\gamma_q^{(\alpha,0)}} b_q^2
		+ C N \frac{h_2(N+1,\alpha)^2}{\gamma_{N+1}^{(\alpha,0)}} (N+\alpha)^2 b_{N+1}^2  \\
		& \qquad + C N\sum_{q = N}^{N+1} \frac{1}{\gamma_{q}^{(\alpha,0)}} (N+\alpha)^2 h_3(q+1,\alpha)^2 b_q^2\\
		& \leq C N \sum_{q \ge N} \frac{1}{\gamma_q^{(\alpha,0)}} b_q^2. 
	\end{align*}
	This allows us to conclude the argument since $\big(P_0^{(\alpha,0)}\big)^\prime(x) = 0$. 
\end{proof}

\begin{lemma}[Hardy inequality]
	\label{lemma:increase-weight}
	For $\beta > -1$ and $U\in C^1(0,1) \cap C((0,1])$ there holds 
	\begin{align*}
		\int_0^1 x^\beta \left|U(x)\right|^2 dx \leq \left(\frac{2}{\beta+1}\right)^2 \int_0^1 x^{\beta+2} \left|U^\prime(x)\right|^2 dx
		+ \frac{1}{\beta+1} |U(1)|^2. 
	\end{align*}
\end{lemma}

\begin{proof}
This variant of the Hardy inequality can be shown using \cite[Thm.~{330}]{hardy-littlewood-polya91}. 
\iftechreport 
  See Appendix~\ref{app:B}  for details. 
\fi
\end{proof}

\section{Properties of expansions on the tetrahedron}
\label{sec:expansions}
To save space we will sometimes denote points in $\BbbR^3$ by just one letter, i.e. $\xi = (\xi_1,\xi_2,\xi_3)$
for points in $\T^3$ and $\eta = (\eta_1,\eta_2,\eta_3)$ for points in $\S^3$.
\subsection{Duffy-Transformation}

We recall the definition of the reference triangle, tetrahedron, and the $d$-dimensional hyper cube in 
(\ref{eq:reference-elements}). 
%
\noindent The 3D-Duffy transformation $D:\S^3 \rightarrow \T^3$, \cite{duffy82}, is given by
\begin{align}
		\label{def:3d-duffy-trafo}
	D(\eta_1,\eta_2,\eta_3) 
	:= (\xi_1,\xi_2, \xi_3)
	= \left( \frac{(1+\eta_1)(1-\eta_2)(1-\eta_3)}{4}-1, \frac{(1+\eta_2)(1-\eta_3)}{2}-1, \eta_3 \right)
\end{align}
with inverse
\begin{align*}
	D^{-1}(\xi_1,\xi_2,\xi_3) = (\eta_1,\eta_2, \eta_3) = 
	\left( -2~\frac{1+\xi_1}{\xi_2+\xi_3}-1, 2~\frac{1+\xi_2}{1-\xi_3}-1, \xi_3 \right).
\end{align*}
\begin{lemma}
	\label{lemma:duffy-properties-1}
	The Duffy transformation is a bijection between the (open) cube 
	$\S^3$ and the (open) tetrahedron $\T^3$. Additionally,
	\[
	D^\prime(\eta):=\left[\frac{\partial\xi_i}{\partial\eta_j}\right]_{i,j=1}^3 = 
	\left[
	\begin{array}{ccc}
		\frac{1}{4}(1 - \eta_2)(1 - \eta_3) & 0 & 0\\
	  -\frac{1}{4}(1 + \eta_1)(1 - \eta_3) & \frac{1}{2} (1 - \eta_3) & 0\\
	  -\frac{1}{4}(1 + \eta_1)(1 - \eta_2) &  -\frac{1}{2} (1 + \eta_2) &1 
	\end{array}
	\right]^\top,
	\]
	\begin{align*}
		\left(D^\prime(\eta)\right)^{-1}
		&=
	 	\frac{1}{(1 - \eta_2)(1-\eta_3)}
	 	\left[
	 	\begin{array}{ccc}
	  	4 & 2(1+\eta_1) &  2(1 +\eta_1) \\
	  	0 & 2 (1-\eta_2) & 1 - \eta_2^2 \\
	 		0 & 0 & (1-\eta_2)(1-\eta_3)
	 	\end{array}
	 	\right],
	\end{align*}
	\begin{align*}
	  \det D'=\left(\frac{1-\eta_2}{2}\right)\left(\frac{1-\eta_3}{2}\right)^2.
	\end{align*}
\end{lemma}

\begin{proof}
	See, for example, \cite{karniadakis-sherwin99}. 
\end{proof}

Restricted to the face $\eta_3=-1$ of the cube $\S^3$, 
the 3D Duffy transformation reduces to the 2D version of the 
Duffy transformation, and there holds:   
\begin{lemma}
	\label{lemma:duffy-properties-2}
	Let $D$ be the Duffy transformation and 
$\Gamma := \T^2 \times \{-1\}$. Set $\tilde \Gamma:= \S^2 \times \{-1\}$. 
  Then $D(\tilde\Gamma) = \Gamma$ and 
  $D|_{\tilde \Gamma}: \tilde \Gamma \rightarrow  \Gamma$ is a bijection. 
  Furthermore, for $\tilde u =  u \circ D$ we have 
  $$\int_{(\xi_1,\xi_2,-1) \in \Gamma} u^2\,d\xi_1\,d\xi_2 
  = \int_{(\eta_1,\eta_2,-1) \in \tilde \Gamma} \tilde u^2 \frac{1-\eta_2}{2}\,d\eta_1\,d\eta_2. 
$$
In fact, $D|_{\tilde \Gamma}$ is the standard Duffy transformation from $\S^2$ to $\T^2$. 
\end{lemma}

\begin{proof} Follows by inspection. 
\end{proof}

\subsection{Orthogonal polynomials on tetrahedra}

\noindent In terms of the Jacobi polynomials $P_n^{(\alpha,\beta)}$ we introduce orthogonal polynomials 
on the reference tetrahedron $\T^3$ often associated with the names of 
Dubiner or Koornwinder, \cite{dubiner91,koornwinder75,karniadakis-sherwin99}: 

\begin{lemma}[orthogonal polynomials on $\T^3$]
	\label{lemma:orthogonal-polynomials-on-tet}
	For $p$, $q$, $r \in \BbbN_0$ set $\psi_{p,q,r} := \tilde{\psi}_{p,q,r} \circ D^{-1}$,
	where $\tilde{\psi}_{p,q,r}$ is defined by
	\begin{align*}
		\tilde{\psi}_{p,q,r}(\eta) := P_p^{(0,0)}(\eta_1)P_q^{(2p+1,0)}(\eta_2)P_r^{(2p+2q+2,0)}(\eta_3) 
		\left( \frac{1-\eta_2}{2} \right)^p \left( \frac{1-\eta_3}{2} \right)^{p+q}.
	\end{align*}
	Then the functions $\psi_{p,q,r}$ are $L^2(\T^3)$-orthogonal, 
        satisfy $\psi_{p,q,r}\in \P_{p+q+r}(\T^3)$, and
	\begin{align*}
		\int_{\T^3} \psi_{p,q,r}(\xi)\psi_{p',q',r'}(\xi)\,d\xi 
		&= \delta_{p,p'}\delta_{q,q'}\delta_{r,r'} \frac{2}{2p+1} \frac{2}{2p+2q+2} \frac{2}{2p+2q+2r+3} \\
		&= \delta_{p,p'}\delta_{q,q'}\delta_{r,r'} \gamma_p^{(0,0)} \frac{\gamma_q^{(2p+1,0)}}{2^{2p+1}} \frac{\gamma_r^{(2p+2q+2,0)}}{2^{2p+2q+2}}.
	\end{align*}
Furthermore, any $u \in L^2(\T^3)$ can be expanded as 
\begin{align}
	\label{eq:expansion-u}
		& u = \sum_{p,q,r=0}^{\infty} \frac{\langle\psi_{p,q,r},u\rangle_{L^2(\T^3)}}{\|\psi_{p,q,r}\|_{L^2(\T^3)}^2} \psi_{p,q,r} 
		= \sum_{p,q,r=0}^{\infty} \frac{1}{\gamma_p^{(0,0)}}\frac{2^{2p+1}}{\gamma_q^{(2p+1,0)}}\frac{2^{2p+2q+2}}{\gamma_r^{(2p+2q+2,0)}} u_{p,q,r} \psi_{p,q,r}, \\
	\label{def:coefficients-u-pqr} 
	& u_{p,q,r} := \langle\psi_{p,q,r},u\rangle_{L^2(\T^3)}.
\end{align}
\end{lemma}

\begin{proof}
\iftechreport 
	The proof can be found in Appendix~\ref{app:B}.
\fi
\end{proof}
\begin{remark}[orthogonal polynomials in 2D]
{\rm 
Lemma~\ref{lemma:duffy-properties-2} stated that $D|_{\eta_3 = -1}$ reduces to the 
standard 2D version of the Duffy transformation 
(see, e.g., \cite[Sec.~{3.2.1.1}]{karniadakis-sherwin99} or \cite[(3.2.20)]{melenk02}
for concrete formulas).  Correspondingly, the polynomials
$$
\tilde \psi^{2D}_{p,q}:= \tilde \psi_{p,q,0}(\cdot,\cdot,-1)
=  P_p^{(0,0)}(\eta_1) P_q^{(2p+1,0)}(\eta_2) 
\left(\frac{1-\eta_2}{2}\right)^{p}
$$
lead to orthogonal polynomials $\psi^{2D}_{p,q}$ given by 
$\tilde \psi^{2D}_{p,q} = \psi^{2D}_{p,q} \circ D|_{\eta_3 = -1}$ 
with orthogonality properties 
	\begin{align*}
		\int_{\T^2} \psi^{2D}_{p,q}(\xi)\psi^{2D}_{p',q'}(\xi)
\,d\xi_1\,d\xi_2
		&= \delta_{p,p'}\delta_{q,q'}\frac{2}{2p+1} \frac{2}{2p+2q+2} 
		= \delta_{p,p'}\delta_{q,q'}\gamma_p^{(0,0)} \frac{\gamma_q^{(2p+1,0)}}{2^{2p+1}} . 
	\end{align*}
An expansion analogous to 
	(\ref{eq:expansion-u}), (\ref{def:coefficients-u-pqr}) below is valid; 
$u \in L^2(\T^2)$ can be written as 
\begin{eqnarray*}
u &=& \sum_{p,q=0}^\infty \frac{\langle u,\psi^{2D}_{p,q}\rangle_{L^2(\T^2)}}{\|\psi^{2D}_{p,q}\|^2_{L^2(\T^2)}} \psi^{2D}_{p,q} 
= \sum_{p,q} \frac{1}{\gamma_p^{(0,0)}} \frac{2^{2p+1}}{\gamma_q^{(2p+1,0)} }
u_{p,q} \psi^{2D}_{p,q}, \\
u_{p,q} &=& \langle \psi^{2D}_{p,q} ,u\rangle_{L^2(\T^2)}.  
\end{eqnarray*}
\iftechreport See Appendix~\ref{app:B} for details. 
\fi
\eremk}
\end{remark}
\subsection{Expansion in terms of $\psi_{p,q,r}$}

A basic ingredient of the proofs of Theorem~\ref{thm:trace-stability} and 
Corollary~\ref{cor:H1-stability} is the reduction of the analysis to one-dimensional settings, 
for which we have provided the necessary results in Section~\ref{sec:1D}. As already flagged 
in \ref{Point:3} of Section~\ref{sec:outline}, 
the $\eta_3$-variable plays a special role. This is captured in  
Definition~\ref{def:U-pq-and-tilde-U-pq} below, where the 
functions $\widetilde U_{p,q}$ and $\widetilde U^\prime_{p,q}$ 
(with expansion coefficients $\tilde u_{p,q,r}$ and $\tilde u^\prime_{p,q,r}$)
are introduced.  Before that, we introduce for a function $u$ 
defined in $\T^3$ the transformed function $\tilde{u} := u \circ D$ and get 
\begin{align}
\label{eq:u_pqr}
	u_{p,q,r} &= \int_{\T^3} u(\xi) \psi_{p,q,r}(\xi)\, d\xi 
	= \int_{\S^3} \tilde{u}(\eta)\tilde{\psi}_{p,q,r}(\eta) \left(\frac{1-\eta_2}{2}\right) \left(\frac{1-\eta_3}{2}\right)^2\, d\eta \\
\nonumber 
	&= \int_{\S^3} \tilde{u}(\eta) 
			P_p^{(0,0)}(\eta_1)P_q^{(2p+1,0)}(\eta_2)P_r^{(2p+2q+2,0)}(\eta_3) 
			\left(\frac{1-\eta_2}{2}\right)^{p+1} \left(\frac{1-\eta_3}{2}\right)^{p+q+2} d\eta. 
\end{align}

\begin{definition}
	\label{def:U-pq-and-tilde-U-pq}
	Let $p,q\in\BbbN_0$ and $u\in L^2(\T^3)$. Define the functions $U_{p,q}: (-1,1) \rightarrow \BbbR$ and $\widetilde{U}_{p,q}: (-1,1) \rightarrow \BbbR$ 
	as well as the coefficients $\tilde{u}_{p,q,r}$ and $\tilde{u}'_{p,q,r}$ by
	\begin{align}
         \label{eq:U_pq}
		U_{p,q}(\eta_3) &:= \int_{-1}^1 \int_{-1}^1 \tilde{u}(\eta)P_p^{(0,0)}(\eta_1)P_q^{(2p+1,0)}(\eta_2)\left(\frac{1-\eta_2}{2}\right)^{p+1} d\eta_1 d\eta_2, \\
		\label{def:tilde-U-pq}
		\widetilde{U}_{p,q}(\eta_3) &:= \frac{U_{p,q}(\eta_3)}{(1-\eta_3)^{p+q}}, \\
         \label{eq:tildeu_pqr}
		\tilde{u}_{p,q,r} &:= \int_{-1}^1 (1-\eta_3)^{2p+2q+2} \widetilde{U}_{p,q}(\eta_3) P_r^{(2p+2q+2,0)}(\eta_3)\, d\eta_3, \\
         \label{eq:tildeuprime_pqr}
		\tilde{u}'_{p,q,r} &:= \int_{-1}^1 (1-\eta_3)^{2p+2q+2} \widetilde{U}'_{p,q}(\eta_3) P_r^{(2p+2q+2,0)}(\eta_3)\, d\eta_3.
	\end{align}
\end{definition}

\noindent With this notation, we have by comparing (\ref{eq:u_pqr}) with (\ref{eq:tildeu_pqr})
\begin{align}
	\label{eq:one-dimensional-form}
	u_{p,q,r}
	&= \frac{1}{2^{p+q+2}} \int_{-1}^1 (1-\eta_3)^{2p+2q+2} \widetilde{U}_{p,q}(\eta_3) P_r^{(2p+2q+2,0)}(\eta_3)\, d\eta_3
	= \frac{1}{2^{p+q+2}} \tilde{u}_{p,q,r}.  
\end{align}
Since for sufficiently smooth functions $u$ the transformed function $\tilde u$ is constant 
on $\eta_3 = 1$, the orthogonality properties of the Jacobi polynomials give us 
\begin{equation}
\label{eq:Upq(1) = 0}
U_{p,q}(1) = 0 \qquad \mbox{ for $(p,q) \ne (0,0)$.} 
\end{equation}
\subsection{Properties of the univariate functions $U_{p,q}$ and $\widetilde{U}_{p,q}$} 

We start with some preliminary considerations regarding estimates for partial derivatives of 
the transformed function $\tilde{u}$. We have 
\begin{align}
	\label{eq:partial-eta1-tilde-u}
	\partial_{\eta_1} \tilde{u}(\eta) &= \frac{(1-\eta_2)(1-\eta_3)}{4} (\partial_1 u)\circ D(\eta), \\
	\label{eq:partial-eta2-tilde-u}
	\partial_{\eta_2} \tilde{u}(\eta) &= -\frac{(1+\eta_1)(1-\eta_3)}{4} (\partial_1 u)\circ D(\eta) + \frac{(1-\eta_3)}{2} (\partial_2 u)\circ D(\eta), \\
	\label{eq:partial-eta3-tilde-u}
	\partial_{\eta_3} \tilde{u}(\eta) &= -\frac{(1+\eta_1)(1-\eta_2)}{4} (\partial_1 u)\circ D(\eta) - \frac{(1+\eta_2)}{2} (\partial_2 u)\circ D(\eta)
	+ (\partial_3 u)\circ D(\eta),
\end{align}
where $\partial_i$ denotes the partial derivative with respect to the $i$-th argument. In particular, 
we get 
\begin{align}
	\label{eq:estimate-partial-eta1-tilde-u}
	\int_{\S^3}\left|\partial_{\eta_1} \tilde{u}(\eta)\right|^2 \left(\frac{2}{1-\eta_2}\right) d\eta
	&= \int_{\S^3}\left|(\partial_1 u)\circ D(\eta)\right|^2 \left(\frac{1-\eta_2}{2}\right)\left(\frac{1-\eta_3}{2}\right)^2 d\eta \notag \\
	&= \|\partial_{\xi_1}u\|^2_{L^2(\T^3)} \leq \|\nabla u\|^2_{L^2(\T^3)}, 
\end{align}
and 
\begin{align}
\notag
	& \int_{\S^3}\left|\partial_{\eta_2} \tilde{u}(\eta)\right|^2 \left(\frac{1-\eta_2}{2}\right) d\eta \\
	& \quad \lesssim \int_{\S^3}\left|(\partial_1 u)\circ D(\eta)\right|^2 \left(\frac{1+\eta_1}{2}\right)^2 \left(\frac{1-\eta_2}{2}\right) \left(\frac{1-\eta_3}{2}\right)^2 d\eta
	+ \|\partial_{\xi_2}u\|^2_{L^2(\T^3)} \notag \\
	\label{eq:estimate-partial-eta2-tilde-u}
	&\quad \lesssim \|\partial_{\xi_1}u\|^2_{L^2(\T^3)} + \|\partial_{\xi_2}u\|^2_{L^2(\T^3)} \lesssim \|\nabla u\|^2_{L^2(\T^3)}.
\end{align}
These estimates are useful to prove the following lemmas. 

\begin{lemma}[properties of $U_{p,q}$]
	\label{lemma:properties-of-U-pq}
	Let $u\in H^1(T)$ and $U_{p,q}$ be defined in Definition~\ref{def:U-pq-and-tilde-U-pq}. 
	Then there exists a constant $C>0$ independent of $u$ such that 
	\begin{align}
		\label{eq:properties-of-U-pq-1}
		\sum_{p,q=0}^{\infty} \frac{1}{\gamma_p^{(0,0)}} \frac{2^{2p+1}}{\gamma_q^{(2p+1,0)}} 
		\int_{-1}^1 |U_{p,q}(\eta_3)|^2 \left(\frac{1-\eta_3}{2}\right)^2 d\eta_3 
		&= \|u\|^2_{L^2(\T^3)}, \\
		\label{eq:properties-of-U-pq-2}
		\sum_{p,q=0}^{\infty} \frac{1}{\gamma_p^{(0,0)}} \frac{2^{2p+1}}{\gamma_q^{(2p+1,0)}} 
		\int_{-1}^1 |U_{p,q}^{\prime}(\eta_3)|^2 \left(\frac{1-\eta_3}{2}\right)^2 d\eta_3
		&\leq C \|\nabla u\|^2_{L^2(\T^3)}, \\
		\label{eq:properties-of-U-pq-3}
		\sum_{p,q=0}^{\infty} \frac{1}{\gamma_p^{(0,0)}} \frac{2^{2p+1}}{\gamma_q^{(2p+1,0)}} 
		(p+q)^2 \int_{-1}^1 \left|U_{p,q}(\eta_3)\right|^2 d\eta_3
		&\leq C \|\nabla u\|^2_{L^2(\T^3)}.
	\end{align}
	Furthermore, we have for $\Gamma = \T^2 \times \{-1\}$
	\begin{align}
		\label{eq:properties-of-U-pq-4}
		\sum_{p=0}^\infty \frac{1}{\gamma_p^{(0,0)}} \frac{2^{2p+1}}{\gamma_q^{(2p+1,0)}} |U_{p,q}(-1)|^2 = \|u\|^2_{L^2(\Gamma)}.
	\end{align}
\end{lemma} 

\begin{proof} 
{\bf Proof of (\ref{eq:properties-of-U-pq-1}) 
and (\ref{eq:properties-of-U-pq-2}):}
	The fact that $P_p^{(0,0)}(\eta_1)P_q^{(2p+1,0)}(\eta_2)\left(\frac{1-\eta_2}{2}\right)^{p}$ 
	are orthogonal polynomials in a weighted $L^2$-space on $\S^2$ and
	the definition of $U_{p,q}$ imply (for fixed $\eta_3$) the representation 
\iftechreport (cf. (\ref{eq:lemma:2D-Duffy-expansion-square}) for details) 
\fi
	\begin{align*}
		\tilde{u}(\eta) = \sum_{p,q=0}^\infty \frac{1}{\gamma_p^{(0,0)}} \frac{2^{2p+1}}{\gamma_q^{(2p+1,0)}} 
		U_{p,q}(\eta_3) P_p^{(0,0)}(\eta_1)P_q^{(2p+1,0)}(\eta_2)\left(\frac{1-\eta_2}{2}\right)^{p},
	\end{align*}
	which in turn gives 
	\begin{align}
		\label{eq:foo1}
		\int_{\S^2} |\tilde{u}(\eta)|^2 \left(\frac{1-\eta_2}{2}\right) d\eta_1 d\eta_2 
		= \sum_{p,q=0}^\infty \frac{1}{\gamma_p^{(0,0)}} \frac{2^{2p+1}}{\gamma_q^{(2p+1,0)}} \left|U_{p,q}(\eta_3)\right|^2.
	\end{align}
	Since $\det D^\prime = \left(\frac{1-\eta_2}{2}\right)\left(\frac{1-\eta_3}{2}\right)^2$, multiplication with $\left(\frac{1-\eta_3}{2}\right)^2$ and integration in $\eta_3$ gives
	(\ref{eq:properties-of-U-pq-1}). 

	Similar to the representation of $\tilde{u}$ above, we get for $\partial_{\eta_3}\tilde{u}$
	\begin{align*}
		\partial_{\eta_3}\tilde{u}(\eta) 
		= \sum_{p,q=0}^\infty \frac{1}{\gamma_p^{(0,0)}} \frac{2^{2p+1}}{\gamma_q^{(2p+1,0)}} 
		U_{p,q}^\prime(\eta_3) P_p^{(0,0)}(\eta_1)P_q^{(2p+1,0)}(\eta_2)\left(\frac{1-\eta_2}{2}\right)^{p}.
	\end{align*}
	Reasoning as in the case of (\ref{eq:properties-of-U-pq-1}) yields
	\begin{align*}
		\sum_{p,q=0}^\infty \frac{1}{\gamma_p^{(0,0)}} \frac{2^{2p+1}}{\gamma_q^{(2p+1,0)}} \int_{-1}^1 \left| U_{p,q}^\prime(\eta_3) \right|^2 \left(\frac{1-\eta_3}{2}\right)^2 d\eta_3
		\lesssim \|\nabla u\|^2_{L^2(\T^3)},
	\end{align*}
	which immediately leads to (\ref{eq:properties-of-U-pq-2}).

{\bf Proof of (\ref{eq:properties-of-U-pq-3}) for $p = 0$:}
        We restrict our attention for the double sum (\ref{eq:properties-of-U-pq-3}) to 
        the special case $p = 0$. We start with the observation 
	\begin{align*}
		U_{0,q} = \int_{-1}^1 \left(\int_{-1}^1 \tilde{u}(\eta) d\eta_1\right) P_q^{(1,0)}(\eta_2)\left(\frac{1-\eta_2}{2}\right) d\eta_2.
	\end{align*}
	Expanding for fixed $\eta_3$ the function $\eta_2 \mapsto \int_{-1}^1 \tilde{u}(\eta_1,\eta_2,\eta_3) d\eta_1$
	as well as its derivative in terms of the orthogonal polynomials $P_q^{(1,0)}$, 
        we get with Lemma~\ref{lemma:connection-U-and-Uprime-2} 
	\begin{align*}
		& \sum_{q=0}^{\infty} \frac{1}{\gamma_0^{(0,0)}} \frac{1}{\gamma_q^{(1,0)}} 
		\,q^2 \int_{-1}^1 \left|U_{0,q}(\eta_3)\right|^2 d\eta_3
		\leq 2\sum_{q=1}^{\infty} \frac{1}{\gamma_q^{(1,0)}} 
		(q+1)^2 \int_{-1}^1 \left|U_{0,q}(\eta_3)\right|^2 d\eta_3 \\
		&\lesssim \int_{-1}^1 \int_{-1}^1 \left| \int_{-1}^1 \partial_{\eta_2}\tilde{u}(\eta) d\eta_1 \right|^2 (1-\eta_2) d\eta_2 d\eta_3
		\lesssim \int_{\S^3} \left|\partial_{\eta_2}\tilde{u}(\eta)\right|^2 \left(\frac{1-\eta_2}{2}\right) d\eta 
		\lesssim \|\nabla u\|^2_{L^2(\T^3)}, 
	\end{align*}
	where we appealed to (\ref{eq:estimate-partial-eta2-tilde-u}) in the last estimate.

{\bf Proof of (\ref{eq:properties-of-U-pq-3}) for $p \ge 1$ and $q \ge 1$:} 
We restrict our attention in the double sum  in (\ref{eq:properties-of-U-pq-3}) 
to the case $p\ge 1$ in conjunction with $q \ge 1$. 

By definition, we have 
	\begin{align}
		\label{eq:U-pq}
		U_{p,q}(\eta_3) 
		= \frac{1}{2^{p+1}}\int_{\S^2} \tilde{u}(\eta) P^{(0,0)}_p(\eta_1) P^{(2p+1,0)}_{q}(\eta_2) (1-\eta_2)^{p+1} d\eta_1 d\eta_2.
	\end{align}
	In this double integral, we consider the integration in $\eta_2$. Integration by parts then yields
	\begin{align*}
		&\int_{-1}^1 \tilde{u}(\eta) P_q^{(2p+1,0)}(\eta_2) (1-\eta_2)^{p+1} d\eta_2 \\
		&\qquad = \left(\tilde{u}(\eta) (1-\eta_2)^{p+1} \widehat{P}_{q+1}^{(2p+1,0)}(\eta_2)\right) \Big|_{-1}^1
		- \int_{-1}^1 \partial_{\eta_2}\left(\tilde{u}(\eta) (1-\eta_2)^{p+1}\right) \widehat{P}_{q+1}^{(2p+1,0)}(\eta_2) d\eta_2 \\
		&\qquad = - \int_{-1}^1 \partial_{\eta_2}\tilde{u}(\eta) (1-\eta_2)^{p+1} \widehat{P}_{q+1}^{(2p+1,0)}(\eta_2)
		- (p+1) \tilde{u}(\eta) (1-\eta_2)^{p} \widehat{P}_{q+1}^{(2p+1,0)}(\eta_2) d\eta_2.
	\end{align*}  
	Hence, we obtain by inserting into (\ref{eq:U-pq})
	\begin{align*}
		U_{p,q}(\eta_3) 
		&= -\frac{1}{2^{p+1}} \int_{\S^2} \partial_{\eta_2}\tilde{u}(\eta) P^{(0,0)}_p(\eta_1) (1-\eta_2)^{p+1} \widehat{P}_{q+1}^{(2p+1,0)}(\eta_2) d\eta_1 d\eta_2 \\
		&\qquad + \frac{p+1}{2^{p+1}} \int_{\S^2} \tilde{u}(\eta) P^{(0,0)}_p(\eta_1) (1-\eta_2)^{p} \widehat{P}_{q+1}^{(2p+1,0)}(\eta_2) d\eta_1 d\eta_2 \\
		&= -\frac{1}{2^{p+1}} \int_{\S^2} \partial_{\eta_2}\tilde{u}(\eta) P^{(0,0)}_p(\eta_1) (1-\eta_2)^{p+1} \widehat{P}_{q+1}^{(2p+1,0)}(\eta_2) d\eta_1 d\eta_2 \\
		&\qquad - \frac{p+1}{2^{p+1}} \int_{\S^2} (\partial_{\eta_1}\tilde{u})(\eta) \widehat{P}^{(0,0)}_{p+1}(\eta_1) (1-\eta_2)^{p} \widehat{P}_{q+1}^{(2p+1,0)}(\eta_2) d\eta_1 d\eta_2,
	\end{align*}
	where in the last equation we used integration by parts in $\eta_1$ and the orthogonality property $\int_{-1}^1 P^{(0,0)}_p(t) dt = 0$ for $p\geq 1$.
	With the abbreviation $g_i:=g_i(q+1,2p+1)$, $i=1$, $2$, $3$, we have by 
Lemma~\ref{lemma:relations-of-jacobi-in-terms-of-jacobi}, (ii) for $p$, $q \geq 1$ the following relationships: 
	\begin{align}
		\label{eq:beuchler-schoeberl-1}
		\widehat{P}_{q+1}^{(2p+1,0)}(\eta_2) & = g_1 P^{(2p+1,0)}_{q+1}(\eta_2) 
		+ g_2 P^{(2p+1,0)}_{q}(\eta_2) + g_3 P^{(2p+1,0)}_{q-1}(\eta_2), \\
		\label{eq:beuchler-schoeberl-2}
		\widehat{P}_{p+1}^{(0,0)}(\eta_1) 
		& = \frac{1}{2p+1} \left(P^{(0,0)}_{p+1}(\eta_1) - P^{(0,0)}_{p-1}(\eta_1)\right).
	\end{align}
	Furthermore, we introduce two abbreviations 
	\begin{align}
		\label{eq:z-pq}
		z_{p,q}(\eta_3) &:= \int_{\S^2} (\partial_{\eta_2}\tilde{u})(\eta) P^{(0,0)}_p(\eta_1)
		\left(\frac{1-\eta_2}{2}\right)^{p+1} P_q^{(2p+1,0)}(\eta_2) d\eta_1 d\eta_2, \\
		\tilde{z}_{p,q}(\eta_3) &:= \int_{\S^2} ((\partial_{1} u)\circ D)(\eta) P^{(0,0)}_p(\eta_1)
		\left(\frac{1-\eta_2}{2}\right)^{p+1} P_q^{(2p+1,0)}(\eta_2) d\eta_1 d\eta_2.
		\label{eq:tilde-z-pq}
	\end{align}
	Since we have (\ref{eq:partial-eta1-tilde-u}),
	using (\ref{eq:beuchler-schoeberl-1}), (\ref{eq:beuchler-schoeberl-2}), (\ref{eq:z-pq}), 
        and (\ref{eq:tilde-z-pq}) we get
	\begin{align*}
		U_{p,q}(\eta_3) 
		&= -\big(g_1z_{p,q+1}(\eta_3) + g_2z_{p,q}(\eta_3) + g_3z_{p,q-1}(\eta_3)\big)\\ 
		&\qquad - \left(\frac{p+1}{2p+1}\right)\left(\frac{1-\eta_3}{4}\right)
		\Big[g_1\big(\tilde{z}_{p+1,q+1}(\eta_3) - \tilde{z}_{p-1,q+1}(\eta_3)\big) \\
		&\qquad + g_2\big(\tilde{z}_{p+1,q}(\eta_3) - \tilde{z}_{p-1,q}(\eta_3)\big)
		+ g_3\big(\tilde{z}_{p+1,q-1}(\eta_3) - \tilde{z}_{p-1,q-1}(\eta_3)\big) \Big].
	\end{align*}
	We use $g_1$, $g_2$, $g_3 \lesssim \frac{1}{p+q}$ to arrive at
	\begin{align}
		\label{eq:U-pq-in-terms-of-z-pq}
		(p+q)|U_{p,q}(\eta_3)| \lesssim \sum_{j=0}^2 \left(\frac{1-\eta_3}{2}\right)\Big(|\tilde{z}_{p+1,q+1-j}(\eta_3)| 
		+ |\tilde{z}_{p-1,q+1-j}(\eta_3)|\Big) + |z_{p,q+1-j}(\eta_3)|.
	\end{align}
	To estimate the terms on the right-hand side we note that the abbreviations 
	$z_{p,q}$ and $\tilde{z}_{p,q}$ lead us to the representations
	\begin{align*}
		\partial_{\eta_2}\tilde{u}(\eta) 
		&= \sum_{p,q=0}^\infty \frac{1}{\gamma_p^{(0,0)}}\frac{2^{2p+1}}{\gamma_q^{(2p+1,0)}} z_{p,q}(\eta_3) 
		P_p^{(0,0)}(\eta_1)\left(\frac{1-\eta_2}{2}\right)^p P_q^{(2p+1,0)}(\eta_2), \\
		((\partial_{1} u)\circ D)(\eta)
		&= \sum_{p,q=0}^\infty \frac{1}{\gamma_p^{(0,0)}}\frac{2^{2p+1}}{\gamma_q^{(2p+1,0)}} \tilde{z}_{p,q}(\eta_3) 
		P_p^{(0,0)}(\eta_1)\left(\frac{1-\eta_2}{2}\right)^p P_q^{(2p+1,0)}(\eta_2).
	\end{align*}
	Since the polynomials $P_p^{(0,0)}(\eta_1)\left(\frac{1-\eta_2}{2}\right)^p P_q^{(2p+1,0)}(\eta_2)$
	are orthogonal polynomials on the reference triangle $\T^2$ we have
	\begin{align}
		\label{eq:representation-1}
		\int_{\S^2} |\partial_{\eta_2}\tilde{u}(\eta)|^2 \left(\frac{1-\eta_2}{2}\right) d\eta_1 d\eta_2 
		&= \sum_{p,q=0}^\infty \frac{1}{\gamma_p^{(0,0)}}\frac{2^{2p+1}}{\gamma_q^{(2p+1,0)}} |z_{p,q}(\eta_3)|^2, \\
		\int_{\S^2} |((\partial_{1} {u})\circ D)(\eta)|^2 \left(\frac{1-\eta_2}{2}\right) d\eta_1 d\eta_2
		&= \sum_{p,q=0}^\infty \frac{1}{\gamma_p^{(0,0)}}\frac{2^{2p+1}}{\gamma_q^{(2p+1,0)}} |\tilde{z}_{p,q}(\eta_3)|^2.
		\label{eq:representation-2}
	\end{align}
	Formula (\ref{eq:representation-1}) together with an integration in $\eta_3$
	and an application of (\ref{eq:estimate-partial-eta2-tilde-u}) gives
	\begin{align*}
		\sum_{p,q=0}^\infty \frac{1}{\gamma_p^{(0,0)}}\frac{2^{2p+1}}{\gamma_q^{(2p+1,0)}} 
		\int_{-1}^1 |z_{p,q}(\eta_3)|^2 d\eta_3
		= \int_{\S^3} |\partial_{\eta_2}\tilde{u}(\eta)|^2 \left(\frac{1-\eta_2}{2}\right) d\eta
		\lesssim \|\nabla u\|_{L^2(\T^3)}^2.
	\end{align*}
	{}From the representation (\ref{eq:representation-2}) we get by a multiplication 
        with $\left(\frac{1-\eta_3}{2}\right)^2$, an integration in $\eta_3$, and the use of 
        (\ref{eq:estimate-partial-eta1-tilde-u}) that 
	\begin{align*}
		&\sum_{p,q=0}^\infty \frac{1}{\gamma_p^{(0,0)}}\frac{2^{2p+1}}{\gamma_q^{(2p+1,0)}} 
		\int_{-1}^1 |\tilde{z}_{p,q}(\eta_3)|^2 \left(\frac{1-\eta_3}{2}\right)^2 d\eta_3 \\
		&\qquad = \int_{\S^3} |((\partial_{1}{u})\circ D)(\eta)|^2 \left(\frac{1-\eta_2}{2}\right) \left(\frac{1-\eta_3}{2}\right)^2 d\eta
		\lesssim \|\nabla u\|_{L^2(\T^3)}^2.
	\end{align*}

{\bf Proof of (\ref{eq:properties-of-U-pq-3}) for $p \ge 1$ and $q = 0$:} The remaining 
case for the double sum in (\ref{eq:properties-of-U-pq-3}) is $p \ge 1$ combined with $q = 0$. 
This no longer difficult since we have with $q = 0$ that the relationship 
(\ref{eq:beuchler-schoeberl-1}) simplifies to 
	\begin{align}
		\label{eq:beuchler-schoeberl-1-a}
		\widehat{P}_{q+1}^{(2p+1,0)}(\eta_2) & = g_1 P^{(2p+1,0)}_{q+1}(\eta_2) 
		+ g_2 P^{(2p+1,0)}_{q}(\eta_2),  
	\end{align}
i.e., the same relationship holds except that $g_3$ is set to zero. Hence, we may proceed exactly
as in in the previous case of $p \ge 1$ and $q \ge 1$ to arrive at the result. 

{\bf Proof of (\ref{eq:properties-of-U-pq-4}):}
	For the estimate (\ref{eq:properties-of-U-pq-4}), we use (\ref{eq:foo1})
  with $\eta_3=-1$. Noting Lemma~\ref{lemma:duffy-properties-2} we have
	\begin{align*}
		\|u\|_{L^2(\Gamma)}^2 
		= \int_{\S^2} \left| \tilde{u}(\eta_1,\eta_2,-1) \right|^2 \left(\frac{1-\eta_2}{2}\right) d\eta_1 d\eta_2 
= \sum_{p,q=0}^\infty \frac{1}{\gamma_p^{(0,0)}} \frac{2^{2p+1}}{\gamma_q^{(2p+1,0)}} |U_{p,q}(-1)|^2. 
	\end{align*}
\end{proof}

The estimates for the functions $U_{p,q}$ imply corresponding bounds for the functions 
$\widetilde U_{p,q}$: 
\begin{lemma}[properties of $\widetilde{U}_{p,q}$]
	\label{lemma:properties-of-tilde-U-pq}
	Let $u\in H^1(T)$ and let $\widetilde{U}_{p,q}$ be defined in Definition~\ref{def:U-pq-and-tilde-U-pq}. 
	Then there exists a constant $C > 0$ independent of $u$ such that
	\begin{align}
		\label{eq:properties-of-tilde-U-pq-1}
		\frac{1}{4}\sum_{p,q=0}^\infty \frac{1}{\gamma^{(0,0)}_p} \frac{2^{2p+1}}{\gamma^{(2p+1,0)}_q}
		\int_{-1}^1 (1-\eta_3)^{2p+2q+2} \left|\widetilde{U}_{p,q}(\eta_3)\right|^2 d\eta_3 
		&= \|u\|^2_{L^2(\T^3)}, \\
		\label{eq:properties-of-tilde-U-pq-2}
		\frac{1}{4}\sum_{p,q=0}^\infty \frac{1}{\gamma^{(0,0)}_p} \frac{2^{2p+1}}{\gamma^{(2p+1,0)}_q}
		\int_{-1}^1 (1-\eta_3)^{2p+2q+2} \left|\widetilde{U}_{p,q}^\prime(\eta_3)\right|^2 d\eta_3  
		&\leq C \|\nabla u\|^2_{L^2(\T^3)}.
	\end{align}
\end{lemma}

\begin{proof}
	We have 
		$\displaystyle \widetilde{U}_{p,q}(\eta_3) = (1-\eta_3)^{-(p+q)} U_{p,q}(\eta_3)$ and therefore 
		
$$\displaystyle \widetilde{U}_{p,q}^\prime(\eta_3) = (1-\eta_3)^{-(p+q)} U_{p,q}^\prime(\eta_3) + (p+q)(1-\eta_3)^{-(p+q+1)} U_{p,q}(\eta_3).
$$
	Hence,
	\begin{align*}
	& \frac{1}{4} \int_{-1}^1 (1-\eta_3)^{2p+2q+2} \left| \widetilde{U}_{p,q}(\eta_3)\right|^2 d\eta_3
	= \int_{-1}^1 \left(\frac{1-\eta_3}{2}\right)^2 \left|U_{p,q}(\eta_3)\right|^2 d\eta_3, \\
	& \frac{1}{4} \int_{-1}^1 (1-\eta_3)^{2p+2q+2} \left| \widetilde{U}_{p,q}^\prime(\eta_3)\right|^2 d\eta_3 
\lesssim \\
	&\qquad \int_{-1}^1 \left(\frac{1-\eta_3}{2}\right)^2 \left|U_{p,q}^\prime(\eta_3)\right|^2 d\eta_3 
	 + (p+q)^2 \int_{-1}^1 \left|U_{p,q}(\eta_3)\right|^2 d\eta_3.
	\end{align*}
	Using the results of Lemma~\ref{lemma:properties-of-U-pq} concludes the argument.
\end{proof}
These results allow us to get bounds for weighted sums of the coefficients $\tilde u_{p,q,r}$ and 
$\tilde u^\prime_{p,q,r}$ given in Definition~\ref{def:U-pq-and-tilde-U-pq}: 
\begin{corollary}
	\label{cor:sum-tilde-u-pqr-and-tilde-u-pqr-prime-lesssim-norm}
	Assume the hypotheses of Lemma~\ref{lemma:properties-of-tilde-U-pq} and let $\tilde u_{p,q,r}$, $\tilde u^\prime_{p,q,r}$ be given by Definition~\ref{def:U-pq-and-tilde-U-pq}. 
	Then there exist constants $C$ independent of $u$ such that
	\begin{align}
		\label{eq:sum-tilde-u-pqr-lesssim-norm}
		\sum_{p,q,r = 0}^\infty \frac{1}{\gamma^{(0,0)}_p} \frac{2^{2p+1}}{\gamma^{(2p+1,0)}_q} \frac{1}{\gamma^{(2p+2q+2,0)}_r} |\tilde{u}_{p,q,r}|^2 
		&\leq C \|u\|^2_{L^2(\T^3)},\\
		\label{eq:sum-tilde-u-pqr-prime-lesssim-norm}
		\sum_{p,q,r = 0}^\infty \frac{1}{\gamma^{(0,0)}_p} \frac{2^{2p+1}}{\gamma^{(2p+1,0)}_q} \frac{1}{\gamma^{(2p+2q+2,0)}_r} |\tilde{u}_{p,q,r}^\prime|^2 
		&\leq C \|\nabla u\|^2_{L^2(\T^3)}.
	\end{align}
\end{corollary}

\begin{proof}
	Since the polynomials $P_r^{(2p+2q+2,0)}$ are orthogonal polynomials in a weighted $L^2$-space, expanding $\widetilde{U}_{p,q}$
	yields the representation
$$
		\widetilde{U}_{p,q}(\eta_3) 
		= \sum_{r=0}^\infty \frac{1}{\gamma^{(2p+2q+2,0)}_r} \tilde{u}_{p,q,r} P^{(2p+2q+2,0)}_r(\eta_3)
$$
and therefore 
	\begin{align*}
		\int_{-1}^1 (1-\eta_3)^{2p+2q+2} |\widetilde{U}_{p,q}(\eta_3)|^2 d\eta_3
		= \sum_{r=0}^\infty \frac{1}{\gamma^{(2p+2q+2,0)}_r} |\tilde{u}_{p,q,r}|^2.
	\end{align*}
	The statement (\ref{eq:sum-tilde-u-pqr-lesssim-norm}) now follows directly from (\ref{eq:properties-of-tilde-U-pq-1})
	of Lemma \ref{lemma:properties-of-tilde-U-pq}.
	Analogously, we deal with (\ref{eq:sum-tilde-u-pqr-prime-lesssim-norm}), where we expand $\widetilde{U}^\prime_{p,q}$ and 
	conclude with (\ref{eq:properties-of-tilde-U-pq-2}) of Lemma~\ref{lemma:properties-of-tilde-U-pq}. 
\end{proof}

\subsection{Connections between $\tilde{u}_{p,q,r}$ and $\tilde{u}^\prime_{p,q,r}$} 

As we have mentioned in \ref{Point:4}  of Section~\ref{sec:outline},  
the multiplicative structure of the estimate of Theorem~\ref{thm:trace-stability} is based on 
a relation between the coefficients $\tilde{u}_{p,q,r}$ and $\tilde{u}^\prime_{p,q,r}$. This 
connection is essentially a one-dimensional effect and follows from Lemma~\ref{lemma:connection-U-and-Uprime}: 

\begin{corollary}
	\label{cor:connection-tilde-u-pqr-and-tilde-u-pqr-prime}
	Let $\tilde{u}_{p,q,r}$ and $\tilde{u}_{p,q,r}^\prime$ be as in Definition~\ref{def:U-pq-and-tilde-U-pq}
	and $h_i$, $i \in \{1,2,3\}$ be given by (\ref{def:hi-gi}). 
	Then for $r \geq 1$ and $p$, $q \geq 0$ there holds 
	\begin{align*}
		\tilde{u}_{p,q,r} = h_1(r,2p+2q+2) \tilde{u}^\prime_{p,q,r+1} 
		+ h_2(r,2p+2q+2) \tilde{u}^\prime_{p,q,r} 
		+ h_3(r,2p+2q+2) \tilde{u}^\prime_{p,q,r-1}. 
	\end{align*}
\end{corollary}

\begin{proof}
	By density we may assume that $u\in C^\infty(\overline{\T^3})$. Lemma~\ref{lemma:connection-U-and-Uprime}
	then implies the result. 
\iftechreport 
For a detailed version of this proof see Appendix~\ref{app:B}.
\fi
\end{proof}

\section{Trace stability (Proof of Theorem~\ref{thm:trace-stability})}
\label{sec:trace-stability}
 
We start with establishing a representation of the transformed function 
$\tilde{u} = u \circ D$ on the face $\Gamma:=\T^2 \times \{-1\}$.  
\iftechreport 
  Using (\ref{eq:jacobi-polynomial-symmetry}) and (\ref{eq:jacobi-polynomial-at-1}) we get
\fi 
\begin{align*}
	P_r^{(2p+2q+2,0)}(-1) = (-1)^r P_r^{(0,2p+2q+2)}(1) = (-1)^r.
\end{align*}
Therefore we have for $\xi\in\Gamma$
\begin{align*}
	\psi_{p,q,r}(\xi) = \tilde{\psi}_{p,q,r}\left(D^{-1}(\xi_1,\xi_2,-1)\right) = (-1)^r \left(\frac{1-\eta_2}{2}\right)^p P_p^{(0,0)}(\eta_1) P_q^{(2p+1,0)}(\eta_2).
\end{align*}
Applying the expansion of $u$ in (\ref{eq:expansion-u}) we arrive at
\begin{align*}
	& \tilde{u}(\eta_1,\eta_2,-1) = u(\xi_1,\xi_2,-1) \\
	&\quad = \sum_{p,q,r=0}^{\infty} \frac{1}{\gamma_p^{(0,0)}}
	\frac{2^{2p+1}}{\gamma_q^{(2p+1,0)}}P_p^{(0,0)}(\eta_1)\left(\frac{1-\eta_2}{2}\right)^p P_q^{(2p+1,0)}(\eta_2)
	(-1)^r u_{p,q,r} \frac{2^{2p+2q+2}}{\gamma_r^{(2p+2q+2,0)}}.  
\end{align*}
In view of Lemma~\ref{lemma:duffy-properties-2} as well as the expansion (\ref{eq:expansion-u}), 
we obtain for the $L^2(\Gamma)$-norm
\begin{align}
\notag 
	\|u\|_{L^2(\Gamma)}^2 &= \int_{\S^2} \left| \tilde u(\eta_1,\eta_2,-1) \right|^2 \left(\frac{1-\eta_2}{2}\right)
\,d\eta_1\,d\eta_2  \\
	\label{eq:trace-norm}
	& =\sum_{p,q=0}^{\infty} \frac{1}{\gamma_p^{(0,0)}} \frac{2^{2p+1}}{\gamma_q^{(2p+1,0)}}
	\left|\sum_{r=0}^{\infty}(-1)^r u_{p,q,r} \frac{2^{2p+2q+2}}{\gamma_r^{(2p+2q+2,0)}}\right|^2. 
\end{align}
Before proceeding further, we recall \ref{Point:4} of Section~\ref{sec:outline},
where we studied $(u - \Pi_N u)(1)$ in equation (\ref{eq:1D-u-PiNu-at-1}). 
There, we observed that the infinite sum $\sum_{q = N+1}^\infty \hat u_q$ 
reduced to a short sum with merely two terms. A very similar effect 
takes place here in the multi-dimensional case in that the 
infinite sum over $r$ in (\ref{eq:trace-norm}) can be expressed as 
a finite sum: 
\begin{lemma} 
	\label{lemma:finite-sum}
	Let $N \ge 1$. Then there holds
	\begin{align*}
		\sum_{r=N}^\infty (-1)^r \frac{2^{2p+2q+2}}{\gamma_r^{(2p+2q+2,0)}} u_{p,q,r} 
		&= (-1)^N h_2(N,2p+2q+2) \frac{2^{p+q}}{\gamma_N^{(2p+2q+2,0)}} \tilde{u}_{p,q,N}^{\prime} \\ 
		&\qquad + \sum_{r=N-1}^N (-1)^{r+1} h_3(r+1,2p+2q+2) \frac{2^{p+q}}{\gamma_{r+1}^{(2p+2q+2,0)}} \tilde{u}_{p,q,r}^{\prime}. 
	\end{align*}
\end{lemma}

\begin{proof}
	We abbreviate $n_{pq}:=2p+2q+2$. In view of Corollary~\ref{cor:connection-tilde-u-pqr-and-tilde-u-pqr-prime} we have
	\begin{align*}
		&\sum_{r=N}^\infty (-1)^r \frac{2^{n_{pq}}}{\gamma_r^{(n_{pq},0)}} u_{p,q,r}
		\overset{(\ref{eq:one-dimensional-form})}{=} \sum_{r=N}^\infty (-1)^r \frac{2^{p+q}}{\gamma_r^{(n_{pq},0)}} \tilde{u}_{p,q,r} \\
		&\qquad = \sum_{r=N}^\infty (-1)^r \frac{2^{p+q}}{\gamma_r^{(n_{pq},0)}}
		\Big( h_1(r,n_{pq}) \tilde u_{p,q,r+1}^\prime + h_2(r,n_{pq}) \tilde u_{p,q,r}^\prime 
		+ h_3(r,n_{pq})\tilde u_{p,q,r-1}^\prime \Big) \\
		&\qquad = \sum_{r=N+1}^\infty (-1)^{r-1} h_1(r-1,n_{pq}) \frac{2^{p+q}}{\gamma_{r-1}^{(n_{pq},0)}} \tilde{u}_{p,q,r}^\prime \\
		&\qquad\qquad + \sum_{r=N}^\infty (-1)^{r} h_2(r,n_{pq}) \frac{2^{p+q}}{\gamma_{r}^{(n_{pq},0)}} \tilde{u}_{p,q,r}^\prime
		+ \sum_{r=N-1}^\infty (-1)^{r+1} h_3(r+1,n_{pq}) \frac{2^{p+q}}{\gamma_{r+1}^{(n_{pq},0)}} \tilde{u}_{p,q,r}^\prime \\
		&\qquad = \sum_{r=N+1}^\infty (-1)^r \tilde u_{p,q,r}^\prime 2^{p+q} 
		\left[ - \frac{h_1(r-1,n_{pq})}{\gamma_{r-1}^{(n_{pq},0)}} + \frac{h_2(r,n_{pq})}{\gamma_r^{(n_{pq},0)}}
		          - \frac{h_3(r+1,n_{pq})}{\gamma_{r+1}^{(n_{pq},0)}} \right] \\
		&\qquad\qquad + (-1)^N h_2(N,n_{pq}) \frac{2^{p+q}}{\gamma_N^{(n_{pq},0)}} \tilde u_{p,q,N}^\prime 
		+ \sum_{r=N-1}^N (-1)^{r+1} h_3(r+1,n_{pq}) \frac{2^{p+q}}{\gamma_{r+1}^{(n_{pq},0)}} \tilde u_{p,q,r}^\prime.
	\end{align*}
	By (\ref{eq:magic-cancellation}), the expression in brackets vanishes and that concludes the proof. 
\end{proof}

\noindent Since Lemma~\ref{lemma:finite-sum} assumes $N \ge 1$, 
the terms corresponding to $r = 0$ in the norm in (\ref{eq:trace-norm}) are not 
included. We now study this special case in the following Lemma~\ref{lemma:case-q=0}:  
\begin{lemma}
	\label{lemma:case-q=0}
	Let $u\in H^1(\T^3)$. Then 
there exists a constant $C > 0$ independent of $p$, $q$, and $u$ such that
	\begin{align*}
		\sum_{p,q = 0}^\infty \frac{1}{\gamma_p^{(0,0)}} \frac{2^{2p+1}}{\gamma_q^{(2p+1,0)}} \left| u_{p,q,0} \frac{2^{2p+2q+2}}{\gamma_0^{(2p+2q+2,0)}}\right|^2 
		\leq C \|u\|_{L^2(\T^3)} \|u\|_{H^1(\T^3)}.
	\end{align*} 
\end{lemma}

\begin{proof}
        We first note that the coefficient $|u_{0,0,0}| \leq \sqrt{|\T^3|} \|u\|_{L^2(\T^3)}$ so that 
        we may focus on the sum with $(p,q) \ne (0,0)$. 
	Since 
	$\gamma_0^{(2p+2q+2,0)} = \frac{2^{2p+2q+3}}{2p+2q+3}$, we get 
	\begin{align*}
		\sum_{p,q=0}^\infty \frac{1}{\gamma_p^{(0,0)}} \frac{2^{2p+1}}{\gamma_q^{(2p+1,0)}} \left| u_{p,q,0} \frac{2^{2p+2q+2}}{\gamma_0^{(2p+2q+2,0)}}\right|^2 
		\lesssim \sum_{p,q=0}^\infty \frac{1}{\gamma_p^{(0,0)}} \frac{2^{2p+1}}{\gamma_q^{(2p+1,0)}} 
                (p+q+1)^2 \left| u_{p,q,0} \right|^2. 
	\end{align*}
	To bound the sum on the right-hand side, we note that an integration by parts 
        and (\ref{eq:Upq(1) = 0}) give (for $(p,q) \ne (0,0)$)
	\begin{align}
        \label{eq:first-estimate}
		2^{p+q+2} u_{p,q,0} 
		&= \int_{-1}^1 (1-\eta_3)^{p+q+2} U_{p,q}(\eta_3) d\eta_3 \\
        \label{eq:second-estimate}
		&= \frac{1}{p+q+3} \left( 2^{p+q+3} U_{p,q}(-1) 
		+ \int_{-1}^1 (1 - \eta_3)^{p+q+3} U_{p,q}^\prime(\eta_3) d\eta_3 \right), 
	\end{align}
        These two equations (\ref{eq:first-estimate}), (\ref{eq:second-estimate}) yield two estimates 
        for $u_{p,q,0}$. From (\ref{eq:first-estimate}), 
        we get with the Cauchy-Schwarz inequality
	\begin{align}
		|2^{p+q+2} u_{p,q,0}|^2 
		&= \left| \int_{-1}^1 (1-\eta_3)^{p+q+2} U_{p,q}(\eta_3) d\eta_3 \right|^2 \notag \\
		&\leq \left( \int_{-1}^1 \left((1-\eta_3)^{p+q+1}\right)^2 d\eta_3 \right) \left( \int_{-1}^1 (1-\eta_3)^2 |U_{p,q}(\eta_3)|^2 d\eta_3 \right) \notag \\
		\label{eq:case-q=0-lemma-1}
		&= \frac{2^{2p+2q+5}}{2p+2q+3} \int_{-1}^1 \left(\frac{1-\eta_3}{2}\right)^2 |U_{p,q}(\eta_3)|^2 d\eta_3. 
	\end{align}
	From (\ref{eq:second-estimate}), we obtain a second estimate as follows: 
	\begin{align*}
		|2^{p+q+2} u_{p,q,0}|^2 
		\leq \frac{2}{(p+q+3)^2} \left( 2^{2p+2q+6} |U_{p,q}(-1)|^2 + \left|\int_{-1}^1 (1 - \eta_3)^{p+q+3} U_{p,q}^\prime(\eta_3) d\eta_3\right|^2 \right),
	\end{align*}
	where
	\begin{align*}
		\left|\int_{-1}^1 (1 - \eta_3)^{p+q+3} U_{p,q}^\prime(\eta_3) d\eta_3\right|^2
		&\!\!\!\leq \left(\int_{-1}^1\!\! \left((1 - \eta_3)^{p+q+2}\right)^2 d\eta_3\!\right)
		\!\left(\int_{-1}^1\!\! (1 - \eta_3)^2 |U_{p,q}^\prime(\eta_3)|^2 d\eta_3\right) \\
		&= \frac{2^{2p+2q+7}}{2(p+q)+5} \int_{-1}^1 
                 \left(\frac{1 - \eta_3}{2}\right)^2 |U_{p,q}^\prime(\eta_3)|^2.
	\end{align*}
	Inserting this in the bound before yields
	\begin{align}
		|2^{p+q+2} u_{p,q,0}|^2
		\label{eq:case-q=0-lemma-2}
		&\leq 2\frac{2^{2p+2q+6}}{(p+q+3)^2} \left( |U_{p,q}(-1)|^2 + \frac{1}{p+q+2} \int_{-1}^1 \left(\frac{1 - \eta_3}{2}\right)^2 |U_{p,q}^\prime(\eta_3)|^2 d\eta_3 \right).
	\end{align}
	Next, we abbreviate
	\begin{align*}
		\sigma_{p,q}^2 := \int_{-1}^1 \left(\frac{1-\eta_3}{2}\right)^2 |U_{p,q}(\eta_3)|^2 d\eta_3, \qquad 
		\tau_{p,q}^2 &:= \int_{-1}^1 \left(\frac{1-\eta_3}{2}\right)^2 |U_{p,q}^\prime(\eta_3)|^2 d\eta_3.
	\end{align*}
	Hence, applying (\ref{eq:case-q=0-lemma-1}) and (\ref{eq:case-q=0-lemma-2}) we have 
        in view of the elementary observation $\min\{a^2,b^2 + c^2\} \leq b^2 + \min\{a^2,c^2\} 
\leq b^2 + |a|\,|c|$ (for real $a$, $b$, $c$): 
	\begin{align*}
		|2^{p+q+2} u_{p,q,0}|^2 
		&\leq \min \left\{ \frac{2^{2p+2q+5}}{p+q+3} \sigma_{p,q}^2 \,\, ,\, 2\frac{2^{2p+2q+6}}{(p+q+3)^2}\left( |U_{p,q}(-1)|^2 + \frac{1}{p+q+2} \tau_{p,q}^2\right) \right\} \\
		&\leq \frac{2^{2p+2q+5}}{(p+q+3)} \left( \frac{4}{p+q+3} |U_{p,q}(-1)|^2 + \frac{2}{p+q+2} \sigma_{p,q} \tau_{p,q}\right). 
	\end{align*}
        This leads us to
	\begin{align}
		\label{eq:case-q=0-lemma-3}
		|u_{p,q,0}|^2 \lesssim \frac{1}{(p+q+3)^2} |U_{p,q}(-1)|^2 + \frac{1}{(p+q+2)^2} \sigma_{p,q} \tau_{p,q}.
	\end{align}
	Hence, we conclude
	\begin{align*}
		 & \sum_{p,q=0}^\infty \frac{1}{\gamma_p^{(0,0)}} \frac{2^{2p+1}}{\gamma_q^{(2p+1,0)}} (p+q+1)^2 |u_{p,q,0}|^2 \\
		&\quad \lesssim \sum_{p,q=0}^\infty \frac{1}{\gamma_p^{(0,0)}} \frac{2^{2p+1}}{\gamma_q^{(2p+1,0)}} |U_{p,q}(-1)|^2 + \sum_{p,q=0}^\infty \frac{1}{\gamma_p^{(0,0)}} \frac{2^{2p+1}}{\gamma_q^{(2p+1,0)}} \sigma_{p,q} \tau_{p,q} \\
		& \quad \lesssim \|u\|^2_{L^2(\Gamma)} + \|u\|_{L^2(\T^3)} \|\nabla u\|_{L^2(\T^3)},
	\end{align*}
	where, in the last inequality, we used Cauchy-Schwarz for the sum involving $\sigma_{pq}$ and $\tau_{pq}$ 
        and appealed to Lemma~\ref{lemma:properties-of-U-pq}.
	The proof of the lemma is now completed with the aid of the multiplicative trace inequality 
        $\|u\|^2_{L^2(\Gamma)} \lesssim \|u\|_{L^2(\T^3)} \|u\|_{H^1(\T^3)}$ (see, e.g., 
        \cite[Thm.~{1.6.6}]{brenner-scott94}). 
\end{proof}

We conclude this section with the proof of 
Theorem~\ref{thm:trace-stability} and Corollary~\ref{cor:trace-approximation}. 

\begin{numberedproof}{of \protect{Theorem~\ref{thm:trace-stability} and Corollary~\ref{cor:trace-approximation}}}
	In view of the multiplicative trace inequality 
        $\|u\|_{L^2(\Gamma)}^2 \lesssim \|u\|_{L^2(\T^3)} \|u\|_{H^1(\T^3)}$ 
        (see, e.g., \cite[Thm.~{1.6.6}]{brenner-scott94})
	we will only show the statement $\|u-\Pi_N u\|_{L^2(\Gamma)}^2 \lesssim \|u\|_{L^2(\T^3)} \|u\|_{H^1(\T^3)}$. 
	We abbreviate $n_{pq}:=2p+2q+2$ and $c_{pq}:=\frac{1}{\gamma_p^{(0,0)}} \frac{2^{2p+1}}{\gamma_q^{(2p+1,0)}}$. 
	By (\ref{eq:trace-norm}), we have to bound 
{\allowdisplaybreaks
	\begin{align*}
		&\|u - \Pi_N u\|_{L^2(\Gamma)}^2
		= \sum_{p,q=0}^\infty c_{pq}
		\left| \sum_{r=\max\{0,N+1-p-q\}}^\infty  (-1)^r \frac{2^{n_{pq}}}{\gamma_r^{(n_{pq},0)}} u_{p,q,r}\right|^2 \\
		& = \sum_{p+q\leq N} c_{pq}
		\left| \sum_{r=N+1-p-q}^\infty  (-1)^r \frac{2^{n_{pq}}}{\gamma_r^{(n_{pq},0)}} u_{p,q,r}\right|^2 
		+ \sum_{p+q\geq N+1} c_{pq}
		\left| \sum_{r=0}^\infty  (-1)^r \frac{2^{n_{pq}}}{\gamma_r^{(n_{pq},0)}} u_{p,q,r}\right|^2 \\
		& = \sum_{p=0}^N \sum_{q=0}^{N-p} c_{pq}
		\left| \sum_{r=N+1-p-q}^\infty  (-1)^r \frac{2^{n_{pq}}}{\gamma_r^{(n_{pq},0)}} u_{p,q,r}\right|^2 \\
		& \quad \mbox{} + \sum_{p=0}^N \sum_{q=N+1-p}^\infty c_{pq}
		\left| \sum_{r=0}^\infty  (-1)^r \frac{2^{n_{pq}}}{\gamma_r^{(n_{pq},0)}} u_{p,q,r}\right|^2 
		 + \sum_{p=N+1}^\infty \sum_{q=0}^\infty c_{pq}
		\left| \sum_{r=0}^\infty  (-1)^r \frac{2^{n_{pq}}}{\gamma_r^{(n_{pq},0)}} u_{p,q,r}\right|^2 \\ 
		& \lesssim \underbrace{\sum_{p=0}^N \sum_{q=0}^{N-p} c_{pq} 
		\left| \sum_{r=N+1-p-q}^\infty  (-1)^r \frac{2^{n_{pq}}}{\gamma_r^{(n_{pq},0)}} u_{p,q,r}\right|^2}_\text{$=:S_1$}
		\!\! + \underbrace{\sum_{p=0}^N \sum_{q=N+1-p}^\infty c_{pq} 
		\left| \sum_{r=1}^\infty  (-1)^r \frac{2^{n_{pq}}}{\gamma_r^{(n_{pq},0)}} u_{p,q,r}\right|^2}_\text{$=:S_2$} \\
		&\quad\qquad + \underbrace{\sum_{p=N+1}^\infty \sum_{q=0}^\infty c_{pq} 
		\left| \sum_{r=1}^\infty  (-1)^r \frac{2^{n_{pq}}}{\gamma_r^{(n_{pq},0)}} u_{p,q,r}\right|^2}_\text{$=:S_3$}
		+ \underbrace{\sum_{p=0}^N \sum_{q=N+1-p}^\infty c_{pq} 
		\left| \frac{2^{n_{pq}}}{\gamma_0^{(n_{pq},0)}} u_{p,q,0}\right|^2}_\text{$=:S_4$} \\
		&\quad\qquad + \underbrace{\sum_{p=N+1}^\infty \sum_{q=0}^\infty c_{pq} 
		\left| \frac{2^{n_{pq}}}{\gamma_0^{(n_{pq},0)}} u_{p,q,0}\right|^2}_\text{$=:S_5$}.
	\end{align*}
}
	Lemma~\ref{lemma:case-q=0} immediately gives 
        $S_4 + S_5 \lesssim \|u\|_{L^2(\T^3)} \|u\|_{H^1(\T^3)}$. Noting the estimates 
	\begin{align*}
		& h_2(N+1-p-q, n_{pq}) \lesssim \frac{p+q}{N^2} \lesssim \frac{1}{N}, \qquad p+q=0,\ldots,N, \\
		& h_3(N+1-p-q, n_{pq}),~h_3(N+2-p-q, n_{pq}) \lesssim \frac{N+p+q}{N^2} \lesssim \frac{1}{N}, \qquad p+q=0,\ldots,N,  \\
		& \frac{1}{\gamma_{N+1-p-q}^{(n_{pq},0)}} = \frac{2N+5}{2^{2p+2q+3}}, \qquad
		\frac{1}{\gamma_{N+2-p-q}^{(n_{pq},0)}} = \frac{2N+7}{2^{2p+2q+3}}, 
	\end{align*}
	we obtain for $S_1$ from Lemma~\ref{lemma:finite-sum} 
	\begin{align*}
		S_1 \lesssim \sum_{p=0}^N \sum_{q=0}^{N-p} c_{pq} 
		\left( 
		\left| 2^{-(p+q+3)} \tilde u_{p,q,N+1-p-q}^\prime\right|^2 + 
		\left| 2^{-(p+q+3)} \tilde u_{p,q,N-p-q}^\prime\right|^2 
		\right).
	\end{align*}
	Analogously, we get for $S_2$ and $S_3$
	\begin{align*}
		S_2 &\lesssim
		\sum_{p=0}^N \sum_{q=N+1-p}^\infty c_{pq} 
		\left( 
		\left| 2^{-(p+q+3)} \tilde u_{p,q,1}^\prime\right|^2 + 
		\left| 2^{-(p+q+3)} \tilde u_{p,q,0}^\prime\right|^2 		
		\right), \\
		S_3 &\lesssim
		\sum_{p=N+1}^\infty \sum_{q=0}^\infty c_{pq} 
		\left( 
		\left| 2^{-(p+q+3)} \tilde u_{p,q,1}^\prime\right|^2 + 
		\left| 2^{-(p+q+3)} \tilde u_{p,q,0}^\prime\right|^2 		
		\right).
	\end{align*}
	\noindent Finally, we may apply Lemma~\ref{lemma:estimate-of-bq-in-terms-of-product-of-sums} 
        to get for $S_1+S_2+S_3$ 
	\begin{align*}
		S_1 + S_2 + S_3 &\lesssim 
		\sum_{p=0}^N \sum_{q=0}^{N-p} c_{pq}\!\!
		\left( \sum_{r \ge N+1-p-q} \frac{1}{\gamma_r^{(n_{pq},0)}} |\tilde u_{p,q,r}|^2 \right)^{1/2} \!\!\!
		\left( \sum_{r \ge N-p-q} \frac{1}{\gamma_r^{(n_{pq},0)}} |\tilde u_{p,q,r}^\prime|^2 \right)^{1/2} \\
		&\qquad + 
		\sum_{p=0}^N \sum_{q=N+1-p}^\infty c_{pq} 
		\left( \sum_{r \ge 1} \frac{1}{\gamma_r^{(n_{pq},0)}} |\tilde u_{p,q,r}|^2 \right)^{1/2}
		\left( \sum_{r \ge 0} \frac{1}{\gamma_r^{(n_{pq},0)}} |\tilde u_{p,q,r}^\prime|^2 \right)^{1/2} \\
		&\qquad + 
		\sum_{p=N+1}^\infty \sum_{q=0}^\infty c_{pq} 
		\left( \sum_{r \ge 1} \frac{1}{\gamma_r^{(n_{pq},0)}} |\tilde u_{p,q,r}|^2 \right)^{1/2}
		\left( \sum_{r \ge 0} \frac{1}{\gamma_r^{(n_{pq},0)}} |\tilde u_{p,q,r}^\prime|^2 \right)^{1/2} \\
		&\lesssim \|u\|_{L^2(\T^3)} \|\nabla u\|_{L^2(\T^3)},
	\end{align*}
	where in the last estimate, we have used the Cauchy-Schwarz inequality for sums and 
	Corollary \ref{cor:sum-tilde-u-pqr-and-tilde-u-pqr-prime-lesssim-norm}. Since 
	$\|\nabla u\|_{L^2(\T^3)} \leq \|u\|_{H^1(\T^3)}$ this concludes the proof of 
        (\ref{thm:trace-stability-1}). 

        \noindent The estimate 
        (\ref{thm:trace-stability-2}) follows directly from 
        (\ref{thm:trace-stability-1}) in view of \cite[Lem.~{25.3}]{tartar07}.

\noindent Finally,
        (\ref{eq:cor:trace-approximation-1}) is obtained in 
a fairly routine way from the stability of $\Pi_N$ just shown. Specifically,  
for arbitrary $v \in {\mathcal P}_N$ we have 
by the projection property of $\Pi_N$ as well as the continuity
of the trace operator 
$\gamma_0:B^{1/2}_{1/2,1}(\T^3) \rightarrow L^2(\partial\T^3)$
(cf., e.g., \cite[Sec.~{32}]{tartar07}) 
\begin{eqnarray*}
\|u - \Pi_N u\|_{L^2(\partial \T^3)} &\leq & 
\|u - v\|_{L^2(\partial\T^3)}+ 
\|\Pi_N( u - v)\|_{L^2(\partial\T^3)} 
\leq C \|u - v\|_{B^{1/2}_{2,1}(\T^3)}. 
\end{eqnarray*}
Hence, 
$\displaystyle \|u - \Pi_N u\|_{L^2(\partial \T^3)} 
\leq C \inf_{v \in {\mathcal P}_N} \|u - v\|_{B^{1/2}_{2,1}(\T^3)}$. 
Fix $s > 1/2$. Let $\widetilde \Pi_N:L^2(\T^3) \rightarrow {\mathcal P}_N$ be an 
approximation operator with simultaneous approximation properties 
in a scale of Sobolev spaces, viz., 
$$
\|u - \widetilde \Pi_N u\|_{L^2(\T^3)} \leq C N^{-s} \|u\|_{H^s(\T^3)}, 
\qquad 
\|u - \widetilde \Pi_N u\|_{H^s(\T^3)} \leq C  \|u\|_{H^s(\T^3)} 
\quad \forall u \in H^s(\T^3). 
$$
(This can be achieved, for example, by combining the approximation 
results of \cite[Appendix~A]{melenk05} for hyper cubes with the 
well-known extension operator of Stein, \cite[Chap.~{VII}]{stein70}.)
The reiteration theorem (cf., e.g., \cite[Thm.~{26.3}]{tartar07}) gives 
$B^{1/2}_{2,1}(\T^3) = (L^2(\T^3),H^s(\T^3))_{(s-1/2)/s,1}$. Also, we have 
$H^{s}(\T^3) = (H^s(\T^3),H^s(\T^3))_{(s-1/2)/s,1}.
$
By interpolation theory, we then have that 
$\operatorname*{Id} - \widetilde \Pi_N$ is a bounded linear operator 
$$
(H^s(\T^3),H^s(\T^3))_{(s-1/2)/s,1} = H^s(\T^3) \rightarrow 
B^{1/2}_{2,1}(\T^3) = (L^2(\T^3),H^s(\T^3))_{(s-1/2)/s,1}
$$
with  norm 
$\|\operatorname*{Id} - \widetilde \Pi_N \|_{B^{1/2}_{2,1}(\T^3) \leftarrow H^s(\T^3)} \leq C N^{-s (s-1/2)/s}$.  
\end{numberedproof}

\section{$H^1$-stability (Proof of \protect{Corollary~\ref{cor:H1-stability}})}
\label{sec:H1-stability}
Our procedure to study the $H^1$-stability of the $L^2$-projection $\Pi_N$ is 
to consider on the hyper cube $\S^3$ the derivative $\partial_{\eta_3} \widetilde v_N$ 
of the transformed function $\widetilde v_N = v_N \circ D$, where $v_N:=\Pi_N u$. 
This provides information about a directional derivative of $v_N$ on $\T^3$. 
Through affine transformations of the tetrahedron $\T^3$, information about 
the full gradient of $v_N$ can be inferred. 

The key step is therefore to control 
$\partial_{\eta_3} \widetilde{\Pi}_N u$, where we denote $\widetilde{\Pi}_N u := (\Pi_N u) \circ D$.
This is the purpose of the ensuing lemma. 
\begin{lemma}
	\label{lemma:stability-of-eta3-derivatives}
	There exists a constant $C > 0$ independent of $N$ such that 
	\begin{align*}
		\int_{\S^3} (1-\eta_2)(1-\eta_3)^2 \left|\partial_{\eta_3} \widetilde{\Pi}_N u(\eta)\right|^2 d\eta \leq C N \|\nabla u\|_{L^2(\T^3)}^2
		\qquad \forall u\in H^1(\T^3).
	\end{align*}
\end{lemma}

\begin{proof}
\iftechreport 
        Given that $\Pi_N$ reproduces constant functions, a Poincar\'e inequality allows us
        to reduce the problem to showing the weaker estimate 
	\begin{align}
\label{lemma:stability-of-eta3-derivatives-foo}
		\int_{\S^3} (1-\eta_2)(1-\eta_3)^2 \left|\partial_{\eta_3} \widetilde{\Pi}_N u(\eta)\right|^2 d\eta \leq C N \|u\|_{H^1(\T^3)}^2
		\qquad \forall u\in H^1(\T^3).
	\end{align}
	We abbreviate $n_{pq}:=2p+2q+2$. We see that 
	\begin{align*}
		\widetilde \Pi_N u(\eta) 
		&= \sum_{p=0}^N \sum_{q=0}^{N-p} \sum_{r=0}^{N-p-q} 
		\frac{1}{\gamma_p^{(0,0)}} \frac{2^{2p+1}}{\gamma_q^{(2p+1,0)}} \frac{2^{p+q}}{\gamma_r^{(n_{pq},0)}} 
		\tilde{u}_{p,q,r} \tilde{\psi}_{p,q,r}(\eta)
	\end{align*}
	by recalling the relation between $u_{p,q,r}$ and $\tilde u_{p,q,r}$. 
	Differentiating with respect to $\eta_3$ shows us that we have to estimate the two terms	
	\begin{align}
		I_1 &:=
		\int_{\S^3} (1-\eta_2)(1-\eta_3)^2 \left|\sum_{p=0}^N \sum_{q=0}^{N-p}\sum_{r=0}^{N-p-q} 
		\frac{1}{\gamma_p^{(0,0)}} \frac{2^{2p+1}}{\gamma_q^{(2p+1,0)}} \frac{1}{\gamma_r^{(n_{pq},0)}}
		\tilde{u}_{p,q,r}\right. \notag \\
		\label{eq:terms-of-eta3-derivative-of-l2-projection-1}
		&\qquad\qquad \left.\vphantom{\sum_{p=0}^N}
		\times P_p^{(0,0)}(\eta_1) P_q^{(2p+1,0)}(\eta_2) \left(\frac{1-\eta_2}{2}\right)^{p} 
		(1-\eta_3)^{p+q} \big(P_r^{(n_{pq},0)}\big)^\prime(\eta_3)\right|^2 d\eta \\
		I_2 &:=
		\int_{\S^3} (1-\eta_2)(1-\eta_3)^2 \left|\sum_{p=0}^N \sum_{q=0}^{N-p}\sum_{r=0}^{N-p-q} 
		\frac{1}{\gamma_p^{(0,0)}} \frac{2^{2p+1}}{\gamma_q^{(2p+1,0)}} \frac{p+q}{\gamma_r^{(n_{pq},0)}}
		\tilde{u}_{p,q,r}\right. \notag \\
		\label{eq:terms-of-eta3-derivative-of-l2-projection-2}
		&\qquad\qquad \left.\vphantom{\sum_{p=0}^N}
		\times P_p^{(0,0)}(\eta_1) P_q^{(2p+1,0)}(\eta_2) \left(\frac{1-\eta_2}{2}\right)^{p} 
		(1-\eta_3)^{p+q-1} P_r^{(n_{pq},0)}(\eta_3)\right|^2 d\eta
	\end{align}
	First, we consider (\ref{eq:terms-of-eta3-derivative-of-l2-projection-1}). 
	From Lemma~\ref{lemma:H1-stability-of-truncation} with $\alpha = n_{pq}$ we get 
	\begin{align*}
		I_1 &= 
		\sum_{p=0}^N \sum_{q=0}^{N-p} \frac{1}{\gamma_p^{(0,0)}} \frac{2^{2p+1}}{\gamma_q^{(2p+1,0)}}
		\int_{-1}^1 (1-\eta_3)^{n_{pq}} \left|\sum_{r=0}^{N-p-q} \frac{1}{\gamma_r^{(n_{pq},0)}} 
		\tilde{u}_{p,q,r} \big(P_r^{(n_{pq},0)}\big)^\prime(\eta_3)\right|^2 d\eta_3 \\
		&\lesssim \sum_{p=0}^N \sum_{q=0}^{N-p} \frac{1}{\gamma_p^{(0,0)}} \frac{2^{2p+1}}{\gamma_q^{(2p+1,0)}} 
		(N-p-q) \sum_{r=0}^\infty \frac{1}{\gamma_r^{(n_{pq},0)}} |\tilde{u}_{p,q,r}^\prime |^2
		\lesssim N \|\nabla u\|^2_{L^2(\T^3)},
	\end{align*}
	where in the last step, we appealed to Corollary~\ref{cor:sum-tilde-u-pqr-and-tilde-u-pqr-prime-lesssim-norm}.
	Thus, we arrive at the desired bound for $I_1$.
	Next, we consider (\ref{eq:terms-of-eta3-derivative-of-l2-projection-2}). 
	We have
	\begin{align*}
		I_2 &= \sum_{p=0}^N \sum_{q=0}^{N-p} \frac{(p+q)^2}{\gamma_p^{(0,0)}} \frac{2^{2p+1}}{\gamma_q^{(2p+1,0)}} 
		\int_{-1}^1 (1-\eta_3)^{2p+2q} \left|\sum_{r=0}^{N-p-q} \frac{1}{\gamma_r^{(n_{pq},0)}}
		\tilde{u}_{p,q,r} P_r^{(n_{pq},0)}(\eta_3)\right|^2 d\eta_3
	\end{align*}
	Lemma~\ref{lemma:increase-weight} with $\beta=2p+2q$ and the normalization convention for Jacobi polynomials 
        (\ref{eq:jacobi-polynomial-at-1}) 
        now yield 
{\allowdisplaybreaks
	\begin{align*}
		& I_2 \leq \sum_{p=0}^N \sum_{q=0}^{N-p} \frac{(p+q)^2}{\gamma_p^{(0,0)}} \frac{2^{2p+1}}{\gamma_q^{(2p+1,0)}} 
		\left( \frac{1}{(p+q)^2} 
		\int_{-1}^1 (1 - \eta_3)^{n_{pq}} \left|\sum_{r=0}^{N-p-q} \frac{1}{\gamma_r^{(n_{pq},0)}}
		\tilde{u}_{p,q,r} \big(P_r^{(n_{pq},0)}\big)^\prime(\eta_3)\right|^2 d\eta_3 \right.\\
		&\qquad \left. + \frac{1}{p+q} \left| \sum_{r=0}^{N-p-q} \frac{1}{\gamma_r^{(n_{pq},0)}}
		\tilde{u}_{p,q,r} P_r^{(n_{pq},0)}(-1)\right|^2 \right)\\
		&\leq I_1 + \sum_{p=0}^N \sum_{q=0}^{N-p} \frac{p+q}{\gamma_p^{(0,0)}} \frac{2^{2p+1}}{\gamma_q^{(2p+1,0)}}
		\left| \sum_{r=0}^{N-p-q} \frac{1}{\gamma_r^{(n_{pq},0)}} \tilde{u}_{p,q,r} P_r^{(n_{pq},0)}(-1) \right|^2 \\
		&= I_1 + \sum_{p=0}^N \sum_{q=0}^{N-p} \frac{p+q}{\gamma_p^{(0,0)}} \frac{2^{2p+1}}{\gamma_q^{(2p+1,0)}}
		\left| \sum_{r=0}^{N-p-q} \frac{1}{\gamma_r^{(n_{pq},0)}} \tilde{u}_{p,q,r} (-1)^r \right|^2 \\
		&\leq I_1 +  2 N \sum_{p=0}^N \sum_{q=0}^{N-p} \frac{1}{\gamma_p^{(0,0)}} \frac{2^{2p+1}}{\gamma_q^{(2p+1,0)}}
		\left( \left| \sum_{r=0}^{\infty} \frac{(-1)^r}{\gamma_r^{(n_{pq},0)}} \tilde{u}_{p,q,r} \right|^2 + 
		\left| \sum_{r=N-p-q+1}^{\infty} \frac{(-1)^r}{\gamma_r^{(n_{pq},0)}} \tilde{u}_{p,q,r} \right|^2 \right) \\
		&\leq I_1 +  2 N \sum_{p=0}^N \sum_{q=0}^{N-p} \frac{1}{\gamma_p^{(0,0)}} \frac{2^{2p+1}}{\gamma_q^{(2p+1,0)}}
		\left( \left| \sum_{r=0}^{\infty} \frac{(-1)^r}{\gamma_r^{(n_{pq},0)}} 2^{p+q+2}u_{p,q,r} \right|^2 + 
		\left| \sum_{r=N-p-q+1}^{\infty} \frac{(-1)^r}{\gamma_r^{(n_{pq},0)}} 2^{p+q+2} u_{p,q,r} \right|^2 \right) \\
		&\leq I_1 +  2 N \sum_{p=0}^N \sum_{q=0}^{N-p} \frac{1}{\gamma_p^{(0,0)}} 
\frac{2^{2p+1}}{\gamma_q^{(2p+1,0)}} 2^{-(p+q)}
		\left( \left| \sum_{r=0}^{\infty} \frac{(-1)^r}{\gamma_r^{(n_{pq},0)}} 2^{2(p+q)+2}u_{p,q,r} \right|^2 + 
		\left| \sum_{r=N-p-q+1}^{\infty} \frac{(-1)^r}{\gamma_r^{(n_{pq},0)}} 2^{2(p+q)+2} u_{p,q,r} \right|^2 \right).
\end{align*}
We now recall $I_1 \lesssim N \|\nabla u\|^2_{L^2(\T^3)}$ from above. The sums can be estimated
very generously: Using $2^{-(p+q)} \leq 1$, we get for the first sum in view of 
(\ref{eq:trace-norm})
\begin{align*}
& \sum_{p=0}^N \sum_{q=0}^{N-p} \frac{1}{\gamma_p^{(0,0)}} 
\frac{2^{2p+1}}{\gamma_q^{(2p+1,0)}} 2^{-(p+q)}
		\left( \left| \sum_{r=0}^{\infty} \frac{(-1)^r}{\gamma_r^{(n_{pq},0)}} 2^{2(p+q)+2}u_{p,q,r} \right|^2 \right) \\
& \leq \sum_{p=0}^N \sum_{q=0}^{N-p} \frac{1}{\gamma_p^{(0,0)}} 
\frac{2^{2p+1}}{\gamma_q^{(2p+1,0)}}
		\left( \left| \sum_{r=0}^{\infty} \frac{(-1)^r}{\gamma_r^{(n_{pq},0)}} 2^{2(p+q)+2}u_{p,q,r} \right|^2 \right)  = \|u\|^2_{L^2(\Gamma)} 
\lesssim \|u\|_{L^2(\T^3)} \|u\|_{H^1(\T^3)}, 
	\end{align*}
} 
where we set $\Gamma = \T^2 \times \{-1\}$ and used the multiplicative trace inequality, 
	\cite[Thm.~{1.6.6}]{brenner-scott94}.  For the second sum, we estimate again generously
$2^{-(p+q)} \leq 1$ and then recognize that several terms are those that have appeared 
in the estimate of $\|u - \Pi_N u\|_{L^2(\Gamma)}$ in the proof of Theorem~\ref{thm:trace-stability}: 
\begin{align*}
& \sum_{p=0}^N \sum_{q=0}^{N-p} \frac{1}{\gamma_p^{(0,0)}} 
\frac{2^{2p+1}}{\gamma_q^{(2p+1,0)}} 2^{-(p+q)}
		\left( \left| \sum_{r=N-p-q+1}^{\infty} \frac{(-1)^r}{\gamma_r^{(n_{pq},0)}} 2^{2(p+q)+2}u_{p,q,r} \right|^2 \right) \\
& \leq \sum_{p=0}^N \sum_{q=0}^{N-p} \frac{1}{\gamma_p^{(0,0)}} 
\frac{2^{2p+1}}{\gamma_q^{(2p+1,0)}}
		\left( \left| \sum_{r=N-p-q+1}^{\infty} \frac{(-1)^r}{\gamma_r^{(n_{pq},0)}} 2^{2(p+q)+2}u_{p,q,r} \right|^2 \right)  
\lesssim \|u\|_{L^2(\T^3)} \|u\|_{H^1(\T^3)}, 
	\end{align*}
and this completes the proof.
\fi
\end{proof}
\begin{numberedproof}{ of Corollary~\ref{cor:H1-stability}} 
\iftechreport 
	For a function $v$ and the transformed function $\tilde{v} = v \circ D$, the formula  
        (\ref{eq:partial-eta3-tilde-u}) provides a relation between $\partial_{\eta_3} \tilde v$ 
        and $\nabla v$. Rearranging terms yields 
	\begin{align*}
		(\partial_{\eta_3} \tilde v) \circ D^{-1}(\xi) =
		- \frac{1+\xi_1}{1-\xi_3} \partial_1 v(\xi) 
		- \frac{1+\xi_2}{1-\xi_3} \partial_2 v(\xi) + \partial_3 v(\xi). 
	\end{align*}
	Therefore, when transforming to $\T^3$ in Lemma~\ref{lemma:stability-of-eta3-derivatives}
	we get
	\begin{align}
		\label{eq:H1-stability-1} 
		\int_{\T^3} \left| - \frac{1+\xi_1}{1-\xi_3} \partial_1\Pi_N u(\xi) - \frac{1+\xi_2}{1-\xi_3} \partial_2\Pi_N u(\xi)
		+ \partial_3\Pi_N u(\xi)\right|^2 d\xi
		\lesssim N \|\nabla u\|^2_{L^2(\T^3)}. 
	\end{align}
	By the symmetry properties of $\T^3$, we see that also the following two other 
	permutations of indices are valid estimates:
	\begin{align}
		\label{eq:H1-stability-2} 
		\int_{\T^3} \left| - \frac{1+\xi_2}{1-\xi_1} \partial_2\Pi_N u(\xi) - \frac{1+\xi_3}{1-\xi_1} \partial_3\Pi_N u(\xi)
		+ \partial_1\Pi_N u(\xi)\right|^2 d\xi
		\lesssim N \|\nabla u\|^2_{L^2(\T^3)}, \\
		\label{eq:H1-stability-3}
		\int_{\T^3} \left| - \frac{1+\xi_3}{1-\xi_2} \partial_3\Pi_N u(\xi) - \frac{1+\xi_1}{1-\xi_2} \partial_1\Pi_N u(\xi)
		+ \partial_2\Pi_N u(\xi)\right|^2 d\xi
		\lesssim N \|\nabla u\|^2_{L^2(\T^3)} . 
	\end{align}
	We abbreviate $a(x,y) := - \frac{1+x}{1-y},~ a_{ij} := a(\xi_i,\xi_j)$ and 
	\begin{align*} 
		A(\xi_1,\xi_2,\xi_3) := 
		\left( 
		\begin{array}{ccc}
			1 + a_{13}^2 + a_{12}^2 
   		& a_{13} a_{23} + a_{21} + a_{12} 
      & a_{13} + a_{31} + a_{32} a_{12} \\
			sym  & 1 + a_{23}^2 + a_{21}^2
     	& a_{23} + a_{21} a_{31} + a_{32} \\
			sym & sym & 1 + a_{31}^2 + a_{32}^2 
		\end{array}
		\right).
	\end{align*}
	\noindent Hence, we see that by adding (\ref{eq:H1-stability-1}), (\ref{eq:H1-stability-2}), 
	and (\ref{eq:H1-stability-3}) we arrive at 
	\begin{align*}
		\int_{\T^3} (\nabla \Pi_N u)^\top A(\xi) \nabla\Pi_N u~ d\xi
		\lesssim N\|\nabla u\|^2_{L^2(\T^3)}
	\end{align*}
       \noindent Next, 
	we observe that near the top vertex $(-1,-1,1)$, we have 
	\begin{align*}
		\left| \frac{1+\xi_1}{1-\xi_3}\right| \leq 1 
		\quad \mbox{ and } \quad 
		\left| \frac{1+\xi_2}{1-\xi_3}\right| \leq 1. 
	\end{align*}
	This implies that the functions $a_{13}$ and $a_{23}$ are uniformly
	bounded on $\T^3$. Analogously, we get bounds for $a_{12}$, $a_{32}$ and 
	$a_{21}$, $a_{31}$ by studying the vertices $(1,-1,-1)$ and $(-1,1,-1)$. Therefore,
	we have 
	\begin{align*}
	\sup_{\xi \in \T^3} \|A(\xi)\|_{L^\infty(\T^3)} < \infty. 
	\end{align*}
	By construction, the matrix $A(\xi)$ is (pointwise) symmetric positive
	semidefinite. Our goal is to show that $A(\xi)$ is in fact positive definite
	on the set that stays away from the face $F$ opposite the vertex $(-1,-1,-1)$. 
        This can be done with techniques as in \cite{eibner-melenk06a}
	by establishing lower bounds for the eigenvalues of $A(\xi)$.  
	A direct calculation reveals 
	\begin{align*}
		\operatorname*{det} A(\xi) &= 
		16 \frac {\xi_1^2 + 2\xi_1\xi_2 + 2\xi_1\xi_3 + 2\xi_1 + \xi_2^2 + 1 + 2\xi_3 + 2\xi_3
		\xi_2 + 2\xi_2 + \xi_3^2}{\left( -1+\xi_1 \right)^2 \left( -1+\xi_3 \right)^2 
		\left( -1+\xi_2 \right)^2}\\
		&= 16 \frac{(\xi_1 + \xi_2 + \xi_3)^2 + 2(\xi_1 + \xi_2 + \xi_3) + 1}
		{\left( -1+\xi_1\right)^2 \left( -1+ \xi_3 \right)^2 \left( -1+\xi_2 \right)^2}
		= 16 \frac{(1 + \xi_1 + \xi_2 + \xi_3)^2} 
		{ \left( -1+ \xi_1 \right)^2 \left( -1+\xi_3 \right)^2 \left( -1+\xi_2 \right)^2}.
	\end{align*}
	The face opposite the vertex $(-1,-1,-1)$ contains the 
	vertices $(-1-1,1)$, $(-1,1,-1)$, $(1,-1,-1)$ and 
	is given by the equation 
	$\xi_1 + \xi_2 + \xi_3 + 1 = 0$.
	Furthermore, we conclude that the signed distance of an arbitrary point $\xi$ 
	from this face $F$ is given by 
	\begin{align*}
		\operatorname*{dist}(\xi,F) = \frac{1}{\sqrt{3}} (\xi_1 + \xi_2 + \xi_3 + 1).
	\end{align*}
	Let, for arbitrary $\delta > 0$, 
	\begin{align*}
		T_\delta  := \{\xi \in \T^3\,|\, \operatorname*{dist}(\xi,F)  < -\delta\}. 
	\end{align*}
	Then, since we stay away from the face F, it is clear that there exists $C_\delta > 0$ such that 
	\begin{align*}
		\operatorname*{det} A(\xi) \ge C_\delta \quad \forall \xi \in \overline{T}_\delta.
	\end{align*}
	Combining the above findings, we have that on $T_\delta$ the matrix 
	$A(\xi)$ is in fact symmetric positive definite. Since the entries of $A(\xi)$ 
	are uniformly bounded in $\xi$, Gershgorin's circle theorem provides
	a constant $C_{upper}$ such that all eigenvalues of $A(\xi)$ are bounded
	by $C_{upper}$. 

	A lower bound for the eigenvalues
	is obtained as follows: Denoting for fixed $\xi \in T_\delta$ the 
	eigenvalues $0< \lambda_1 \leq \lambda_2 \leq \lambda_3$, we get 
	from $\operatorname*{det} A = \lambda_1 \lambda_2 \lambda_3$ 
	\begin{align*}
		C_\delta \leq \operatorname*{det} A = \lambda_1 \lambda_2 \lambda_3 
		\leq \lambda_1 C_{upper}^2. 
	\end{align*}
	This provides the desired lower bound for $\lambda_1$. Thus, we conclude that
	for every $\delta > 0$ we can find $c_{\delta} > 0$ such that $A(\xi) \geq c_{\delta}I$
	on $T_{\delta}$. Hence,   
	\begin{align*}
		c_{\delta} \int_{T_{\delta}} \left|\nabla \Pi_N u\right|^2 d\xi 
		\leq \int_{\T^3} (\nabla \Pi_N u)^\top A(\xi) \nabla\Pi_N u~ d\xi
		\lesssim N\|\nabla u\|^2_{L^2(\T^3)}. 
	\end{align*}
	Affine transformations allow us to get analogous estimates for the sets that stay away
	from the other faces of $\T^3$. We therefore get the desired result.
\fi
\end{numberedproof}

\section{Numerical results}
\label{sec:numerics}

In this section, we illustrate the sharpness of 
Theorem~\ref{thm:trace-stability} and Corollary~\ref{cor:H1-stability}
for the 1D and the 2D case. We present the best constants in the following
1D and 2D situations: 
\begin{eqnarray}
|(\Pi_N u)(1)|^2 &\leq &C^{1D}_{mult} \|u\|_{L^2(I)} \|u\|_{H^1(I)} 
\qquad \forall {\mathcal P}_{2N}(I),  \\
\|\Pi_N u\|_{H^1(I)} &\leq & C^{1D}_{H^1} \sqrt{N+1} \|u\|_{H^1(I)} 
\qquad \forall u \in {\mathcal P}_{2N}(I), \\
\|(\Pi_N u)\|^2_{L^2(\Gamma)} &\leq &
C^{2D}_{mult} \|u\|_{L^2(\T^2)} \|u\|_{H^1(\T^2)} 
\qquad \forall u \in {\mathcal P}_{2N}(\T^2), \\ 
\|\Pi_N u\|_{H^1(\T^2)} &\leq & C^{2D}_{H^1} \sqrt{N+1} \|u\|_{H^1(\T^2)} 
\qquad \forall u \in {\mathcal P}_{2N}(\T^2),  
\end{eqnarray}
where $I = (-1,1)$ and $\Gamma = (-1,1) \times \{-1\} \subset \partial \T^2$. 
%
The best constants $C^{1D}_{mult}$, $C^{2D}_{mult}$ are solutions of 
constrained maximization problem. For example, 
$$
C^{2D}_{mult} = \max\{\|\Pi_N u\|^2_{L^2(\Gamma)}\,|\, \|u\|^2_{L^2(\T^2)} \|u\|^2_{H^1(\T^2)} = 1, \quad u \in {\mathcal P}_{2N}\}, 
$$
which can be solved using the technique of Lagrange multipliers. 
The constant 
$C^{2D}_{H^1}$ is more readily accessible as the solution of an eigenvalue 
problem since 
$$
C^{2D}_{H^1} = \sup_{u \in {\mathcal P}_{2N}} 
\frac{\|u\|^2_{L^2(\Gamma)}}{\|u\|^2_{H^1(\T^2)}}. 
$$
The result of the 1D situation is presented in Table~\ref{tab:1D-constants}
whereas the outcome of the 2D calculations is shown in 
Table~\ref{tab:2D-constants}. The 2D calculations are in agreement
with the results of Theorem~\ref{thm:trace-stability} and Corollary~\ref{cor:H1-stability}
whereas the 1D results illustrate 
\cite[Lem.~{4.1}]{georgoulis-hall-melenk10} and \cite[Thm.~{2.2}]{canuto-quarteroni82}. 
\begin{table}[htb]
        \begin{center}
                \begin{tabular}{|c|c|c|}
                        \hline
                        $N$ & {\scriptsize$\displaystyle \sup_{u\in\P_{2N}} \frac{|(\Pi_N u)(1)|^2}{\|u\|_{L^2(I)}\|u\|_{H^1(I)}}$ }
                        & {\scriptsize $\displaystyle \sup_{u\in\P_{2N}} \frac{|(\Pi_N u)(1)|^2}{\|u\|^2_{H^1(I)}}$} \\
                        \hline \hline
                        $1$ & 1.1818 & 0.8750 \\
                        $2$ & 1.8298 & 1.1436 \\
                        $3$ & 2.1527 & 1.1507 \\
                        $4$ & 2.3410 & 1.1353 \\
                        $5$ & 2.4594 & 1.1199 \\
                        $10$ & 2.7219 & 1.0826 \\
                        $15$ & 2.8221 & 1.0685 \\
                        $20$ & 2.8740 & 1.0611 \\
                        $25$ & 2.9051 & 1.0565 \\
                        $30$ & 2.9254 & 1.0534 \\
                        $35$ & 2.9394 & 1.0512 \\
                        $40$ & 2.9497 & 1.0495 \\
                        $45$ & 2.9574 & 1.0481 \\
                        $50$ & 2.9633 & 1.0471 \\
                        \hline
                \end{tabular}
                \begin{tabular}{|c|c|c|}
                        \hline
                        $N$ & {\scriptsize$\displaystyle \sup_{u\in\P_{2N}} \frac{|(\Pi_N u)(1)|^2}{\|u\|_{L^2(I)}\|u\|_{H^1(I)}}$ }
                        & {\scriptsize $\displaystyle \sup_{u\in\P_{2N}} \frac{|(\Pi_N u)(1)|^2}{\|u\|^2_{H^1(I)}}$} \\
                        \hline \hline
                        $55$ & 2.9680 & 1.0462 \\
                        $60$ & 2.9719 & 1.0455 \\
                        $65$ & 2.9750 & 1.0448 \\
                        $70$ & 2.9776 & 1.0443 \\
                        $75$ & 2.9798 & 1.0438 \\
                        $80$ & 2.9817 & 1.0434 \\
                        $85$ & 2.9833 & 1.0431 \\
                        $90$ & 2.9847 & 1.0428 \\
                        $95$ & 2.9859 & 1.0425 \\
                        $100$ & 2.9869 & 1.0422 \\
                       $105$ &  2.9879 & 1.0420 \\
                        $110$ & 2.9887 & 1.0418 \\
                       $115$ &  2.9895 & 1.0416 \\
                        $120$ & 2.9901 & 1.0414 \\
                        \hline
                \end{tabular}
                \caption{\label{tab:1D-constants} 1D maximization problems}
\end{center}
\end{table}
\begin{table}[htb]
	\begin{center}
		\begin{tabular}{|c|c|c|c|}	
			\hline
		$N$ & {\scriptsize $\displaystyle \sup_{u\in\P_{2N}} \frac{\|\Pi_N u\|^2_{L^2(\Gamma)}}{\|u\|_{L^2(\T^2)}\|u\|_{H^1(\T^2)}}$}
			& {\scriptsize $\displaystyle \sup_{u\in\P_{2N}} \frac{\|\Pi_N u\|^2_{L^2(\Gamma)}}{\|u\|^2_{H^1(\T^2)}}$} 
			& {\scriptsize $\displaystyle \sup_{u\in\P_{2N}} \frac{\|\Pi_N u\|^2_{H^1(\T^2)}}{\|u\|^2_{H^1(\T^2)}(N+1)}$}\\
			\hline \hline
			$1$  & 1.8417 & 1.4717 & 0.63072 \\
			$2$  & 2.4820 & 1.7051 & 0.53460 \\
			$3$  & 2.8401 & 1.7157 & 0.50256 \\
			$4$  & 3.0694 & 1.6988 & 0.46983 \\
			$5$  & 3.2214 & 1.6814 & 0.45149 \\
			$6$  & 3.3282 & 1.6683 & 0.44284 \\
			$7$  & 3.4079 & 1.6585 & 0.43994 \\
			$8$  & 3.4701 & 1.6508 & 0.43835 \\
			$9$  & 3.5203 & 1.6448 & 0.43732 \\
			$10$ & 3.5619 & 1.6398 & 0.43664 \\
			$15$ & 3.6957 & 1.6244 & 0.43502 \\
			$20$ & 3.7681 & 1.6165 & 0.43421 \\
			$25$ & 3.8129 & 1.6117 & 0.43370 \\
			$30$ & 3.8430 & 1.6084 & 0.43346 \\
			$35$ & 3.8643 & 1.6061 & 0.43344 \\
			$40$ & 3.8801 & 1.6043 & 0.43361 \\
			$45$ & 3.8921 & 1.6029 & 0.43396 \\
			$50$ & 3.9015 & 1.6018 & 0.43445 \\
			$55$ & 3.9090 & 1.6009 & 0.43506 \\
			\hline
		\end{tabular}
		\caption{\label{tab:2D-constants} 2D maximization problems}
	\end{center}
\end{table}
\iftechreport
\appendix
\def\appendixname{}

\section{Properties of Jacobi polynomials}
\label{app:A}

We have the following useful formulas (see \cite[p.~350 f]{karniadakis-sherwin99}, \cite{szego39}):

\subsubsection*{Recursion Relations}
\begin{align}
	a_n^1 P_{n+1}^{(\alpha,\beta)}(x) = (a_n^2 + a_n^3 x)P_n^{(\alpha,\beta)}(x) - a_n^4 P_{n-1}^{(\alpha,\beta)}(x)
  \label{eq:three-term-recurrence}
\end{align}
with
\begin{align*}
	a_n^1 &:= 2(n+1)(n+\alpha+\beta+1)(2n+\alpha+\beta) \\
	a_n^2 &:= (2n+\alpha+\beta+1)(\alpha^2-\beta^2) \\
	a_n^3 &:= (2n+\alpha+\beta)(2n+\alpha+\beta+1)(2n+\alpha+\beta+2) \\
	a_n^4 &:= 2(n+\alpha)(n+\beta)(2n+\alpha+\beta+2)
\end{align*}
\newline
\begin{align}
	b_n^1(x) \frac{d}{dx} P_n^{(\alpha,\beta)}(x) = b_n^2(x)P_n^{(\alpha,\beta)}(x) + b_n^3(x) P_{n-1}^{(\alpha,\beta)}(x)
	\label{eq:jacobi-recursion-diffeq}
\end{align}
with
\begin{align*}
	b_n^1(x) &:= (2n+\alpha+\beta) (1 - x^2) \\
	b_n^2(x) &:= n \left( \alpha-\beta-(2n+\alpha+\beta)x \right) \\
	b_n^3(x) &:= 2(n+\alpha)(n+\beta) \\
\end{align*}
\subsubsection*{Special Values}
\begin{align} 
        P_0^{(\alpha,\beta)} & \equiv 1 \\
	\label{eq:jacobi-polynomial-at-1}
	P_n^{(\alpha,\beta)}(1) &= \binom{n+\alpha}{n} \\
	\label{eq:jacobi-polynomial-symmetry}
	P_n^{(\alpha,\beta)}(-x) &= (-1)^n P_n^{(\beta,\alpha)}(x) \\
        P_1^{(\alpha,0)}(x) &= 1+\alpha + \frac{1}{2} (2+\alpha+\beta)(x-1) \\
        \widehat P_1^{(\alpha,0)}(x) &= \int_{-1}^x P_p^{(\alpha,0)}(t)\,dt = 1+x
\end{align}
\newline
\textbf{Special Cases}
For the Legendre Polynomial $L_n(x)$ there holds
\begin{align}
	L_n(x) = P_n^{(0,0)}(x)
	\label{eq:legendre-jacobi-relation}
\end{align}
\subsubsection*{Miscellaneous Relations}
\begin{align}
	\label{eq:differentiated-jacobi-polynomial}
	\frac{d}{dx} P_n^{(\alpha,\beta)}(x) &= \frac{1}{2} (\alpha+\beta+n+1) P_{n-1}^{(\alpha+1,\beta+1)}(x) \\
	\label{eq:integrated-jacobi-polynomial}
	2n \int_{-1}^x (1-t)^\alpha(1+t)^\beta P_{n}^{(\alpha,\beta)}(t)\,dt &= - (1-x)^{\alpha+1}(1+x)^{\beta+1} P_{n-1}^{(\alpha+1,\beta+1)}(x)
\end{align}

\section{Details for selected proofs}
\label{app:B}
\subsection{Selected proofs for Section~\ref{sec:1D}}
\label{app:B1}
\subsubsection{Extended proofs of Lemma~\ref{lemma:relations-between-hi-gi}}
\begin{proof}[Proof of \protect{(\ref{eq:magic-cancellation})}]
	By definition of $\gamma_p^{(\alpha,\beta)}$
	we obtain in particular
	\begin{align*}
		\gamma_q^{(\alpha,0)} = \frac{2^{\alpha+1}}{2q+\alpha+1},
	\end{align*}
	which leads in combination with the definition of $h_1, h_2$ and $h_3$ to
	\begin{align*}
		&(-1)^q \frac{1}{\gamma^{(\alpha,0)}_q}h_1(q,\alpha) + 
		(-1)^{q+1} \frac{1}{\gamma^{(\alpha,0)}_{q+1}}h_2(q+1,\alpha) + 
		(-1)^{q+2} \frac{1}{\gamma^{(\alpha,0)}_{q+2}}h_3(q+2,\alpha) \\
		&\qquad= (-1)^{q+1} \frac{2q+\alpha+1}{2^{\alpha+1}}\frac{2q+2}{(2q+\alpha+1)(2q+\alpha+2)}
		+ (-1)^{q+1} \frac{2q+\alpha+3}{2^{\alpha+1}}\frac{2\alpha}{(2q+\alpha+4)(2q+\alpha+2)} \\
		&\qquad\qquad+ (-1)^{q+2} \frac{2q+\alpha+5}{2^{\alpha+1}}\frac{2(q+\alpha+2)}{(2q+\alpha+5)(2q+\alpha+4)} \\
		&\qquad = \frac{(-1)^{q+1}}{2^{\alpha}}\left( \frac{q+1}{2q+\alpha+2} + \frac{(2q+\alpha+3)\alpha}{(2q+\alpha+4)(2q+\alpha+2)}
		- \frac{q+\alpha+2}{2q+\alpha+4} \right) \\
		&\qquad = \frac{(-1)^{q+1}}{2^{\alpha}}\left( \frac{(q+1)(2q+\alpha+4) + (2q+\alpha+3)\alpha 
		- (q+\alpha+2)(2q+\alpha+2)}{(2q+\alpha+4)(2q+\alpha+2)} \right)
	\end{align*}
	Simply multiplying out the numerator concludes the proof regarding the first equation.

	Inserting the definition of $h_1$, $h_2$ and $h_3$ also leads in the case of the second equation to the conclusion
	\begin{align*}
		h_2(q,\alpha) - h_1(q,\alpha) 
		&= \frac{2\alpha(2q+\alpha+1)+2(q+1)(2q+\alpha)}{(2q+\alpha)(2q+\alpha+1)(2q+\alpha+2)} \\
		&= \frac{4q^2+4q+6q\alpha+2\alpha^2+4\alpha}{(2q+\alpha)(2q+\alpha+1)(2q+\alpha+2)} \\
		&= \frac{(2q+\alpha+2)(2q+2\alpha)}{(2q+\alpha)(2q+\alpha+1)(2q+\alpha+2)} = h_3(q,\alpha).
	\end{align*}
\end{proof}

\subsubsection{Extended proofs of Lemma~\ref{lemma:relations-of-jacobi-in-terms-of-jacobi}}
\begin{proof}[Proof of Lemma~\ref{lemma:relations-of-jacobi-in-terms-of-jacobi} (i)] 
	Using rearranged versions of (\ref{eq:jacobi-recursion-diffeq}), (\ref{eq:differentiated-jacobi-polynomial}) 
	and (\ref{eq:integrated-jacobi-polynomial}) we obtain 
	\begin{align*} 
		\int_{-1}^x (1-t)^\alpha P_q^{(\alpha,0)}(t) dt 
		&\overset{(\ref{eq:integrated-jacobi-polynomial})}{=} - \frac{1}{2q} (1+x)(1-x)^{\alpha+1} P_{q-1}^{(\alpha+1,1)}(x) \\
		&\overset{\hphantom{(\ref{eq:integrated-jacobi-polynomial})}}{=} - \frac{1}{2q} (1-x^2) (1-x)^\alpha P_{q-1}^{(\alpha+1,1)}(x) \\ 
		&\overset{(\ref{eq:differentiated-jacobi-polynomial})}{=} - (1-x)^\alpha \frac{1}{2q} (1-x^2) 
			\frac{2}{q+\alpha+1}\frac{d}{dx} P_{q}^{(\alpha,0)}(x) \\
		&\overset{(\ref{eq:jacobi-recursion-diffeq})}{=} - (1-x)^\alpha \frac{1}{q} \frac{1}{q+\alpha+1} 
		\frac{q(\alpha-(2q+\alpha) x) P_q^{(\alpha,0)}(x) + 2 q(q+\alpha)P_{q-1}^{(\alpha,0)}(x)}
		     {2q+\alpha}\\
		&\overset{\hphantom{(\ref{eq:integrated-jacobi-polynomial})}}{=} - (1-x)^\alpha \frac{\alpha P_q^{(\alpha,0)}(x) + 2 (q+\alpha) P_{q-1}^{(\alpha,0)}(x) - 
			(2q+\alpha) x P_q^{(\alpha,0)}(x)}{(q+\alpha+1)(2q+\alpha)}
	\end{align*}
	(\ref{eq:three-term-recurrence}) allows us now to replace the term $x P_q^{(\alpha,0)}(x)$ 
	by terms involving $P_{q+1}^{(\alpha,0)}(x)$, $P_{q}^{(\alpha,0)}(x)$ and $P_{q-1}^{(\alpha,0)}(x)$. 
	Hence, we get 
	\begin{align*}
		&\int_{-1}^x (1-t)^\alpha P_q^{(\alpha,0)}(t) dt 
		= - (1-x)^\alpha \frac{1}{(q+\alpha+1)(2q+\alpha)} \Bigg\{ \alpha P_q^{(\alpha,0)}(x) + 2(q+\alpha) P_{q-1}^{(\alpha,0)}(x) \\
		&\qquad - \frac{1}{(2q+\alpha+1)(2q+\alpha+2)} \Big( 2 (q+1)(q+\alpha+1)(2q+\alpha) P_{q+1}^{(\alpha,0)}(x) \\
		&\qquad + 2q(q+\alpha)(2q+\alpha+2) P_{q-1}^{(\alpha,0)}(x) - (2q+\alpha+1) \alpha^2 P_q^{(\alpha,0)}(x) \Big) \Bigg\}
	\end{align*}
	Rearranging terms gives 
	\begin{align*}
	\int_{-1}^x (1-t)^\alpha P_q^{(\alpha,0)}(t) dt &= - (1-x)^\alpha \frac{1}{(q+\alpha+1)(2q+\alpha)} 
	\Bigg\{ - \frac{2(q+1)(q+\alpha+1)(2q+\alpha)}{(2q+\alpha+1)(2q+\alpha+2)} P_{q+1}^{(\alpha,0)}(x) \\
	&\qquad + \alpha \frac{2q+2\alpha+2}{2q+\alpha+2} P_{q}^{(\alpha,0)}(x) 
	+ 2 (q+\alpha) \frac{q+\alpha+1}{2q+\alpha+1} P_{q-1}^{(\alpha,0)}(x) \Bigg\} \\
	&= - (1-x)^\alpha \Bigg\{\underbrace{- \frac{2(q+1)}{(2q+\alpha+1)(2q+\alpha+2)}}_\text{$h_1(q,\alpha)$}
	P_{q+1}^{(\alpha,0)}(x) \\
	&\qquad + \underbrace{\frac{2\alpha}{(2q+\alpha+2)(2q+\alpha)}}_\text{$h_2(q,\alpha)$} P_{q}^{(\alpha,0)}(x) 
	+ \underbrace{\frac{2(q+\alpha)}{(2q+\alpha+1)(2q+\alpha)}}_\text{$h_3(q,\alpha)$} P_{q-1}^{(\alpha,0)}(x) \Bigg\}
	\end{align*}
\end{proof}

\subsubsection{Extended proofs of Lemma~\protect{\ref{lemma:norm-of-Pqprime}}}
\begin{lemma}[details of the proof of Lemma~\protect{\ref{lemma:norm-of-Pqprime}}]
\label{lemma:details-of-norm-of-Pqprime-1}
Write 
$$
I^2_{q,\alpha}:= \int_{-1}^1 (1-x)^\alpha \left|\left(P_q^{(\alpha,0)}\right)^\prime(x)\right|^2\,dx
$$
Then
\begin{enumerate}[(i)] 
\item 
\label{item:lemma:details-of-norm-of-Pqprime-1:i}
for $\alpha = 0$ we have $I^2_{q,0} = q (q+1) \leq K \cdot q (q+1)^2 \gamma^{(0,0)}_q$
with $K =1$
\item
\label{item:lemma:details-of-norm-of-Pqprime-1:ii}
for $\alpha \ge 1$ we have 
\begin{eqnarray}
\label{eq:item:lemma:details-of-norm-of-Pqprime-1:ii-1}
I_{1,\alpha}^2 &=& \frac{(\alpha+2)^2 2^{\alpha+1}}{4 (\alpha+1)} 
\leq K \cdot 1 \cdot (1 + 1+\alpha)^2 \gamma^{(\alpha,0)}_1 \quad \mbox{ with $K = 1/2$}\\
\label{eq:item:lemma:details-of-norm-of-Pqprime-1:ii-2}
I_{2,\alpha}^2 &=& \frac{(3+\alpha) (\alpha+2) 2^{\alpha+1}}{2 (\alpha + 1)}
\leq K \cdot 2 \cdot (2 + 1+\alpha)^2 \gamma^{(\alpha,0)}_2 \quad \mbox{ with $K = 9/16$}
\end{eqnarray}
\end{enumerate}
\end{lemma}
\begin{proof}
{\em Proof of (\ref{item:lemma:details-of-norm-of-Pqprime-1:i}):}
We have $I^2_{q,0} = q (q+1)$ by \cite[(5.3)]{bernardi-maday97}. The statement
therefore follows directly.

{\em Proof of (\ref{item:lemma:details-of-norm-of-Pqprime-1:ii}):}
A direct calculation shows 
$$
I_{1,\alpha}^2 \frac{1}{1 \cdot (1+1+\alpha)^2  \gamma^{(\alpha,0)}_1} 
= \frac{(\alpha+2)^2}{4} \frac{2^{\alpha+1}}{\alpha+1} \frac{(3+\alpha)}{2^{\alpha+1}(2+\alpha)^2}
= \frac{(\alpha+3)}{4 (\alpha+1)}
$$
this last function is monotone decreasing in $\alpha$ so that its maximum on $[1,\infty)$
is attained for $\alpha = 1$, which has the value $1/2$. Analogously, we 
proceed for the case $q = 2$. We have 
$$
I_{2,\alpha}^2 \frac{1}{2 \cdot (2+1+\alpha)^2  \gamma^{(\alpha,0)}_2} 
= 
\frac{(\alpha+2)(\alpha+3)}{2(\alpha+1)}2^{\alpha+1} \frac{5+\alpha}{2^{\alpha+1} 2 (3+\alpha)^2}
=  \frac{(\alpha+2)(\alpha+5)}{4(\alpha+1)(\alpha+3)}. 
$$
Again, this last function is monotone decreasing on $(0,\infty)$ so that its maximum
on $[1,\infty)$ is attained for $\alpha = 1$. 
\end{proof}
\begin{lemma}[details of the proof of Lemma~\protect{\ref{lemma:norm-of-Pqprime}}]
Define for $\alpha$, $q \in \BbbN_0$
	\begin{align*}
		\varepsilon_q:= - \frac{g_2(q+1,\alpha) g_3(q,\alpha)}{g_1(q+1,\alpha) g_1(q,\alpha)} = 
		\frac{\alpha (2q+1+\alpha)(q-1)}{(q+1+\alpha)(2q+\alpha-2)(q+\alpha)}.
	\end{align*}
Then, for $\alpha$, $q \ge 1$ we have $0 \leq \varepsilon_q \leq 1$. 
\end{lemma}
\begin{proof}
Clearly, $\varepsilon_q \ge 0$. To see the estimate $\varepsilon_q \leq 1$, we have to show 
\begin{eqnarray*}
\alpha (2q +1+\alpha)(q-1) &\stackrel{?}{\leq}& (q+\alpha)(q+1+\alpha)(2q+\alpha-2) \\
\Longleftrightarrow \quad 
\alpha (q-1)(2q+\alpha-2) + 3 \alpha (q-1)  &\stackrel{?}{\leq}& (2q+\alpha-2)(q+\alpha)(q+\alpha+1) \\
\Longleftrightarrow \quad 
\underbrace{(2q+\alpha-2)}_{\ge 1}(\alpha(q-1)-(q+\alpha+1)(q+\alpha))+ 3 \alpha (q-1)  &\stackrel{?}{\leq}&  0
\end{eqnarray*}
This last inequality is certainly true if 
\begin{eqnarray*}
\alpha (q-1) - (q+\alpha+1)(q+\alpha) + 3\alpha(q-1) &\stackrel{?}{\leq} & 0\\
\Longleftrightarrow \qquad 4 \alpha (q-1) - (q+\alpha+1)(q+\alpha)  &\stackrel{?}{\leq} & 0\\
\Longleftrightarrow \qquad 4 \alpha (q-1) - (q+\alpha)^2 - (q+\alpha)  &\stackrel{?}{\leq} & 0\\
\Longleftrightarrow \qquad 4 \alpha (q-1) - q^2 - 2 \alpha q - \alpha^2  - (q+\alpha)  &\stackrel{?}{\leq} & 0\\
\Longleftrightarrow \qquad 2 \alpha q- 4\alpha  - q^2  - \alpha^2  - (q+\alpha)  &\stackrel{?}{\leq} & 0\\
\Longleftrightarrow \qquad - 4\alpha  - (q - \alpha)^2  - (q+\alpha)  &\stackrel{?}{\leq} & 0, 
\end{eqnarray*}
which is indeed the case. 
\end{proof}
\begin{lemma}[details of the proof of Lemma~\protect{\ref{lemma:norm-of-Pqprime}}]
For $\alpha \in \BbbN$, $q \ge 1$ we have 
\begin{eqnarray}
\label{eq:foo-100}
\left(\frac{1}{g_1(q+1,\alpha)}\right)^2 \gamma_q^{(\alpha,0)} &\leq& 
\frac{4}{q+1} \cdot  (q+1) ((q+1)+\alpha+1)^2 \gamma_{q+1}^{(\alpha,0)}  
\\
\label{eq:foo-101}
\left(\frac{g_2(q+1,\alpha)}{g_1(q+1,\alpha)g_1(q,\alpha)}\right)^2 \gamma_{q-1}^{(\alpha,0)}&\leq& 
\frac{4}{q+1} \cdot  (q+1)  ((q+1)+\alpha+1)^2 \gamma_{q+1}^{(\alpha,0)}, \\
\label{eq:foo-102}
4 (q-1) (q+\alpha)^2 \gamma_{q-1}^{(\alpha,0)} &=& 
4 ((q+1)+1+\alpha)^2 (q+1) \gamma_{q+1}^{(\alpha,0)} 
\left(1 - \frac{2}{q+\alpha+ 2}\right)^2 
\left( 1-  \frac{2 +2 \alpha}{(q+1)(2q+\alpha -1)}\right) 
\end{eqnarray}
Furthermore, we have 
\begin{eqnarray*}
\frac{8}{4(q+1)}  + 
\left(1 - \frac{2}{q+\alpha+ 2}\right)^2 
\left( 1-  \frac{2 +2 \alpha}{(q+1)(2q+\alpha -1)}\right) 
&=& 
1 - 4 \frac{1-2q-\alpha q+q^2\alpha+q^3}{(q+1)(2q+\alpha-1)(q+\alpha+2)^2} 
\\
&=& 1 - 4 \frac{(q-1)^2+q^2(q-1)+\alpha q(q-1)}{(q+1)(2q+\alpha-1)(q+\alpha+2)^2} 
\end{eqnarray*}
\end{lemma}
\begin{proof} 
We start with the bound (\ref{eq:foo-100}). We compute 
\begin{align*}
& \frac{1}{(g_1(q+1,\alpha))^2} \gamma_q^{(\alpha,0)} \frac{1}{(q+1)(q+2+\alpha)^2 \gamma^{(\alpha,0)}_{q+1}} 
= 
\frac{ (2(q+1)+\alpha-1)^2 (2 (q+1)+\alpha)^2 (2(q+1)+\alpha+1) }
{(2q+\alpha+1) (2(q+1)+2\alpha)^2 (q+2+\alpha)^2 (q+1)} \\
& = 
\frac{1}{q+1}
\frac{ (2q+1+\alpha) (2 (q+1)+\alpha)^2 (2(q+1)+\alpha+1) }
{(2(q+1)+2\alpha)^2 (q+2+\alpha)^2 } 
= 
\frac{4}{q+1}
\frac{ (2q+1+\alpha) (2 q+\alpha +2 )^2 (2q+\alpha+3) }
{(2q+2\alpha+2)^2 (2 q+2 \alpha +4)^2 } \\
&\leq \frac{4}{q+1}.
\end{align*}
We now turn to the bound (\ref{eq:foo-101}). 
We first compute 
\begin{align*}
\frac{g_2(q+1,\alpha)}{g_1(q+1,\alpha) g_1(q,\alpha)} & = 
\frac{2\alpha}{(2(q+1)+\alpha-2)(2(q+1)+\alpha)} 
\frac{(2(q+1)+\alpha-1)(2(q+1)+\alpha)}{2(q+1)+2\alpha} 
\frac{(2q+\alpha-1)(2q+\alpha)}{2q+2\alpha}  \\
& = 
\frac{2\alpha (2q+\alpha+1) (2q+\alpha-1)}{(2q+2\alpha+2)(2q+2\alpha)}. 
\end{align*}
Then 
\begin{align*} 
& 
\left(\frac{g_2(q+1,\alpha)}{g_1(q+1,\alpha) g_1(q,\alpha)} \right)^2 
\frac{\gamma_{q-1}^{(\alpha,0)}}{\gamma_{q+1}^{(\alpha,0)}} \frac{1}{(q+1) (q+2+\alpha)^2}  \\
& =  
\left(\frac{2\alpha (2q+\alpha+1) (2q+\alpha-1)}{(2q+2\alpha+2)(2q+2\alpha)}\right)^2 
\frac{2(q+1)+\alpha+1}{2(q-1)+\alpha+1} \frac{1}{(q+1)(q+2+\alpha)^2}  \\
& = \frac{4}{q+1} \frac{(2\alpha)^2 (2q+\alpha+1)^2 (2q+\alpha-1)^2}{(2q+2\alpha+2)^2(2q+2\alpha)^2}
\frac{2q+\alpha+3}{2q+\alpha-1} \frac{1}{(2q+2\alpha+4)^2}  \\
&\leq 
\frac{4}{q+1} \frac{(2\alpha)^2}{(2q+2\alpha)^2} \\
&\leq \frac{4}{q+1}. 
\end{align*}
Finally, we show (\ref{eq:foo-102}).
\begin{align*}
& 4 (q-1) (q+\alpha)^2 \gamma_{q-1}^{(\alpha,0)} =
4(q+2+\alpha)^2 (q +1)\gamma_{q+1}^{(\alpha,0)} 
\frac{(q+\alpha)^2 (q-1) \gamma_{q-1}^{(\alpha,0)}}{(q+2+\alpha)^2 (q+1)\gamma_{q+1}^{(\alpha,0)}} \\
&= 4(q+2+\alpha)^2 (q +1)\gamma_{q+1}^{(\alpha,0)} 
\left(1 - \frac{2}{q+\alpha+2}\right)^2 
\frac{q-1}{q+1} \frac{2q+\alpha+3}{2q+\alpha-1} \\
&= 4(q+2+\alpha)^2 (q +1)\gamma_{q+1}^{(\alpha,0)} 
\left(1 - \frac{2}{q+\alpha+2}\right)^2 
\left(1 - \frac{2}{q+1}\right)
\left(1 + \frac{4}{2q+\alpha-1}\right) \\
&= 4(q+2+\alpha)^2 (q +1)\gamma_{q+1}^{(\alpha,0)} 
\left(1 - \frac{2}{q+\alpha+2}\right)^2 
\left(1 - \frac{2}{q+1} + \frac{4}{2q+\alpha-1}
- \frac{2}{q+1} \frac{4}{2q+\alpha-1}
\right)  \\
&= 4(q+2+\alpha)^2 (q +1)\gamma_{q+1}^{(\alpha,0)} 
\left(1 - \frac{2}{q+\alpha+2}\right)^2 
\left(1 - \frac{2\alpha+2}{(q+1)(2q+\alpha-1)}
\right), 
\end{align*}
which is the claimed statement. 
\end{proof}
\begin{lemma} 
	\label{lemma:connection-U-and-Uprime-appendix-B}
	Let $U\in C^1(-1,1)$ and let $(1-x)^\alpha U(x)$ as well as $(1-x)^{\alpha+1} U^\prime$ be 
        integrable. Furthermore, let 
	\begin{align*}
		\lim_{x \rightarrow 1} (1-x)^{1+\alpha} U(x) = 0\quad and \quad
		\lim_{x\rightarrow -1}(1+x) U(x) = 0.
	\end{align*} 
	Consider  $h_1$, $h_2$ and $h_3$ from (\ref{def:hi-gi}). We define
	\begin{align*}
		u_q &:= \int_{-1}^1 (1-x)^\alpha U(x) P_q^{(\alpha,0)}(x) dx, \\
		b_q &:= \int_{-1}^1 (1-x)^\alpha U^\prime(x) P_q^{(\alpha,0)}(x) dx. 
	\end{align*}
	Then for $q \ge 1$ and $\alpha\in\BbbN_0$ the following relationship holds: 
	\begin{align*}
		u_q = h_1(q,\alpha) b_{q+1} + h_2(q,\alpha) b_q + h_3(q,\alpha) b_{q-1}.
	\end{align*}
\end{lemma}

\begin{proof}
	From (\ref{eq:integrated-jacobi-polynomial}) we have for $x\rightarrow -1$ 
	\begin{align*}
		\int_{-1}^x (1-t)^\alpha P_q^{(\alpha,0)}(t) dt = O(1+x)
	\end{align*}
	and for $x\rightarrow 1$
	\begin{align*}
		&\int_{-1}^x (1-t)^\alpha P_q^{(\alpha,0)}(t) dt = O\left((1-x)^{\alpha+1}\right).
	\end{align*}
	Hence, using the stipulated behavior of $U$ at the endpoints, 
	the following integration by parts can be justified: 
	\begin{align*}
		u_q 
		&= \int_{-1}^1 (1-x)^\alpha U(x) P^{(\alpha,0)}_q(x) dx \\
		&= \bigg(U(x)\int_{-1}^x (1-t)^\alpha P_q^{(\alpha,0)}(x)\bigg) \bigg|_{-1}^{1} 
		- \int_{-1}^{1} U^\prime(x) \int_{-1}^x (1-t)^\alpha P_q^{(\alpha,0)}(t) dt\,dx.  
	\end{align*}
        In particular, we note that $b_q$ is well-defined. Furthermore, 
	\begin{align*} 
		u_q &= - \int_{-1}^{1} U^\prime(x) \int_{-1}^x (1-t)^\alpha P^{(\alpha,0)}_q(t) dt\,dx \\
		&= \int_{-1}^1 (1-x)^\alpha U^\prime(x) \left( 
		h_1(q,\alpha) P_{q+1}^{(\alpha,0)}(x) + h_2(q,\alpha) P_q^{(\alpha,0)}(x) + 
		h_3(q,\alpha) P_{q-1}^{(\alpha,0)}(x) \right) dx \\
		&= h_1(q,\alpha) b_{q+1} + h_2(q,\alpha) b_{q} + h_3(q,\alpha) b_{q-1},
	\end{align*}
	where in the second equation we appealed to Lemma~\ref{lemma:relations-of-jacobi-in-terms-of-jacobi} (i).
\end{proof}
\begin{lemma}
	\label{lemma:increase-weight-appendix-B}
	For $\beta > -1$ and $U\in C^1(0,1) \cap C((0,1])$ there holds 
	\begin{align*}
		\int_0^1 x^\beta \left|U(x)\right|^2 dx \leq \left(\frac{2}{\beta+1}\right)^2 \int_0^1 x^{\beta+2} \left|U^\prime(x)\right|^2 dx
		+ \frac{1}{\beta+1} |U(1)|^2. 
	\end{align*}
\end{lemma}

\begin{proof}
	We define according to the notation of \cite[Thm.~{330}]{hardy-littlewood-polya91}
	\begin{align*}
		\begin{split}
		f(x) := 
		\begin{cases} 
    	U^\prime(x) & 0 < x < 1 \\
    	0  &  x > 1 \\
  	\end{cases}
		\quad\textnormal{and}\quad 
		F(x) := \int_x^\infty f(t)\,dt = 
		\begin{cases}
			0 & x > 1 \\
			U(1) - U(x) & 0 < x < 1.
		\end{cases}
		\end{split}
	\end{align*}
	For $\beta > -1$, \cite[Thm.~{330}]{hardy-littlewood-polya91} states
	\begin{align*}
		\int_0^\infty x^\beta |F(x)|^2 dx \leq \left(\frac{2}{\beta + 1}\right)^2 \int_0^\infty x^{\beta+2} |f(x)|^2 dx
	\end{align*}
	Hence, 
	\begin{align*}
		\int_0^1 x^\beta (U(x) - U(1))^2 dx \leq \left(\frac{2}{\beta+1}\right)^2 \int_0^1 x^{\beta+2} |U^\prime(x)|^2 dx.
	\end{align*}
	By rearranging terms, we get 
	\begin{align*}
		\int_0^1 x^\beta |U(x)|^2 dx 
		&\leq \left(\frac{2}{\beta+1}\right)^2 \int_0^1 x^{\beta+2} |U^\prime(x)|^2 dx
		+ |U(1)|^2 \int_0^1 x^\beta dx\\
		&\leq \left(\frac{2}{\beta+1}\right)^2 \int_0^1 x^{2+\beta} |U^\prime(x)|^2\,dx 
		+ \frac{1}{\beta+1} |U(1)|^2. 
	\end{align*}
\end{proof}

\subsection{Selected proofs for Section~\ref{sec:expansions}}
\label{app:B2}
We start by recalling the definition of the 2D version of 
the Duffy transformation 
$D^{2D}:\BbbR^2\rightarrow\BbbR^2$ given by
$$
D^{2D}: (\eta_1,\eta_2)\mapsto (\xi_1,\xi_2) = \left(\frac{1}{2} (1+\eta_1)(1-\eta_2)-1,\eta_2\right). 
$$
It maps $\S^2$ onto $\T^2$. Its inverse $(D^{2D})^{-1}:\T^2 \rightarrow\S^2$ 
is given by
\[
(D^{2D})^{-1}: (\xi_1,\xi_2)\mapsto (\eta_1,\eta_2)=\left(2\frac{1+\xi_1}{1-\xi_2}-1,\xi_2\right).
\]
A calculation shows 
$$
(D^{2D})^\prime = \left(\begin{array}{cc} \frac{1-\eta_2}{2} & - \frac{1+\eta_1}{2} \\ 0 & 1 \end{array}\right), 
\qquad \operatorname*{det} (D^{2D})^\prime = \frac{1-\eta_2}{2}
$$
We also recall the following expansion result for the 
triangle $\T^2$ (this is well-known, see, e.g., 
\cite[Sec.~{3.2.3}]{melenk02} for details):
\begin{lemma}
\label{lemma:2D-Duffy}
\begin{enumerate}[(i)]
\item The polynomials $\psi^{2D}_{p,q}$ defined by 
$$
\psi^{2D}_{p,q} \circ D^{2D}:= 
P^{(0,0)}_p(\eta_1) 
\left(\frac{1-\eta_2}{2}\right)^p P^{(2p+1,0)}_q(\eta_2) 
$$
are orthogonal polynomials on $\T^2$ and satisfy 
$$
\int_{\T^2} \psi^{2D}_{p,q} \psi^{2D}_{p^\prime,q^\prime}\,d\xi_1\,d\xi_2 = \delta_{p,p^\prime} \delta_{q,q^\prime} 
\frac{2}{2p+1} \frac{2}{2p+2q+2} = 
\delta_{p,p^\prime} \delta_{q,q^\prime} \gamma_p^{(0,0)}
\frac{\gamma_q^{(2p+1,0)}}{2^{2p+1}}
$$
\item 
Any $u \in L^2(\T^2)$ can be expanded as 
\begin{equation}
\label{eq:lemma:2D-Duffy-expansion-triangle}
u = \sum_{p,q} \frac{1}{\gamma_p^{(0,0)}} \frac{2^{2p+1}}{\gamma_q^{(2p+1,0)}} 
\left(\int_{\T^2} u \psi^{2D}_{p,q}\,d\xi_1\,d\xi_2\right)\psi^{2D}_{p,q}, 
\end{equation}
or, written as integrals over $\S^2$ with $\tilde u:= u \circ D^{2D}$
\begin{equation}
\label{eq:lemma:2D-Duffy-expansion-square}
\tilde u = \sum_{p,q} \frac{1}{\gamma_p^{(0,0)}} \frac{2^{2p+1}}{\gamma_q^{(2p+1,0)}} \left( \int_{\S^2} u P^{(0,0)}_q(\eta_1) 
\left(\frac{1-\eta_2}{2}\right)^{p+1} P^{(2p+1,0)}_q(\eta_2) \,d\eta_1\,d\eta_2 \right) 
P^{(0,0)}_p(\eta_1) \left(\frac{1-\eta_2}{2}\right)^p P^{(2p+1,0)}_q(\eta_2)
\end{equation}
\end{enumerate}
\end{lemma}
\begin{lemma}[details of \protect{Lemma~\ref{lemma:duffy-properties-2}}]
	\label{lemma:duffy-properties-2-appendix-B}
	Let $D$ be the Duffy transformation and $\Gamma := \T^2 \times \{-1\}$. 
        Set $\tilde \Gamma:= \S^2 \times \{-1\}$. 
  Then $D(\tilde \Gamma) = \Gamma$ and with the notation 
$\tilde u:= u \circ D$ we have 
$$
\int_{\tilde \Gamma} |\tilde u(\eta_1,\eta_2,-1)|^2 
\frac{1-\eta_2}{2} \,d\eta_1\,d\eta_2
 = \int_\Gamma |u(\xi_1,\xi_2,-1)|^2\,d\xi_1\,d\xi_2
$$
\end{lemma}

\begin{proof}
	Obviously $D$ is an isomorphism and by definition $D(\Gamma) = \Gamma$, so we will only show the isometry property.

	Let u be a quadratic integrable function on $\T^3$ and consider the transformed function $\tilde{u} = u \circ D$. We have
	\begin{align*}
\int_{\tilde\Gamma} |\tilde u(\eta_1,\eta_2,-1)|^2 \frac{1-\eta_2}{2}\,d\eta_1\,d\eta_2 
		&= \int_{-1}^1 \int_{-1}^1 \left| u\big(D(\eta_1,\eta_2,-1)\big) \right|^2 \frac{1-\eta_2}{2} d\eta_1 d\eta_2 \\
		&= \int_{-1}^1 \int_{-1}^1 \left| u\left( \frac{(1+\eta_1)(1-\eta_2)}{2}-1,\eta_2,-1 \right) \right|^2 \frac{1-\eta_2}{2} d\eta_1 d\eta_2 \\
		&= \int_{-1}^1 \int_{-1}^1 \left| u\left( D^{2D}(\eta_1,\eta_2),-1 \right) \right|^2 \frac{1-\eta_2}{2} d\eta_1 d\eta_2 \\
		&= \int_{\T^2} \left| u(\xi_1,\xi_2,-1) \right|^2 d\xi_1 \xi_2 = \|u\|_{L^2(\Gamma)}^2.
	\end{align*}
\end{proof}
\begin{proof}[Proof of Lemma~\ref{lemma:orthogonal-polynomials-on-tet}]
	With the definition of $D^{-1}$ and the abbreviation $n_{pq}:=2p+2q+2$ we get
	\begin{align*}
		&\psi_{p,q,r}(\xi_1,\xi_2,\xi_3)
		= \tilde{\psi}_{p,q,r}\left(-2~\frac{1+\xi_1}{\xi_2+\xi_3}-1, 2~\frac{1+\xi_2}{1-\xi_3}-1, \xi_3\right) \\
		&= P_p^{(0,0)}\left(-2~\frac{1+\xi_1}{\xi_2+\xi_3}-1\right)P_q^{(2p+1,0)}\left(2~\frac{1+\xi_2}{1-\xi_3}-1\right)
		P_r^{(n_{pq},0)}(\xi_3) \left( \frac{1-2~\frac{1+\xi_2}{1-\xi_3}-1}{2} \right)^p \left( \frac{1-\xi_3}{2} \right)^{p+q}.
	\end{align*}
	Expanding $P_p^{(0,0)}(x-1)P_q^{(2p+1,0)}(y-1) = \sum_{k=0}^p\sum_{l=0}^q c_{kl} x^k y^l$ leads to
 	\begin{align*}
 		\psi_{p,q,r}(\xi_1,\xi_2,\xi_3)
 		&= \sum_{k=0}^p\sum_{l=0}^q c_{kl}~2^{k+l} \left(\frac{1+\xi_1}{\xi_2+\xi_3}\right)^k \left(\frac{1+\xi_2}{1-\xi_3}\right)^l 
 		P_r^{(n_{pq},0)}(\xi_3) \left(1-\frac{1+\xi_2}{1-\xi_3}\right)^p \left(\frac{1-\xi_3}{2}\right)^{p+q} \\
 		&= \sum_{k=0}^p\sum_{l=0}^q c_{kl}~\frac{2^{k+l}}{2^{p+q}} (1+\xi_1)^k (1+\xi_2)^l \frac{(1-\xi_3)^{q-l} (\xi_2-\xi_3)^p}{(\xi_2+\xi_3)^k} P_r^{(n_{pq},0)}(\xi_3).
 	\end{align*}
 	Since $P_r^{(n_{pq},0)}$ is a polynomial of degree $r$, we see by the last 2 terms in the sum above that $\psi_{p,q,r}\in \P_{p+q+r}(\T^3)$.

 	To see the orthogonality property, we transform to the cube $S^3$ and make use of (\ref{eq:jacobipoly-orthogonality}) three times
 	\begin{align*}
 		\int_{\T^3} \psi_{p,q,r}(\xi)\psi_{p',q',r'}(\xi) d\xi 
 		&= \int_{\S^3} \tilde{\psi}_{p,q,r}(\eta)\tilde{\psi}_{p',q',r'}(\eta) \left(\frac{1-\eta_2}{2}\right) \left(\frac{1-\eta_3}{2}\right)^2 d\eta \\
 		&= \int_{-1}^1 \int_{-1}^1 \int_{-1}^1 P_p^{(0,0)}(\eta_1) P_{p'}^{(0,0)}(\eta_1) \left( \frac{1-\eta_2}{2} \right)^{p+p'+1} P_q^{(2p+1,0)}(\eta_2) P_{q'}^{(2p'+1,0)}(\eta_2) \\
 		&\qquad\times \left( \frac{1-\eta_3}{2} \right)^{p+q+p'+q'+2} P_r^{(n_{pq},0)}(\eta_3) P_{r'}^{(n_{p'q'},0)}(\eta_3)~d\eta_1 d\eta_2 d\eta_3 \\
 		&= \frac{2}{2p+1}\delta_{pp'} 2^{-(2p+1)} \int_{-1}^1 \int_{-1}^1 (1-\eta_2)^{2p+1} P_q^{(2p+1,0)}(\eta_2) P_{q'}^{(2p+1,0)}(\eta_2) \\
 		&\qquad\times \left( \frac{1-\eta_3}{2} \right)^{2p+q+q'+2} P_r^{(n_{pq},0)}(\eta_3) P_{r'}^{(n_{pq'},0)}(\eta_3)~d\eta_2 d\eta_3 \\
 		&= \frac{2}{2p+1}\delta_{pp'} \frac{2}{2p+2q+2}\delta_{qq'} 2^{-n_{pq}} \int_{-1}^1 (1-\eta_3)^{n_{pq}} P_r^{(n_{pq},0)}(\eta_3) P_{r'}^{(n_{pq},0)}(\eta_3)~d\eta_3 \\
 		&= \frac{2}{2p+1}\delta_{pp'} \frac{2}{2p+2q+2}\delta_{qq'} \frac{2}{2p+2q+2r+3}\delta_{rr'}.
 	\end{align*}
\end{proof}

\begin{proof}[Proof of Corollary~\ref{cor:connection-tilde-u-pqr-and-tilde-u-pqr-prime}]
	To prove this corollary we want to make use of Lemma \ref{lemma:connection-U-and-Uprime}. Therefore, 
	we have to clarify that the conditions in the lemma are satisfied. 
  We proceed in two steps. First, we require $u\in C^\infty(\overline{\T^3})$ and 
  show the statement in this case and then we argue by density to achieve results in $H^1(\T^3)$. 

	Step 1: By assuming that $u\in C^\infty(\overline{\T^3})$ we get $\tilde{u}\in C^1(\S^3)$.
  Hence, for fixed $p$ and $q$, if we recall the definition of $U_{p,q}$, we see that the map     
	$\eta_3 \mapsto U_{p,q}(\eta_3)$ is smooth on $[-1,1]$.
  Considering the definition of $\widetilde{U}_{p,q}$
  \begin{align*}
  	\widetilde{U}_{p,q}(\eta_3) = \frac{U_{p,q}(\eta_3)}{(1-\eta_3)^{p+q}},
  \end{align*}
  we see that $\widetilde{U}_{p,q}\in C^1([-1,1))$ and that $\widetilde{U}_{p,q}$ has at most one pole of     
	maximal order $p+q$ at the point $\eta_3 = 1$. 
  In view of these preliminary considerations we conclude that the following limits exist and that the conditions in     
	Lemma~\ref{lemma:connection-U-and-Uprime} are satisfied:
	\begin{align*}
    \lim_{\eta_3 \rightarrow 1} (1-\eta_3)^{2p+2q+3}\widetilde{U}_{p,q}(\eta_3)         
		= \lim_{\eta_3 \rightarrow 1} (1-\eta_3)^{p+q+3}U_{p,q}(\eta_3) = 0.
  \end{align*}
	and
  \begin{align*}
  	\lim_{\eta_3 \rightarrow -1} (1+\eta_3) \widetilde{U}_{p,q}(\eta_3) = 0.
  \end{align*}
  Now the statement follows directly from Lemma~\ref{lemma:connection-U-and-Uprime} when looking at the definition of     
	$\tilde{u}_{p,q,r}$ and $\tilde{u}_{p,q,r}^\prime$ and consequently replacing 
	$U$ with $\widetilde{U}_{p,q}$ and $\alpha$ with $2p+2q+1$.  

	Step 2: Let $u\in H^1(\T^3)$. Since $C^\infty(\overline{\T^3})$ is dense in $H^1(\T^3)$, there exists 
	a sequence $(u_n)_{n\in\BbbN} \subset C^\infty(\overline{\T^3})$ such that $u_n \rightarrow u$ in $H^1(\T^3)$ for $n \rightarrow \infty$.
	Because we have already proved that $u_n,\,\, n\in\BbbN$ satisfies our statement, ensuring that the sequences of coefficients $\tilde{u}_{p,q,r}^n$ and 
	$\tilde{u}^{\prime^n}_{p,q,r}$ corresponding to $u_n$ converge for fixed $p$,$q$ and $r$ will conclude the proof: 

	We have
	\begin{align*}
		 \tilde{u}_{p,q,r}^n = 2^{p+q+2}u_{p,q,r}^n = 2^{p+q+2} \int_{\T^3} u_n(\xi_1,\xi_2,\xi_3) \psi_{p,q,r}(\xi_1,\xi_2,\xi_3) d\xi_1 d\xi_2 d\xi_3.
	\end{align*} 
	Since $(\psi_{p,q,r})_{p,q,r\in\BbbN_0}$ forms an orthogonal basis
	for $L^2(\T^3)$ and since $H^1(\T^3) \subset L^2(\T^3)$, the maps $F: u \mapsto \tilde{u}_{p,q,r}$ are continuous 
  linear functionals on $H^1(\T^3)$ and thus $\lim_{n \rightarrow \infty} F(u_n) = F(u)$. \newline 
	In case of $\tilde{u}^\prime_{p,q,r}$ we study the functionals $\widetilde{F}: u \mapsto \tilde{u}_{p,q,r}^\prime$ that map
	$C^\infty(\overline{\T^3})$ into $\BbbR$. Since $\widetilde{F}$ is a linear functional that is continuous with respect 
  to the $H^1(\T^3)$-norm, we see by density of $C^\infty(\overline{\T^3})$ in $H^1(\T^3)$ that it is indeed a well-defined 
  continuous linear functional on $H^1(\T^3)$ 
  and thus again $\lim_{n \rightarrow \infty} \widetilde{F}(u_n) = \widetilde{F}(u)$. 
\end{proof}
%

\section{Extended versions of the tables of Section~\ref{sec:numerics}}

\begin{table}[ht]
        \begin{center}
                \begin{tabular}{|c|c|c|}
                        \hline
                        $N$ & $\sup_{u\in\P_{2N}} \frac{|(\Pi_N u)(1)|^2}{\|u\|_{L^2}\|u\|_{H^1}}$ 
                        & $\sup_{u\in\P_{2N}} \frac{|(\Pi_N u)(1)|^2}{\|u\|^2_{H^1}}$ \\
                        \hline \hline
                        $1$ & 1.18184916854199 & 0.87500000000000 \\
                        $2$ & 1.82982112979637 & 1.14361283167718 \\
                        $3$ & 2.15270769390416 & 1.15072048261852 \\
                        $4$ & 2.34106259718609 & 1.13538600864567 \\
                        $5$ & 2.45948991256407 & 1.11992788388317 \\
                        $10$ & 2.72197668882430 & 1.08267283507986 \\
                        $15$ & 2.82210388053635 & 1.06853381605478 \\
                        $20$ & 2.87406416223951 & 1.06111064764886 \\
                        $25$ & 2.90512455645115 & 1.05653834380496 \\
                        $30$ & 2.92540310256400 & 1.05343954290018 \\
                        $35$ & 2.93948150346734 & 1.05120092371494 \\
                        $40$ & 2.94971129296239 & 1.04950802550192 \\
                        $45$ & 2.95741139670201 & 1.04818299976127 \\
                        $50$ & 2.96337279789140 & 1.04711769879646 \\
                        $55$ & 2.96809547801154 & 1.04624257963827 \\
                        $60$ & 2.97190920471158 & 1.04551089085115 \\
                        $65$ & 2.97503926131730 & 1.04489004319026 \\
                        $70$ & 2.97764417211453 & 1.04435662477469 \\
                        $75$ & 2.97983833781387 & 1.04389338318215 \\
                        $80$ & 2.98170613800160 & 1.04348732490767 \\
                        $85$ & 2.98331100621813 & 1.04312847760604 \\
                        $90$ & 2.98470143645119 & 1.04280906028005 \\
                        $95$ & 2.98591505801301 & 1.04252291272595 \\
                        $100$ & 2.98698146107879 & 1.04226509441461 \\
                       $105$ & 2.98792419448666 & 1.04203159687688 \\
                        $110$ & 2.98876220322068 & 1.04181913380901 \\
                       $115$ & 2.98951087919110 & 1.04162498545355 \\
                        $120$ & 2.99018284042270 & 1.04144688155757 \\
                        \hline
                \end{tabular}
                \caption{\label{tab:1D-constants-extended} Computed constants $C$ for 1D maximization problems}
\end{center}
\end{table}
\begin{table}[ht]
{\small 
 \vspace*{-10mm}
	\begin{center}
		\begin{tabular}{|c|c|c|c|}	
			\hline
			$N$ & $\sup_{u\in\P_{2N}} \frac{\|\Pi_N u\|^2_{L^2(\Gamma)}}{\|u\|_{L^2(\T^2)}\|u\|_{H^1(\T^2)}}$ 
			& $\sup_{u\in\P_{2N}} \frac{\|\Pi_N u\|^2_{L^2(\Gamma)}}{\|u\|^2_{H^1(\T^2)}}$ 
			& $\sup_{u\in\P_{2N}} \frac{\|\Pi_N u\|^2_{H^1(\T^2)}}{\|u\|^2_{H^1(\T^2)}(N+1)}$\\
			\hline \hline
			$1$ & 1.841709179979923 & 1.471718130438879 & 0.630765682656827 \\
			$2$ & 2.482008261007026 & 1.705122181047698 & 0.534605787746493 \\
			$3$ & 2.840187661660685 & 1.715716151553367 & 0.502563892227422 \\
			$4$ & 3.069499343400879 & 1.698817794545474 & 0.469837538241653 \\
			$5$ & 3.221409644442824 & 1.681456018543639 & 0.451490548981791 \\
			$6$ & 3.328227344860761 & 1.668376974400849 & 0.442841777476419 \\
			$7$ & 3.407909278234681 & 1.658518924126468 & 0.439947515744027 \\
			$8$ & 3.470131098161663 & 1.650883702558015 & 0.438357597812212 \\
			$9$ & 3.520381327094282 & 1.644802661753009 & 0.437324308925344 \\
			$10$ & 3.561974205023130 & 1.639847070385123 & 0.436649057543612 \\
			$11$ & 3.597047215946491 & 1.635732061317178 & 0.436162962404347 \\
			$12$ & 3.627055403122174 & 1.632261308353419 & 0.435795222777319 \\
			$13$ & 3.653032616328430 & 1.629295122853766 & 0.435497902586582 \\
			$14$ & 3.675739731999510 & 1.626731407637615 & 0.435247417526015 \\
			$15$ & 3.695751965776597 & 1.624493814472113 & 0.435028483384906 \\
			$16$ & 3.713514379154265 & 1.622524114095057 & 0.434833152302013 \\
			$17$ & 3.729377544445405 & 1.620777127516557 & 0.434656212926501 \\
			$18$ & 3.743622134346961 & 1.619217268307188 & 0.434494915639700 \\
			$19$ & 3.756475749181842 & 1.617816128453177 & 0.434347360750190 \\
			$20$ & 3.768125200720072 & 1.616550757911536 & 0.434212531866692 \\
			$21$ & 3.778725310656065 & 1.615402415968316 & \\
			$22$ & 3.788405606490951 & 1.614355650089980 & \\
			$23$ & 3.797275282737163 & 1.613397606293759 & \\
			$24$ & 3.805427098445950 & 1.612517505912479 & \\
			$25$ & 3.812940340675182 & 1.611706243752921 & 0.433709285223336 \\
			$26$ & 3.819883212382239 & 1.610956076034725 &\\
			$27$ & 3.826314685175567 & 1.610260375564773 &\\
			$28$ & 3.832286023912474 & 1.609613437845351 &\\
			$29$ & 3.837841988875210 & 1.609010326178230 &\\
			$30$ & 3.843021842931179 & 1.608446746919096 & 0.433462666731238 \\
			$31$ & 3.847860156122867 & 1.607918948256749 &\\
			$32$ & 3.852387490412042 & 1.607423637503739 &\\
			$33$ & 3.856630952540156 & 1.606957913069210 &\\
			$34$ & 3.860614671188471 & 1.606519208163681 &\\
			$35$ & 3.864360185593495 & 1.606105243942980 & 0.433441925213746 \\
			$36$ & 3.867886785255372 & 1.605713990297966 &\\
			$37$ & 3.871211788573052 & 1.605343632872021 &\\
			$38$ & 3.874350789295843 & 1.604992545187163 &\\
			$39$ & 3.877317859826744 & 1.604659264974558 &\\
			$40$ & 3.880125733031547 & 1.604342473993514 & 0.433618345964902 \\
			$41$ & 3.882785952906251 & 1.604040980750689 &\\
			$42$ & 3.885309010715330 & 1.603753705645651 &\\
			$43$ & 3.887704458222425 & 1.603479668157049 &\\
			$44$ & 3.889981011011230 & 1.603217975745874 &\\
			$45$ & 3.892146634658522 & 1.602967814218030 & 0.433964716705375 \\
			$46$ & 3.894208624113123 & 1.602728439323389 &\\
			$47$ & 3.896173670029407 & 1.602499169413209 &\\
			$48$ & 3.898047920431772 & 1.602279379001790 &\\
			$49$ & 3.899837032314585 & 1.602068493106013 &\\
			$50$ & 3.901546220032895 & 1.601865982255488 & 0.434455030379973 \\
			$51$ & 3.903180295826488 & 1.601671358082506 &\\
			$52$ & 3.904743708152147 & 1.601484169413596 &\\
			$53$ & 3.906240573788223 & 1.601303998798579 &\\
			$54$ & 3.907674708463756 & 1.601130459422183 &\\
			$55$ & 3.909049652499364 & 1.600963192345851 & 0.435064216731618 \\
			\hline
		\end{tabular}
		\caption{\label{tab:2D-constants-extended} Extended version of Table~\ref{tab:2D-constants}}
	\end{center}
}
\end{table}
\clearpage
\fi 
\section*{References}
\bibliographystyle{elsart-num-sort}
\bibliography{nummech}

\end{document}